\documentclass[12pt,twoside]{article}
\usepackage{amsthm,amsmath,amssymb,natbib,a4wide,color,graphicx,dsfont,xr}
\numberwithin{equation}{section}
\usepackage[T1]{fontenc}
\usepackage[utf8]{inputenc}
\usepackage[shortlabels]{enumitem}
\RequirePackage[colorlinks,citecolor=blue,urlcolor=blue]{hyperref}

\theoremstyle{plain} \newtheorem{theorem}{Theorem}[section]
\theoremstyle{plain} \newtheorem{proposition}[theorem]{Proposition}
\theoremstyle{plain} \newtheorem{lemma}[theorem]{Lemma}
\theoremstyle{plain} \newtheorem{corollary}[theorem]{Corollary}
\theoremstyle{definition} \newtheorem{definition}[theorem]{Definition}
\theoremstyle{definition} 
\theoremstyle{remark} \newtheorem{remark}[theorem]{Remark}
\theoremstyle{remark} \newtheorem{example}[theorem]{Example}

\newcommand{\rmd}{\mathrm{d}}
\newcommand{\rme}{\mathrm{e}}
\newcommand{\cB}{\mathcal{B}}

\newcommand{\cX}{\mathcal{X}}
\newcommand{\cT}{\mathcal{T}}
\newcommand{\cN}{\mathcal{N}}
\newcommand{\bbM}{\mathbb{M}}
\newcommand{\bbE}{\mathbb{E}}

\newcommand{\bbP}{\mathbb{P}}
\newcommand{\bbD}{\mathbb{D}}

\begin{document}

\title{Hidden regular variation for point processes and the single/multiple large point heuristic}

\author{Cl\'{e}ment Dombry\footnote{Université de Franche-Comt\'{e}, Laboratoire de Mathématiques de Besançon, 16 route de Gray, 25030 Besançon cedex, France.  \texttt{Email: clement.dombry(at)univ-fcomte.fr }} \and Charles Tillier\footnote{ Université Paris-Saclay, UVSQ, CNRS, Laboratoire de Mathématiques de Versailles, 45 avenue des états-unis, 78000 Versailles, France.  \texttt{Email: charles.tillier@uvsq.fr}} \and Olivier Wintenberger\footnote{Sorbonne Université, Laboratoire de probabilités et modèles aléatoires, 4 place Jussieu, 75005 Paris, France.  \texttt{Email: olivier.wintenberger@upmc.fr}}  }

\maketitle

\begin{abstract}
We consider regular variation for marked point processes with independent heavy-tailed marks and prove a single large point heuristic: the limit measure is concentrated on the cone of point measures with one single point. We then investigate successive hidden regular variation removing the cone of point measures with at most $k$ points, $k\geq 1$, and prove a multiple large point phenomenon: the limit measure is concentrated on the cone of point measures with $k+1$ points. We show how these results imply  hidden regular variation in Skorokhod space of the associated risk process, in connection with the single/multiple large point heuristic from \cite{RBZ19}. Finally, we provide an application to risk theory in a reinsurance model where the $k$ largest claims are covered and we study the asymptotic behavior of the residual risk.
\end{abstract}


\section{Introduction}

Regular variation is a fundamental concept in the analysis of rare event probabilities for heavy-tailed models that was widely popularized by  \citet{Resnick2007, Resnick2008} and finds natural applications in risk theory \citep{Embr1997, Assmussen2010, Mikosch, HL11}.

Regular variation was first considered on the finite-dimensional space $\mathbb{R}^d$ and formulated in terms of vague convergence on the compactified space $E=[-\infty,+\infty]^d$. In such a context, a random element $X \in \mathbb{R}^d$ is said to be regularly varying  if there exists a positive sequence $a_n\to\infty$ such that 
\[
n\mathbb{P}(a_n^{-1}X\in \cdot)\stackrel{v}\longrightarrow \mu(\cdot),\quad \mbox{as $n\to\infty$},
\]
where $\stackrel{v}\longrightarrow$ stands for vague convergence on $E\setminus\{0\}$. If the limit measure $\mu$  concentrates only on the axes 
$F=\cup_{i=1}^d \big( \{0\}^{i-1}\times {\mathbb R}\times \{0\}^{d-i}\big)$,
we say that $X$ has  asymptotically independent components.
One can then wonder if some dependence among the components appears in a different regime. The concept of hidden regular variation was introduced for this purpose by \cite{Resnick2002} and formulated in terms of vague convergence on the space $E\setminus F$  of the form $n\mathbb{P}(\tilde a_n^{-1}X\in \cdot)\stackrel{v}\longrightarrow \tilde\mu(\cdot)$. We refer to \cite{heffernan_resnick_2005},  \cite{Maulik2004}, \cite{Mitra2011} for further developments on finite-dimensional hidden regular variation.

To go beyond finite-dimensional spaces,  $M_0$-convergence on a general metric space was introduced by \cite{LR06}; see also \cite{hult2010large} for the seminal treatment of $M_0$-convergence on spaces of point measures. This notion avoids the compactification procedure and replaces the compactly supported continuous test functions from vague convergence by  bounded continuous functions with support bounded away from the origin. This results in an elegant theory that provides a convenient framework for regular variation in infinite dimensional spaces. Infinite dimensional regular variation theory includes the analysis of heavy-tailed stochastic processes \citep{HL05,HL07,HLMS05} or times series \citep{basrak:segers:2009,DHS18}; see also the seminal contributions on regular variation on infinite dimensional spaces in \cite{araujo1980central}, \cite{gine1990max} and \cite{de2001convergence}.

In order to consider hidden regular variation in a function space, the theory of $M$-convergence was further extended in \cite{LRR14} removing cones larger than the origin. The main example provided there is the infinite dimensional space $\mathbb{R}_+^\infty$  with application to hidden regular variation of L\'evy processes in Skorokhod space $\mathbb{D}([0,1],\mathbb{R})$. \cite{RBZ19} provide further insight into the  hidden regular variation for regularly varying L\'evy processes and random walks in connection with sample path large deviations. They propose a single/multiple large jump heuristic where the limit measure in (hidden) regular variation is supported by the cone of functions with one single (multiple) large jumps; see also \cite{pinelis1981problem} for a seminal contribution on path-wise large deviations for heavy-tailed stochastic processes.

In this paper, we provide the first detailed analysis of (hidden) regular variation properties for point processes. Point processes are an important tool in applied probability and stochastic modelling and are widely used in risk theory. The L\'evy processes mentioned above can be seen as  functionals of Poisson point processes, so that one can expect to deduce the regular variation of the former from those of the latter \textit{via} a continuous mapping theorem. For these reasons, we believe regular variation at the level of point process is an important and fundamental conceptual tool. A first result stating the regular variation of  Poisson point processes with regularly varying intensity measure appears in \cite{DHS18} with a single large point heuristic. We propose here an analysis beyond the Poisson case and consider independently marked point processes (with regularly varying mark distribution) and successive hidden regular variation with different orders. We derive a general criterion for (hidden) regular variation in terms of pointwise Laplace functional and apply it to several models of increasing complexity: marked Poisson point processes, independently marked point processes and triangular arrays of independently marked point processes. For all these models, a similar structure for successive hidden regular variation is discovered and a single/multiple jump heuristic is proved.

The structure of the paper is the following. In Section~\ref{sec:background}, we set the necessary background on measure spaces and regular variation following the lines of \cite{LRR14}. We also settle the point process framework and provide in  Theorem~\ref{theo:cv-N-Nk} a characterization of hidden regular variation of point processes in terms of convergence of their Laplace functionals. Section~\ref{sec:RVPP} states the main results of this paper on the successive hidden regular variation for independently marked point processes. In Section~\ref{sec:app}, we use the continuous mapping theorem to derive the hidden regular variation properties in Skorokhod space of the risk processes associated to a marked point process. An application to a reinsurance problem is also discussed with the asymptotic analysis of the residual risk after reinsurance of the largest claims. All the proofs are gathered in Sections~\ref{sec:proof2},~\ref{sec:proof3} and~\ref{sec:proof1}.

\medskip
\textbf{Notation and shortcuts:} In the following, $[a]$ denotes the integer part of $a$, $(a)_+=\max(a,0)$ the positive part of $a$ and $a \wedge b$ the minimum between $a$ and $b$, for $a,b \in \mathbb{R}$.  The indicator function of the set $A$ is denoted by $\mathds{1}_{A}$. For a set $A$, $\mathrm{int}A$, $\mathrm{cl}A$ and $\partial A $ are respectively the interior, closure and boundary of $A$. The equivalence of two real sequences $u_n\sim v_n$ means that $u_n/v_n\to 1$ as $n\to\infty$; the notation $u_n=o(v_n)$ and $u_n=O(v_n)$ mean respectively that  $u_n/v_n\to 0$ as $n \to \infty$ and that $u_n/v_n$ remains bounded  as $n\to\infty$.  We denote by $\varepsilon_x(\cdot)$  the Dirac measure at $x$ and by $\otimes$ the tensor product of measures.

\section{Background on measure spaces, regular variation and point processes}\label{sec:background}

We set up in this section  the mathematical background  necessary for this paper.  
We start by defining the framework and notation for measure spaces and then turn to  regular variation and hidden regular variation, following the lines of \cite{LRR14}, Sections 2 and 3. We   finally present some background on point processes.

\subsection{Background on measure spaces}  
We denote by $(E, d_E)$  a complete separable metric space endowed with its Borel  $\sigma$-algebra $\mathcal{B}(E)$ generated by the open balls  $B_{x,r}^E= \{ x' \in E : d_E(x,x')<r \}$, $x\in E$, $r>0$.  When there is no confusion, we omit the superscript $E$ and  write simply  $B_{x,r}$. The $r$-neighborhood of a subset $A\subset E$ is the open set $A^r$ of points that are at distance less than $r$ from $A$, that is $A^r=\cup_{x\in A} B_{x,r}$. We say that a subset $B$ is bounded away from $A$ if   $B\cap A^r=\emptyset$ for some $r>0$.

The space of bounded continuous real-valued  functions on $E$ is denoted by $\mathcal{C}_b(E)$. The set of finite Borel measures on  $E$ is denoted by $\mathbb{M}_b (E)$. A sequence of measures $(\mu_n)_{n\geq 1}$ is said to converge weakly to $\mu$ in  $\mathbb{M}_b (E)$, denoted $\mu_n\stackrel{\mathbb{M}_b (E)}{\longrightarrow} \mu$,  if $\int f\rmd \mu_n\to \int f\rmd\mu$ for all $f\in \mathcal{C}_b(E)$. The Prohorov distance on $\mathbb{M}_b(E)$ defined by
	\begin{align*}
	&d_{\mathbb{M}_b(E)}(\mu_1,\mu_2)=\inf_{\varepsilon>0}\left\{\mu_1(A)\leq \mu_2(A^\varepsilon)+\varepsilon \ \mbox{for all } A\in\mathcal{B}(E) \right\}
	\end{align*}
metrizes weak convergence and, equipped with this metric, $\mathbb{M}_b(E)$ is a complete separable metric space; see \cite{Kallenberg}.
	
For $F\subset E$ a closed subset, we denote by $\mathbb{M}(E\setminus F)$ the set of Borel measures $\mu$ on $E\setminus F$ that assign finite mass on  sets bounded away from $F$, that is, such that 
$\mu(E\setminus F^r)<\infty$ for all $r>0$. In the following proposition, corresponding to  \citet[Theorem 2.2]{LRR14}, we provide equivalent characterizations of convergence in $\mathbb{M}(E\setminus F)$.

\begin{proposition}[Convergence in $\mathbb{M}(E\setminus F)$]\label{ConvMEF}
Let $\mu_n,\mu \in \mathbb{M}(E\setminus F)$. The convergence $\mu_n \to \mu$ in $\mathbb{M}(E\setminus F)$, denoted $\mu_n\stackrel{\mathbb{M}(E\setminus F)}\longrightarrow \mu$, is defined  by the following equivalent properties:
\begin{enumerate}[i)]
\item  for all $f\in\mathcal{C}_b(E)$ with support bounded away from $F$,
$$\int_{E}f\rmd \mu_n\to\int_{E}f\rmd \mu\quad \mbox{as $n\to\infty$};$$
\item for all $A\in\cB(E)$ bounded away from $F$,
$$\mu(\mathrm{int}A)\leq \mathop{\mathrm{liminf}}_{n\to\infty} \mu_n(A)\leq \mathop{\mathrm{limsup}}_{n\to\infty} \mu_n(A)\leq \mu(\mathrm{cl}A);$$
\item there exists a  sequence $r_i\downarrow 0$ such that $\mu_n^{r_i}\longrightarrow \mu^{r_i}$ in $\mathbb{M}_b(E\setminus F^{r_i})$ as $n\to\infty$, for each $i\geq 1$, where $\mu_n^{r_i}$ (resp. $\mu^{r_i}$) denotes the restriction of $\mu_n$ (resp. $\mu$) to $E\setminus F^{r_i}$.
\end{enumerate}
\end{proposition}

This notion of convergence is metrized by the distance
\begin{align}\label{distance_rho}
\rho(\mu_1,\mu_2)=\int_0^\infty \left\{\rho_r\left(\mu_1^{r},\mu_2^{r}\right)\wedge 1\right\}e^{-r}\rmd r
\end{align}
where $\mu^{r}$ denotes the restriction of $\mu$ to  $E\setminus F^r$ and $\rho_r$ the Prohorov metric on $\mathbb{M}_b(E\setminus F^r)$. Furthermore, $\mathbb{M}(E\setminus F)$ endowed with the distance $\rho$ is a complete separable metric space \citep[Theorem 2.3]{LRR14}.

\subsection{Background on regular variation}\label{sec:BRV}

Regular variation intrinsically involves the notion of scaling and cones. A scaling on a complete and separable metric space is a multiplication by positive reals, that is a continuous mapping  $(0, \infty) \times E\to E$ satisfying
\begin{align*}
&1x=x,\\
&u_1 (u_2 x)=(u_1 u_2)x,\quad \mbox{for all $u_1,u_2>0$}.
\end{align*}
Equivalently, a scaling is a continuous group action of  $(0,\infty)$ on $E$.

A cone is a Borel set $F \subset E$ that is stable under the group action, that is $x \in F$ implies $ux\in F$ for all $u>0$. In the following, we assume that  $F$ is a closed cone such that
\[
d(x,F)<d(ux,F),\quad \mbox{for all } u>1,\, x\in E\setminus F, 
\]
where $d(x,F)=\inf\{d(x,y):y\in F\}$ denotes the distance to the cone $F$.

\begin{definition}[Regular variation]\label{def:RV}\ \\ 
\vspace*{-0.5cm}
\begin{itemize}
\item	A measure $\nu\in\bbM(E\setminus F)$ is said to be regularly varying   if there exists a positive sequence $a_n\to\infty$ and a non-null measure $\mu \in \bbM(E\setminus F)$ such that 
	$$n\nu(a_n \cdot) \longrightarrow\mu (\cdot)\quad \mbox{in }\bbM(E\setminus F).$$ 
	When such a convergence holds, we write $\nu\in \mathrm{RV}(E\setminus F, \{a_n\}, \mu)$.
	\item An $E$-valued random element $X$ defined on a probability space $(\Omega, \mathcal{A},\bbP)$ is said to be regularly varying on $E\setminus F$ if there exists a positive sequence $a_n\to\infty$ and a non-null measure $\mu \in \bbM(E\setminus F)$ such that
	$$n \bbP(a_n^{-1} X \in \cdot )  \longrightarrow\mu (\cdot) \quad \mbox{in }\bbM(E\setminus F).$$  
	When such a convergence holds, we write $X \in \mathrm{RV}(E\setminus F, \{a_n\}, \mu)$.
	\end{itemize}
\end{definition}
Here, by abuse of notation, $n\bbP(a_n^{-1}X\in\cdot)$ is seen  as the restriction to $E\setminus F$ of the rescaled distribution of $X$; similarly
we say that a measure $\nu\in\bbM_b(E)$ is regularly varying on $E\setminus F$ if its restriction to $E\setminus F$ is regularly varying.

There are many equivalent formulations of regular variation, another important one being the convergence of 
\[
b_n\nu(n\cdot)\longrightarrow\mu (\cdot) \quad \mbox{or}\quad b_n\bbP(n^{-1}X\in\cdot)\longrightarrow\mu (\cdot) \quad \mbox{in $\bbM(E\setminus F)$}
\]
for some other positive sequence $b_n$, related to $a_n$ by $a_{[b_n]}\sim b_{[a_n]}\sim n$.  Also the convergence along integers $n\to\infty$ can be reinforced into convergence along a real variable $x\to\infty$. We refer to \cite{LRR14}, Definition 3.2 and Theorem 3.1 for exhaustive statements. 

An important consequence of regular variation is the existence of a regular variation index $\alpha>0$ such that $\mu$ is homogeneous of order $-\alpha<0$, that is 
$$\mu(u\,\cdot)=u^{-\alpha}\mu(\cdot)\quad \mbox{for all $u>0$}.$$
Then,  $(a_n)$ and $(b_n)$ are regularly varying sequences at infinity with index  $1/\alpha$ and $\alpha$ respectively, that is 
$$\lim_{n\to\infty} \frac{a_{[nu]}}{a_n}= u^{1/\alpha}  \quad \mbox{and}\quad \lim_{n\to\infty} \frac{b_{[nu]}}{b_n}= u^{\alpha},\quad \mbox{for all $u>0$}.
$$

\medskip
Importantly, regular variation not only gives a rate of convergence for rare events probabilities of the type $\bbP(X\in xA)$ for large $x>0$ but also provides the typical behavior of $X$ given the rare event $X\in xA$. This is formulated in terms of a conditional limit theorem, as in the following proposition.
\begin{proposition}\label{prop:cond_limit_thm}
Let $X\in\mathrm{RV}(E\setminus F,\{a_n\},\mu)$ and $A\in\mathcal{B}(E)$ be bounded away from $F$  such that $\mu(A)>0$ and $\mu(\partial A)=0$. Then, as $x\to\infty$,  
\[
\mathbb{P}(x^{-1}X\in \cdot \mid X\in xA)\longrightarrow \mu_A (\cdot) :=\frac{\mu(A\cap\cdot)}{\mu(A)} \quad \mbox{in $\bbM_b(E)$}.
\]
\end{proposition}

The next proposition will be crucial in our proof of Theorems~\ref{theo:RV-D} and~\ref{theo:RV-D2}. It states  a regular variation criterion similar to the well known second converging together  theorem for weak convergence (see \cite{B68} Theorem 4.2  or \cite{Resnick2007} Theorem 3.5) and is especially useful when combined with truncation arguments.
\begin{proposition}\label{prop:HRV-criterion}
Let $E$ be a complete separable metric space and consider $E$-valued random variables $X$ and $X_{n,m}$, $n,m\geq 1$. Let $F\subset E$ be a closed cone. Assume that there is a positive sequence $(a_n)$ and $k\ge 1$ such that:
\begin{enumerate}[i)]
\item for each $m\geq 1$, $n^{k}\bbP(a_n^{-1}X_{n,m}\in\cdot )\longrightarrow \mu_m (\cdot)$ in $\bbM(E\setminus F)$ as $n\to\infty$;
\item $\mu_m\longrightarrow \mu$ in $\bbM(E\setminus F)$ as $m\to\infty$;
\item for all $\varepsilon>0$, $r>0$
\begin{align*}
\lim_{m\to\infty}\limsup_{n\to\infty} n^{k} \bbP(d(a_n^{-1}X_{n,m},a_n^{-1}X)>\varepsilon,d(a_n^{-1}X_{n,m},F)>r)&=0\\
\lim_{m\to\infty}\limsup_{n\to\infty} n^{k} \bbP(d(a_n^{-1}X_{n,m},a_n^{-1}X)>\varepsilon,d(a_n^{-1}X,F)>r)&=0.
\end{align*}
\end{enumerate}
Then, $n^{k}\bbP(a_n^{-1}X_{n}\in\cdot )\longrightarrow \mu (\cdot)$ in $\bbM(E\setminus F)$ as $n\to\infty$.
\end{proposition}

Remark that Proposition \ref{prop:HRV-criterion} holds replacing the normalizing sequence $n^k$ by any increasing function of $n$ say $(l_n)$. The proof of Proposition \ref{prop:HRV-criterion} in Section 7.1 generalises effortless to this more general case.

We finally provide  some intuition on successive hidden regular variation. Most often, regular variation is used when $F=F_0=\{0\}$ is reduced to a single point $0$, called the origin of $E$ and satisfying $u0=0$ for all $u>0$. Then, for an $E$-valued random variable $X$, $X/n$ converges in distribution to $0$ as $n\to\infty$. This limit theorem is quite uninformative since it simply states that $\bbP(X/n\in A)\to 0$ for all Borel sets $A$ bounded away from $0$. It is hence sensible to rescale these  probabilities and consider the convergence  $b_n\bbP(X/n\in \cdot)\to\mu_0(\cdot)$ in  $\bbM(E\setminus F_0)$. This implies roughly $\bbP(X/n\in  A)\sim \mu_0(A)/b_n$ and is much more informative, provided $\mu_0(A)>0$. When $\mu_0(A)=0$,  it is natural to look for a higher order scaling $b_n^{(1)}$ such that $b_n^{(1)}/b_n\to\infty$. The support of the homogeneous measure $\mu_0$ is a closed cone $F_1$ such that $\mu_0(A)=0$ for all $A$ bounded away from $F_1$. This leads us to consider the convergence  $b_n^{(1)}\bbP(X/n\in \cdot)\to \mu_1(\cdot)$ in  $\bbM(E\setminus F_1)$. This procedure can be repeated and we may obtain successive hidden regular variation of the form
\begin{equation}\label{eq:succ-RV}
b_n^{(k)}\bbP(n^{-1}X\in \cdot)\to \mu_k(\cdot) \ \ \mbox{in  $\bbM(E\setminus F_k)$},\quad k\geq 0,
\end{equation}
where $(F_{k})_{k\geq 1}$ are increasing cones, $\mu_k\in \bbM(E\setminus F_k)$ are measures with disjoint supports and  $b_n^{(k)}>0$ are rate functions such that $\lim_{n\to\infty} b_n^{(k)}/b_n^{(k-1)}= \infty$.

These successive regular variation results can also be formulated in terms of large deviations as in \cite{RBZ19}, Theorem 3.2. For a Borel set $A\subset E\setminus\{0\}$, define 
\[
\mathcal{K}(A)=\max\{k\geq 0: A\cap F_k=\emptyset \}
\quad\mbox{and}\quad
\mathcal{I}(A)=\mu_{\mathcal{K}(A)}( A).
\]
The successive regular variation from Equation~\eqref{eq:succ-RV} is equivalent to the large deviations 
\[
\mathcal{I}(\mathrm{int} A)\leq \liminf_{n\to\infty} b_n^{(\mathcal{K}(A))}\mathbb{P}(n^{-1}X\in A)\leq \limsup_{n\to\infty} b_n^{(\mathcal{K}(A))}\mathbb{P}(n^{-1}X\in A)\leq  \mathcal{I}(\mathrm{cl} A),
\]
for all Borel set $A\subset E\setminus\{0\}$ such that $\mathcal{K}(A)$ is  finite with  $A$ bounded away from $F_{\mathcal{K}(A)}$. In the following, we use mostly the terminology of regular variation instead of large deviations.

\subsection{Background on point processes}
 We refer to  \cite{DVJ03, DVJ08} and \cite{Snyder91} for a complete review of point process theory. 
 
Given a complete separable metric space $E$ and  a closed subset $F\subset E$, we have seen that the measure space $\bbM(E\setminus F)$ endowed with the metric $\rho$ is a complete and separable metric space. The subset $\cN(E\setminus F)$ of $\mathbb{N}$-valued measures is the subset of  point measures of the form $\pi=\sum_{i\in I} \varepsilon_{x_i}$, where the $x_i$'s are in $E\setminus F$ and $I$ is a countable index set. The condition that $\pi$ is finite on subsets that are bounded away from $F$ implies that the family $(x_i)_{i\in I}$ is at most countable and any accumulation point must belong to $F$. The mapping $x\mapsto \varepsilon_x$ defines an isometric embedding  $E\setminus F \to\cN(E\setminus F)$.

\cite{Kallenberg} shows that $\cN(E\setminus F)$ is a closed subset of $\bbM(E\setminus F)$ and hence a complete separable metric space.  A point process in $E\setminus F$ is a random variable with values in $\cN(E\setminus F)$, that is a measurable application from some probability space $(\Omega,\mathcal{F},\bbP)$ into $\cN(E\setminus F)$ endowed with its $\sigma$-algebra. When $E$ is equipped with a scaling function, a natural scaling induced on $\cN(E\setminus F)$ is 
\[
u\pi=\sum_{i\in I} \varepsilon_{ux_i},\quad u>0,\ \pi=\sum_{i\in I} \varepsilon_{x_i}\in \cN(E\setminus F).
\]
 That is, the multiplication acts on each point of the point measure.
Using this structure, one can define regularly varying point processes. It has been shown in Theorem 3.3 in \cite{DHS18} that if $\Pi$ is a Poisson point process on $E\setminus \{0\}$ with regularly varying intensity measure $\nu\in\mathrm{RV}(E\setminus \{0\},\{a_n\},\mu)$, then $\Pi$ is regularly varying in $\cN(E\setminus\{0\})$ with sequence $(a_n)$ and limit measure $\mu^*$ defined as the image of $\mu$ under $x\mapsto \varepsilon_x$, that is
\[
\mu^*(B)=\int_{E}\mathds{1}_{\{\varepsilon_x\in B\}}\mu(\rmd x),\quad B\in\mathcal{B}(\cN(E\setminus\{0\})).
\]
In this paper, we extend this result in several ways: we obtain successive hidden regular variation not only for Poisson point processes with $F=\{0\}$, but for general independently marked point process with regularly varying mark distribution. 

Our results are based on  a criterion for regular variation in $\cN(E\setminus F)$ extending Theorem A.1 in \cite{DHS18}. For brevity, we note here $\cN=\cN(E\setminus F)$. For $k\geq 0$, we consider the closed cone $\cN_k\subset \cN$ of point measures with at most $k$ points that is $$\cN_k := \left\lbrace \pi= \sum_{i=1}^p \varepsilon_{ x_i} ; \ 0 \leq p \leq k;\  x_1,\ldots,x_p \in E \setminus F \right\rbrace.$$ When $k=0$, $\cN_0=\{0\}$ is reduced to the null measure. For $\mu^*$ a Borel measure on $\cN$, we denote by $\mathcal{B}_{\mu^*}$ the class of Borel sets $A\in\cB(E)$ that are bounded away from $F$ and such that $\mu^*(\pi(\partial A)>0)=0$. The following theorem provides a criterion for convergence in $\bbM(\cN\setminus\cN_k)$ in terms of finite-dimensional distributions and Laplace functional. 

\begin{theorem}[Convergence in $\bbM(\cN\setminus\cN_k)$]\label{theo:cv-N-Nk}\ \\
Let $k\geq 0$ and $\mu^*,\mu_1^*,\mu_2^*,\ldots \in \bbM(\cN\setminus\cN_k)$.  The following statements are equivalent, where convergences are meant as $n\to\infty$:
	\begin{enumerate}[(i)]
		\item \label{item:convmzero} $\mu_n^*\longrightarrow \mu^*$ in $\bbM(\cN\setminus\cN_k)$;
		\item \label{item:convfidi} for all $p\geq 1$, $A_1,\ldots,A_p\in \mathcal{B}_{\mu^*}$ and $(m_1,\ldots,m_p)\in \mathbb{N}^p$ such that $\sum_{i=1}^p m_i\geq k+1$,
		\[
		\mu_n^*\left(\pi(A_i)=m_i,\ 1\leq i\leq p\right) \to \mu^*\left(\pi(A_i)=m_i,\ 1\leq i\leq p\right);
		\]
		\item  \label{item:convlaplace} there exists a decreasing sequence $r_i\downarrow 0$ such that
		\begin{align*}
		\int_{\mathcal{N}}\rme^{-\pi(f)}\mathds{1}_{\{\pi(E\setminus F^{r_i})\geq k+1\}}\mu_n^*(\rmd \pi) 
		&\longrightarrow     \int_{\mathcal{N}} \rme^{-\pi(f)}\mathds{1}_{\{\pi(E\setminus F^{r_i})\geq k+1\}}\mu^*(\rmd \pi),
		\end{align*}
for all bounded Lipschitz functions $f:E\to [0,\infty)$ vanishing on $F^{r_i}$  and where $\pi(f)=\int_E f(x)\pi(\rmd x)$.
\end{enumerate}
\end{theorem}
As a consequence of a standard approximation argument,  when $(iii)$ holds for  bounded Lipschitz test functions $f:E\to [0,\infty)$ vanishing on $F^{r_i}$, it holds also for all $f\in\mathcal{C}_b(E)$ vanishing on $F^{r_i}$. Furthermore, when $k=0$, we retrieve exactly Theorem A.1 in \cite{DHS18} since $(iii)$ is then equivalent to 
\begin{align}\label{Laplacek0}
\int_{\mathcal{N}}\left(1-\rme^{-\pi(f)}\right)\mu_n^*(\rmd \pi) 
		\longrightarrow     \int_{\mathcal{N}}\left(1-\rme^{-\pi(f)}\right)\mu^*(\rmd \pi),
\end{align}
because the contribution of the event $\{\pi(E\setminus F^{r_i})=0\}$ in the integral vanishes since $f$ is supported by $E\setminus F^{r_i}$.

\section{Regular variation for marked point processes}\label{sec:RVPP}
The simplest example from risk theory we want to consider is the Poisson point process $\Pi$ on $E=[0,T]\times [0,\infty)$ with fixed $T>0$ and intensity $\lambda(\rmd t)\nu(\rmd x)$, where $\lambda$ is a finite measure on $[0,T]$ and $\nu$ a probability measure on $(0,\infty)$. Then $\Pi$ has finitely many points almost surely and can be represented as  
\[
\Pi=\sum_{i=1}^N \varepsilon_{(T_i,X_i)}.
\]
Each point $(T_i,X_i)$ represents a claim occuring at time $T_i$ with size $X_i>0$. The random variable $N$ defined as  $N:= \# \{ i \geq 1 : T_i \leq T \} $  denotes the total number of claims up to time $T$ which is here Poisson distributed with mean $\lambda([0,T])$. The arrival times $T_1\leq \cdots \leq T_N$ form  a Poisson point process with intensity $\lambda$ on $[0,T]$. The claim sizes $X_1,\ldots,X_N$ distributed as $\nu$ are independent of the claim number and arrival times. We consider regular variation of the Poisson point process $\Pi$ when the claim size distribution $\nu$ is regularly varying.

\medskip
For the purpose of generality, we consider the more abstract and general framework where $E=\cT\times \cX$ is the cartesian product of two complete and separable metric spaces $(\cT,d_\cT)$ and $(\cX,d_\cX)$. We think of $\cT$ as the time component and $\cX$ as the space component. We equip $E$ with the distance 
\[
d_E(z,z')=d_\cT(t,t')+d_\cX(x,x')
\]
for $z=(t,x),z'=(t',x')\in E=\cT\times\cX.$ The induced scaling on $E$ is defined by $u\cdot(t,x)=(t,ux)$ for $u>0$, $t\in \cT$, $x\in\cX$.  That is the scaling operates on the space component only.  We also assume that $\cX$  possesses an origin noted $0_\cX$, or simply $0$ when no confusion is possible. The subset $F=\cT\times \{0\}\subset E$ is a cone representing the time axis.  Note that the $r$-neighborhood of $F$ is simply $F^r=\cT\times B_{0,r}^\cX$.
\newline

We develop a regular variation theory in $\cN=\cN(E\setminus F)$ for independently marked point processes of the form
\begin{align}\label{MPP}
\Pi=\sum_{i=1}^N \varepsilon_{(T_i,X_i)}
\end{align}
where 
\begin{equation}\label{MPPb} 
\Psi=\sum_{i=1}^N \varepsilon_{T_i}
\end{equation} 
is a finite point process on $\cT$ representing the claim arrivals and, independently, $X_1,X_2,\ldots$ are i.i.d. random variables on $\cX\setminus\{0\}$ with regularly varying distribution $\nu$.  We  consider three particular situations, namely, marked Poisson point processes,  independently marked point processes, and triangular arrays of independently marked point processes.

\subsection{Regular variation for marked Poisson point processes}\label{subsec:MPP}

We first focus on the simple situation of a marked Poisson point process, that is the base point process $\Psi$ in Equation~\eqref{MPPb} is a Poisson point process on $\cT$ with finite intensity measure $\lambda\in\bbM_b(\cT)$. Then $\Pi$ is a Poisson point process with product intensity measure $\lambda \otimes \nu$. Recall the notation  $E=\cT\times \cX$, $F=\cT\times\{0\}$ and $\cN=\cN(E\setminus F)$ the space of point measures on $E\setminus F$. For $k\geq 0$, $\cN_k\subset\cN$ denotes the closed cone of point measures with at most $k$ points. 

\begin{theorem}[RV for marked Poisson point processes] \label{thm:RV-PP1}
Consider  $\Pi$ a Poisson point process on $E\setminus F$ with intensity $\lambda(\rmd t)\nu(\rmd x)$ where $\lambda\in \bbM_b(\cT)$ and  $\nu \in RV(\cX\setminus\{0\}, (a_n),\mu)$. Then, for $k\geq 0$,
 \begin{align}\label{thm:PiRV1}
 n^{k+1}\mathbb{P}(a_n^{-1}\Pi\in\cdot) \longrightarrow \mu^\ast_{k+1} (\cdot)\quad \mbox{in $\bbM(\mathcal{N}\setminus\cN_k)$},
 \end{align}
  where the limit measure $\mu^\ast_{k+1}$ is non-null and given by
\begin{align}\label{thm:mu_star1}
\mu^\ast_{k+1}(B)=\frac{1}{(k+1)!}\int_{E^{k+1}} \mathds{1}_{\{\sum_{i=1}^{k+1}\varepsilon_{(t_i,x_i)}\in B\}} \otimes_{i=1}^{k+1} \lambda(\rmd t_i)\mu(\rmd x_i),\quad  B\in \mathcal{B}(\mathcal{N}\setminus\cN_k).
\end{align}
\end{theorem}
Equivalently, Equation~\eqref{thm:PiRV1} can be rephrased in terms of regular variation as
\[
\Pi\in\mathrm{RV}(\mathcal{N}\setminus\cN_k, \{a_{[n^{1/(k+1)}]}\}, \mu^\ast_{k+1})
\]
or 
\begin{equation}\label{eq:succ-RV-bis}
b_n^{k+1}\mathbb{P}(n^{-1}\Pi\in\cdot) \longrightarrow \mu^\ast_{k+1}(\cdot) \quad \mbox{in $\bbM(\mathcal{N}\setminus\cN_k)$}
\end{equation}
where $b_n$ is such that $b_n\nu(n\cdot)\to \mu(\cdot)$ in $\bbM(\cX\setminus\{0\})$. If $\alpha$ denotes the regular variation index of $\nu$, then $\Pi$ is regularly varying in $\cN\setminus\cN_k$ with index  $(k+1)\alpha$. The limit measure $\mu^\ast_{k+1}$ is the image  of  $(\lambda\otimes \mu)^{ \otimes k+1}$ under the mapping 
$$(z_1,\ldots,z_{k+1})\in E^{k+1} \longmapsto \sum_{i=1}^{k+1} \varepsilon_{z_i}\in\cN$$ 
and  is concentrated on the  cone of point measures with exactly $k+1$ points. 

\medskip
Theorem~\ref{thm:RV-PP1} provides successive hidden regular variation as discussed in Section~\ref{sec:BRV} and can be interpreted as a single/multiple large point heuristic, see  \cite{RBZ19}. For $k=0$,  the regular variation writes $b_n\mathbb{P}(\Pi/n\in\cdot)\longrightarrow \mu^\ast_1(\cdot)$ in $\bbM(\cN\setminus\{0\})$. This is a single large jump heuristic  since the limit measure is supported by the cone $\cN_1$ of point measures with at most one point.  Removing this cone, we obtain  hidden regular variation with index $2\alpha$ in $\bbM(\cN\setminus\cN_1)$ and the convergence  $b_n^2\mathbb{P}(\Pi/n\in\cdot)\longrightarrow \mu^\ast_2(\cdot)$. This is a multiple large jump heuristic with two large points and $\mu^\ast_2$ is supported by the cone $\cN_2$. Removing the cone $\cN_2$, we obtain hidden regular variation with index  $3\alpha$ in $\bbM(\cN\setminus\cN_2)$ and so forth.

\subsection{Independently marked point processes}\label{subsec:general}
The results on regular variation for marked Poisson point processes are extended to general independently marked point processes as defined in Equation~\eqref{MPP}. For $k\geq 1$,  the $k$-th factorial moment measure of the base point process $\Psi=\sum_{i=1}^N \varepsilon_{T_i}$ is defined  by
\[
M_{k}(A)=\bbE\left[\Psi^{(k)}(A) \right],\quad A\in  \mathcal{B}(\cX^k) ,
\]
where
\[
\Psi^{(k)}=\sum_{1\leq i_1\neq\cdots\neq i_k\leq N}\varepsilon_{(T_{i_1},\ldots,T_{i_k})}
\]
is the $k$-th factorial power of $\Psi$. The $k$-th factorial moment measure $M_k$ is finite if and only if  $N$ has a finite moment of order $k$, see \cite{DVJ03} Chapter 5.2  for more details on these notions.

\begin{theorem}[RV for independently marked point processes] \label{thm:RV-PP2}
Consider the independently marked point process $\Pi$ defined by Equation~\eqref{MPP} with $\nu\in RV(\cX\setminus\{0\}, (a_n),\mu)$. Assume that, for $k\geq 0$, the base point process $\Psi$ has a finite and non-null $(k+1)$-th factorial moment measure $M_{k+1}$.  Then,
 \begin{align}\label{thm:PiRV2}
 n^{k+1}\mathbb{P}(a_n^{-1}\Pi\in\cdot) \longrightarrow \mu^\ast_{k+1}(\cdot) \quad \mbox{in $\bbM(\mathcal{N}\setminus\cN_k)$},
 \end{align}
with non-null limit measure defined, for $B\in \mathcal{B}(\mathcal{N}\setminus\cN_k)$, by
 \begin{align}\label{thm:mu_star2}
\mu_{k+1}^\ast(B)=\frac{1}{(k+1)!}\int_{E^{k+1}} \mathds{1}_{\left\{\sum_{i=1}^{k+1} \varepsilon_{(t_i,x_i)}\in B\right\}} M_{k+1}(\rmd t_1,\ldots,\rmd t_{k+1})\otimes_{i=1}^{k+1}\mu(\rmd x_i).
\end{align}
\end{theorem}
Theorem~\ref{thm:RV-PP2} is indeed a generalization of Theorem~\ref{thm:RV-PP1}: for a Poisson point process $\Psi$ with finite intensity measure $\lambda$, the $k$-th factorial moment measure is finite for all $k\geq 1$ and equal to $M_{k}=\lambda^{\otimes k}$, so that Equations \eqref{thm:mu_star1} and \eqref{thm:mu_star2}  agree. 

\begin{example}\label{ex:renewal}
We provide an application of Theorem~\ref{thm:RV-PP2} and consider a stationary renewal point process on $\mathbb{R}$ observed on a finite time window $\mathcal{T}=[0,T]$ and regularly varying marks on $\mathcal{X}$. The distribution of the point process  is completely determined by the inter-arrival distribution $G$ on $(0,\infty)$  assumed to have a finite first moment $\tau>0$. The construction is as follows, see e.g. \citet[Chapter 4.2]{DVJ03} for more details. Let $(T_i)_{i\geq 1}$ be a sequence of positive random variables (arrival times) such that $T_1$ follows the equilibrium distribution $G_{eq}(dt)=\tau^{-1} tG(dt)$ and the inter-arrival times $T_{i+1}-T_i$, $i\geq 1$, are i.i.d. with distribution $G$ and independent of $T_1$. The number of arrivals up to time $t$ is given by the counting process $N(t)=\sum_{i\geq 1}\mathds{1}_{\{T_i\leq t\}}$, $t\geq 0$, and the renewal point process observed on the finite window $[0,T]$ is the finite point process  $\Psi= \sum_{i=1}^{N(T)}\varepsilon_{T_i}$. The existence of a first moment $\tau>0$ for $G$ and the choice of the initial distribution $T_1\sim G_{eq}$ ensure that $\Psi$ has intensity measure $M_1(\rmd t)=\tau^{-1}\rmd t$. For the sake of simplicity, we assume that $G$ has density $g$ so that the renewal measure $U=\sum_{i\geq 1} G^{\ast i}$ has Radon-Nikodym derivative $u=\sum_{i=1}^\infty g^{\ast i}$. Renewal theory yields the following expression for  higher order factorial moment measures:
\[
M_k(\rmd t_1,\ldots,\rmd t_k)=\tau^{-1}u( t_{(2)}-t_{(1)})\cdots u( t_{(k)}-t_{(k-1)})\rmd t_1\ldots\rmd t_k \quad k\geq 2,
\]
with $t_{(1)}<\ldots<t_{(k)}$ the order statistics pertaining to $(t_1,\ldots,t_k)$; see \citet[Example 5.4.b p 139]{DVJ03}. Then Equations~\eqref{thm:PiRV2} and~\eqref{thm:mu_star2} provide formulas for the successive hidden regular variation of the marked point process $\Pi=\sum_{i=1}^{N(T)} \varepsilon_{(T_i,X_i)}$. 

A completely explicit example is given by the inter-renewal distribution $G(\rmd t)=g(t)\rmd t$ with Gamma density $g(t)=te^{-t}\mathds{1}_{[0,\infty)}(t)$ and first moment $\tau=2$. The convolution property of the Gamma family entails $g^{\ast i}(t)=t^{2i-1}e^{-t}/(2i-1)!\mathds{1}_{[0,\infty)}(t)$. We deduce the renewal density $u(t)=(1-e^{-2t})/2\mathds{1}_{[0,\infty)}(t)$ by recognizing the hyperbolic sine in the power series. For $k=2$, the factorial moment is given by
\[
M_2(\rmd t_1,\rmd t_2)=\frac{1}{4}\left(1-e^{-2(t_{(2)}-t_{(1)})}\right)\rmd t_1\rmd t_2.
\]
Observe that the factorial density vanishes on the diagonal $t_1=t_2$ which corresponds to a repulsive effect compared to the Poisson case; see also Example 4.11, where further calculations for higher factorial moment measures  are given. This simple explicit example could be generalized to the class of matrix-exponential distributions for which the renewal density is available in analytic form, see~\cite{AB97} Theorem~3.1 for more details.
\end{example}

\subsection{Marked point processes based on triangular arrays}
This section is motivated by the following simple situation. When $E=[0,T]\times [0,\infty)$, consider i.i.d. claim sizes $X_i$ and a finite deterministic number of claims $m$ arising at discrete times $T_i=Ti/m$, $i=1,\ldots, m$. We are interested in the asymptotic regime when $m=m_n\to\infty$ and the regular variations in this regime. A singular feature of the arrival times is that  the empirical distribution $m_n^{-1}\sum_{i=1}^{m_n} \varepsilon_{iT/m_n}$  converges weakly as $n\to\infty$ to the uniform distribution on $[0,T]$.

The general abstract setting is the following: on $E=\cT\times\cX$, we consider a sequence  of independently marked point processes   
\begin{align}\label{triarry}
\Pi_n=\sum_{i=1}^{N_n} \varepsilon_{(T_i^n,X_i)},\quad n\geq 1,
\end{align}
where  the number of points $N_n$ is random with finite expectation $\mathbb{E}[N_n]=m_n$ and $m_n\to \infty$ as $n\to\infty$, the arrival times are given by a triangular array of $\cT$-valued random variables $\{T_i^n,n\geq 1, 1\leq i\leq N_n\}$ and, independently, the marks $X_i$ are i.i.d.  with distribution $\nu$ on $\cX\setminus\{0\}$. We assume no  independence in the triangular array, but we suppose the weak convergence in probability of the empirical distribution to a probability measure $\lambda$, that is 
\begin{equation}\label{Psi_triangle}
m_n^{-1}\Psi_n:=m_n^{-1}\sum_{i=1}^{N_n} \varepsilon_{T_i^n} \overset{\bbM_b(\cT)}{\longrightarrow} \lambda\quad \mbox{in probability as $n\to\infty$}.
\end{equation} 
The main difference from the previous Sections \ref{subsec:MPP} and \ref{subsec:general} is that the mean number of points $m_n=\mathbb{E}[N_n]$ tends to infinity as $n\to\infty$  whereas it was previously fixed.  We prove in the following theorem that similar regular variation results still hold, but with different rates.

\begin{theorem}[RV for sequences of marked point processes based on triangular arrays] \label{thm:RV-PP3}
Consider the sequence  of independently marked point processes    $\Pi_n$, $n\geq 1$, defined by Equation \eqref{triarry}. Assume  that $\nu\in RV(\cX\setminus\{0\}, (a_n),\mu)$ and  that Equation~\eqref{Psi_triangle} holds.
For $k\geq 0$, assume that $(N_n/m_n)^{k+1}$, $n\geq 1$, is uniformly integrable. Then,
 \begin{align}\label{thm:PiRV3}
 n^{k+1}\mathbb{P}(a_{nm_n}^{-1}\Pi_n\in\cdot) \longrightarrow \mu^\ast_{k+1}(\cdot) \quad \mbox{in $\bbM(\mathcal{N}\setminus\cN_k)$},
 \end{align}
with non-null limit measure $\mu_{k+1}^\ast$  as in Theorem~\ref{thm:RV-PP1} Equation~\eqref{thm:mu_star1} and $a_{nm_n}=a_{[nm_n]}$.
\end{theorem}
In the following examples, we apply Theorem \ref{thm:RV-PP3} interchanging the roles of the sequences $(m_n)$ and $(n)$ in order to compare the result with classical large deviation principles.
\begin{example}\label{ex:rescbinproc}
Consider a probability measure $\lambda\in\mathbb{M}_b(\mathcal{T})$ and  the binomial point process $\Pi_n= \sum_{i=1}^n\varepsilon_{(T_i,X_i)}$, 
with the $(T_i,X_i)$, $1\leq i\leq n$ i.i.d. with distribution $\lambda\otimes\nu$.  By the law of large numbers,  Equation~\eqref{Psi_triangle} holds  interchanging the roles of the sequences $(m_n)$ and $(n)$ and  with $T_i^n=T_i$. If $\nu\in RV(\cX\setminus\{0\}, (a_n),\mu)$, Theorem \ref{thm:RV-PP3} provides the convergence 
\begin{equation}\label{eq:RV-PP3-particularcase}
m_n^{k+1}\bbP(a_{nm_n}^{-1}\Pi_n\in \cdot) \longrightarrow \mu^\ast_{k+1}(\cdot) \quad \mbox{in $\bbM(\mathcal{N}\setminus\cN_k)$}\,,\qquad k\ge0\,,
\end{equation}
for any sequence $m_n\to \infty$ as $n\to \infty$, with $ \mu^\ast_{k+1}$ as in Equation~\eqref{thm:mu_star1}.
\end{example}

\begin{example}
A typical situation where Theorem~\ref{thm:RV-PP3} applies is when $\mathcal{T}=[0,T]$ and $T_i^n=iT/n$, $1\leq i\leq n$  interchanging the roles of the sequences $(m_n)$ and $(n)$ as above.  Then Equation~\eqref{Psi_triangle} holds with $\lambda$ the uniform distribution on $[0,T]$. If $\nu\in RV(\cX\setminus\{0\}, (a_n),\mu)$, it is well-known that 
$\Pi_n= \sum_{i=1}^n\varepsilon_{(iT/n,X_i)}$ suitably rescaled converges in distribution to a Poisson point process \citep[Theorem 6.3]{Resnick2007}. More precisely, $a_{n}^{-1}\Pi_n\stackrel{d}\longrightarrow \Pi$ in $\mathcal{N}$, with $\Pi$ a Poisson point process with intensity $\lambda \otimes \mu$. Theorem \ref{thm:RV-PP3} considers the large deviation regime and states the  regular variation properties of $\Pi_n$ as in Equation~\eqref{eq:RV-PP3-particularcase}. In particular when $\alpha>1$ and using the notation of Equation \eqref{eq:succ-RV-bis} we obtain 
\[
(b_n/n)^{k+1}\bbP(n^{-1}\Pi_n\in \cdot) \longrightarrow \mu^\ast_{k+1}(\cdot) \quad \mbox{in $\bbM(\mathcal{N}\setminus\cN_k)$}\,,\qquad k\ge0\,.
\]
To be even more specific,  assume  that the marks $(X_i)_{i\geq 1}$ are i.i.d. with Pareto distribution $\nu$, i.e. $\nu((x,\infty))=\mathbb{P}(X_i>x)=x^{-\alpha}$, $x>1$, $\alpha>1$. Then $b_n\nu(n\cdot)\to \mu(\cdot)$ in $\bbM(\cX\setminus\{0\})$ with sequence $b_n=n^{\alpha}$ and we obtain the large deviations of $\Pi_n$ with successive rates $n^{(k+1)(\alpha-1)}$. 
\end{example}

\begin{example}\label{ex:ldpp}
A slightly more complex situation is based on a stationary renewal sequence 
$(T_i)_{i\geq 1}$ on $\mathcal{T}=[0,\infty)$. As in Example~\ref{ex:renewal}, we assume the  inter-arrival times $T_i - T_{i-1}$ to be i.i.d. with finite mean $\tau>0$. We observe the process on a growing window $[0,nT]$, $T>0$, $n\geq 1$. Let $N_n=N(nT)$ be the number of arrivals up to time $nT$ and $T_i^n=T_i/n$, $1\leq i\leq N_n$. By stationarity, $m_n=\bbE[N_n]=nT/\tau$. Renewal theory ensures that   $m_n^{-1}\Psi_n=m_n^{-1}\sum_{i=1}^{N(nT)}\varepsilon_{T_i/n}$ converges to the uniform distribution $\lambda$ on $[0,T]$ so that assumption~\eqref{Psi_triangle} holds.  Furthermore, $\mathbb{E}[(N_n/m_n)^k]\to 1$ for all $k\geq 0$ which implies the uniform integrability assumptions, see Exercise 4.1.2 in \cite{DVJ03}.  Theorem~\ref{thm:RV-PP3} states the successive regular variations of  $\Pi_n= \sum_{i=1}^{N(nT)}\varepsilon_{(T_i/n,X_i)}$: if $\nu\in \mathrm{RV}(\cX\setminus\{0\}, (a_n),\mu)$, then
$$
n^{k+1}\bbP(a_{nm_n}^{-1}\Pi_n\in \cdot) \longrightarrow \mu^\ast_{k+1}(\cdot) \quad \mbox{in $\bbM(\mathcal{N}\setminus\cN_k)$}\,.
$$
\end{example}

\section{Applications}\label{sec:app}
\subsection{Regular variation of risk processes in Skorokhod space}\label{sec:app1}
We focus on risk processes that are based on marked point processes on $\cT\times\cX=[0,T]\times [0,\infty)$. Recall that the marked point process
\[
\Pi=\sum_{i=1}^N \varepsilon_{(T_i,X_i)}
\]
represents  the situation where $N$ claims arise on $[0,T]$ at  times $0\leq T_1\leq\cdots\leq T_N\leq T$ and with sizes $X_1,\ldots,X_N>0$.  In other words, $N$ is an integer-valued random variable defined as $N:= \# \{ i \geq 1 : T_i \leq T \}$ that is finite almost surely. We focus here on the associated risk process defined as
\begin{equation}\label{eq:risk_process}
R(t)=\sum_{i=1}^N X_i \mathds{1}_{\{T_i\leq t\}},\quad t\in[0,T].
\end{equation}
For an insurance company, it models the evolution over time of the total claim amount. When the claim arrival times $T_1\leq \cdots\leq T_N$ form an homogeneous Poisson point process and the claim sizes are i.i.d.,  $(R(t))_{0\leq t\leq T}$ is a compound Poisson process. More generally, we consider the case when $\Pi$ is an independently marked point process as in Section~\ref{subsec:general}.
The risk process $R$ is a pure jump process and can be seen as a random element of the Skorokhod space $\mathbb{D}=\mathbb{D}([0,T],\mathbb{R})$ of c\`ad-l\`ag functions. Recall that endowed with the Skorokhod metric, $\mathbb{D}$ is a complete separable metric space,  see \citet[Chapter 3]{B68}. We define $\mathbb{D}_k\subset\mathbb{D}$ the closed cone of c\`ad-l\`ag functions with at most $k$ discontinuity points (or jumps) on $[0,T]$. 

In the next theorem, we derive the (hidden) regular variation properties of the risk process $R$ on $\mathbb{D}$ from the regular variation properties of $\Pi$ on $\cN$. The statement is very similar to Theorem~\ref{thm:RV-PP2} and is in fact derived from it using a continuous mapping theorem together with technical truncation arguments relying on Proposition~\ref{prop:HRV-criterion}.  

\begin{theorem}\label{theo:RV-D} Let $\cT\times\cX=[0,T]\times [0,\infty)$ and $\Pi$ be an independently marked point process on $\cT\times\cX$ as defined in \eqref{MPP}. Assume $\nu\in\mathrm{RV}([0,\infty)\setminus\{0\},\{a_n\},\mu)$ with $\mu(\rmd x)=\alpha x^{-\alpha-1}\rmd x$ for some $\alpha>0$. Let $k\geq 0$ and assume $\Psi$ has a finite and non-null $(k+1)$-th factorial moment measure noted $M_{k+1}$.  Then 
\begin{equation}\label{eq:thm:RV-D}
 n^{k+1}\mathbb{P}(a_{n}^{-1}R\in\cdot) \longrightarrow \mu^{\#}_{k+1}(\cdot) \quad \mbox{in $\bbM(\mathbb{D} \setminus \mathbb{D}_k)$},
 \end{equation}
 with non-null limit measure defined, for $B\in\cB( \mathbb{D} \setminus \mathbb{D}_k)$, by
\begin{equation}\label{eq:thm:RV-D-mu}
\mu^{\#}_{k+1}(B)=\int_{E^{k+1}} \mathds{1}_{\left\{\left(\sum_{i=1}^{k+1}x_i\mathds{1}_{\{t_i\le u\}}\right)_{0\leq u\leq T}\in B\right\}} M_{k+1}(\rmd t_1,\ldots,\rmd t_{k+1})\otimes_{i=1}^{k+1}\mu(\rmd x_i).
\end{equation}
\end{theorem}

 The limit measure $\mu^{\#}_{k+1}$ is the image measure of  $\mu^*_{k+1}$ defined in Equation~\eqref{thm:mu_star2} under the mapping
\[
\pi=\sum_{i\in I}\varepsilon_{(t_i,x_i)}\in\cN \mapsto \left(\sum_{i\in I} x_i\mathds{1}_{\{t_i\leq u\}}\right)_{0\leq u\leq T}\in\mathbb{D}
\]
with $I$ a countable index set. In the case $k=0$,  $\mathbb{D}_0$ is the space of continuous functions on $[0,T]$ and Theorem~\ref{theo:RV-D} provides hidden regular variation in  $\mathbb{D}\setminus\mathbb{D}_0$ with a single large jump heuristic: in this regime, the rescaled distribution of the risk process $R$ converges in $\bbM(\mathbb{D}\setminus \mathbb{D}_0)$ to $\mu^{\#}_1$ which is concentrated on the cone of c\`ad-l\`ag functions with exactly one jump.
For $k\geq 1$, we obtain successive hidden regular variation and multiple large jump heuristics: removing the cone $\mathbb{D}_k$, we obtain a limit measure $\mu^{\#}_{k+1}$ concentrated on the cone of c\`ad-l\`ag functions with  $k+1$ jumps. This is closely related to the results by~\cite{RBZ19}.

\begin{remark}
The cone $\mathbb{D}_k$ is in fact larger than the exact support of $\mu^{\#}_k$. Indeed, in view of Equation~\eqref{eq:thm:RV-D-mu}, the support of $\mu^{\#}_k$ is the cone of pure jump process with exactly $k$ jumps. Let us denote by $\mathbb{J}_k$, $k\geq 0$, the cone  consisting of pure jump functions with at most $k$ jumps. In comparison, $\mathbb{D}_k$ is the cone of c\`ad-l\`ag functions with at most $k$ jumps and is strictly larger than $\mathbb{J}_k$. When considering successive hidden regular variations as discussed in the end of Section~\ref{sec:BRV}, we should rather consider the increasing sequence of cones $\mathbb{J}_k$, $k\geq 0$, instead of the sequence $\mathbb{D}_k$, $k\geq 0$. Such results are stronger but harder to establish. In Corollary~\ref{cor:reinsurance} below, we cover the triangular array case in connection with an application to reinsurance.
\end{remark}

\medskip
The next theorem considers  risk processes built on triangular arrays. Let $\Pi_n$, $n\geq 1$, be the sequence of marked point processes  defined by Equation~\eqref{triarry}, where for simplicity $N_n=\mathbb{E}[N_n]=m_n$ is deterministic. When $\alpha>1$, it is as usual necessary to center the risk process  and we consider
\[
\tilde{R}_n(t)=\sum_{i=1}^{m_n} (X_i-c)\mathds{1}_{\{T_i^n\leq t\}},\quad t\in [0,T], n\geq 1,
\]
where $c=\bbE[X]$ if $X$ has a finite expectation and $c=0$ otherwise. 
\begin{theorem}\label{theo:RV-D2} Let $\cT\times\cX=[0,T]\times [0,\infty)$ and $\Pi_n$, $n\geq 1$, be the sequence of marked point processes defined by Equation~\eqref{triarry}. Assume $\Psi_n$  satisfies assumption~\eqref{Psi_triangle} and the regular variation condition  $\nu\in\mathrm{RV}([0,\infty)\setminus\{0\},\{a_n\},\mu)$ with $\mu(\rmd x)=\alpha x^{-\alpha-1}\rmd x$ for some $\alpha>0$. Depending on the value of $\alpha>0$, we assume furthermore:
\begin{itemize}
\item[-] $\bbE[X_1]<\infty$ if $\alpha =1$;
\item[-] $\bbE[X_1^2]<\infty$ if $\alpha=2$;
\item[-] $a_{nm_n}^2/m_n\to \infty$ as $n\to\infty$ if $\alpha \ge 2$.
\end{itemize}
Then, for $k\geq 0$, 
\begin{align}\label{eq:thm:RV-D2}
 n^{k+1}\mathbb{P}(a_{nm_n}^{-1}\tilde R_n\in\cdot) \longrightarrow \mu^{\#}_{k+1}(\cdot) \quad \mbox{in $\bbM(\mathbb{D} \setminus \mathbb{D}_k)$},
 \end{align}
 with $\mu^{\#}_{k+1}$ given by Equation~\eqref{eq:thm:RV-D-mu}.
\end{theorem}

\begin{example}
We consider the setting of \cite{RBZ19}. Let $(X_n(t))_{0\le t\le 1}$ be a L\'evy process  with L\'evy measure $\nu$ and jump part 
$$
J_n(t)=\int_{x>1} xN([0,nt]\times dx)
$$
where $N$ is the Poisson random measure with  measure $Leb\otimes \nu$ on $[0,n]\times (0,\infty)$ and $Leb$
denotes the Lebesgue measure. Then $J_n(t)$ is a compound Poisson process  and its centered version $\tilde J_n$ is tail equivalent to the centered version of the L\'evy process $X_n$; see Proposition 6.1  of \cite{RBZ19}. The process $\tilde J_n$ is close to $\tilde R_n$ when approximating the number of jumps $N([0,n]\times [1,\infty))$ by its expectation $n$, interchanging the roles of $m_n$ and $n$ and considering $T_i^n$ as in Example \ref{ex:ldpp} for $T=\tau=1$. Theorem \ref{theo:RV-D2} yields
$$
m_n^{k+1}\mathbb{P}(a_{nm_n}^{-1}\tilde R_n\in\cdot) \longrightarrow \mu^{\#}_{k+1}(\cdot) \quad \mbox{in $\bbM(\mathbb{D} \setminus \mathbb{D}_k)$}\,.
$$ 
This is the one-sided large deviation principles for centered L\'evy processes of \cite{RBZ19} considering $(m_n)$ so that $a_{n m_n}\sim n$, i.e. $m_n\sim (n\nu(n,\infty))^{-1}\to \infty$ as $n\to \infty$. It shows that $X_n$, $\tilde J_n$ and $\tilde R_n$ satisfies the same one-sided large deviation principles.
\end{example}

\begin{remark}
Some remarks on the conditions of Theorem \ref{theo:RV-D} are in order.\\
The integrability conditions for $\alpha= 1$ or $\alpha=2$ could be dropped thanks to extra classical but technical arguments. For the sake of simplicity, we focus on the integrable cases only.\\
Because the sequence $(a_n)$ is regularly varying with index $1/\alpha$, the condition $a_{nm_n}^2/m_n\to\infty $ is satisfied as soon as $\alpha>2$ and $m_n=o(n^p)$ for some $p<2/(\alpha-2)$. This ensures that the growth $m_n\to\infty$ is not too fast.
\end{remark}
\begin{remark} When $\alpha\geq 1$, the centering is not necessary under the extra assumption $a_{nm_n}/m_n\to \infty$. That is the risk process 
$R_n(t) = \sum_{i=1}^{m_n} X_i\mathds{1}_{\{T_i^n\leq t\}}$ satisfies
\[
 n^{k+1}\mathbb{P}(a_{nm_n}^{-1}R_n\in\cdot) \longrightarrow \mu^{\#}_{k+1}(\cdot) \quad \mbox{in $\bbM(\mathbb{D} \setminus \mathbb{D}_k)$}.
\]
This is immediately derived  from Equation \eqref{eq:thm:RV-D2} because the magnitude of the centering term is bounded by $m_n/a_{nm_n}\to 0$. Note that 
the condition  $a_{nm_n}/m_n\to\infty $ holds as soon as $\alpha>1$ and $m_n=o(n^p)$ for some $p<1/(\alpha-1)$.
\end{remark}

\subsection{Reinsurance of the largest claims}\label{sec:app2}
We provide in this section  an application of the preceding results to risk theory with the study of a reinsurance model. The reader may find an exhaustive review of risk theory and the mathematical issues that it raises in \cite{Assmussen2010} and \cite{Mikosch}. We focus on a reinsurance treaty of extreme value type called the \textit{largest claims reinsurance}. Assume that at the time the contract is underwritten, say $t=0$, the reinsurance company guarantees that it will cover the $k$ largest claims over the period $[0,T]$. For a risk process of the form \eqref{eq:risk_process}, the contract covers the risk
\[
R_k^+=\sum_{i=1}^k X_{N+1-i:N},
\]
where $X_{1:N}\leq \cdots\leq X_{N:N}$ denote the order statistics of $X_1,\ldots,X_N$. The value $R_k^+$ is known only at time $T$ and, during the contract life, the covered risk  evolves as
\[
R_k^+(t)=\sum_{i=1}^k X_{N(t)+1-i:N(t)}, \quad t\in [0,T],
\]
where $N(t)=\sum_{i=1}^N \mathds{1}_{\{T_i\leq t\}}$ is the number of claims up to time $t$ and $X_{1:N(t)}\leq \cdots\leq X_{N(t):N(t)}$  the order statistics of $X_1,\ldots,X_{N(t)}$. We use here the convention $X_{N(t)+1-i:N(t)}=0$ if $i>N(t)$. The subscriber of the contract needs to assess its residual risk 
\[
R_k^-=\sum_{i=1}^{N-k} X_{i:N}
\]
that is not covered by the reinsurance treaty and its evolution over time is
\[
R_k^-(t)=\sum_{i=1}^{N(t)-k} X_{i:N(t)}, \quad t\in [0,T].
\]

Our results state the regular variation properties of the residual risk $R_k^-$. The following proposition is crucial in our approach. When $k=0$, it states that the sum of independent regularly varying random variables has the same tail behavior as their maximum. For non-ordered random variables, the result has been shown in \cite{tillier2018regular} Proposition 7 under similar moment conditions.
To our best knowledge, the  more general statement with random number of terms $N$ and arbitrary order statistic $k\geq 0$ is new. 

\begin{proposition}\label{prop:reinsurance}
Let  $(X_i)_{i \geq 1}$ be i.i.d.  non-negative random variables with cumulative distribution function $F$, assumed to be regularly varying with index $\alpha>0$. Independently, let $N$ be random variable with values in $\mathbb{N}$. Assume  $\mathbb{E}[N^{k+1}]<\infty$, and, if $\alpha\geq 1$, assume furthermore that $\mathbb{E}[N^p]<\infty$ for some $p>(k+1) \alpha$. Then,  
\begin{align*}
\mathbb{P}\left(  \sum_{i=1}^{N-k}  X_{i:N} >x \right) \sim \mathbb{P}\left( X_{N-k:N}>x\right) \sim \frac{\mathbb{E}[N^{[k+1]}]}{(k+1)!} (1-F(x))^{k+1},\quad \mbox{as $x\to\infty$},
\end{align*} 
with $\mathbb{E}[N^{[k+1]}] =\mathbb{E}[N!/(N-k-1)!] $ the $(k+1)$-th factorial moment of $N$.
\end{proposition}

Interestingly, we have the relationships 
\begin{equation}\label{eq:dist_expl}
X_{N-k:N}=2d(R,\mathbb{D}_k) \quad \mbox{and}\quad \sum_{i=1}^{N-k}  X_{i:N}=d(R,\mathbb{J}_k),
\end{equation}
where $d$ is the Skorokhod metric, $\mathbb{D}_k$ is the cone of c\`ad-l\`ag functions with at most $k$ jumps (see Lemma~\ref{lem:dist-Dk} below) and $\mathbb{J}_k$ is the cone of piecewise constant c\`ad-l\`ag functions with at most $k$ jumps. Proposition~\ref{prop:reinsurance} states that, for large $x$, the events $\{d(R,\mathbb{D}_k)\geq x/2\}$ and $\{d(R,\mathbb{J}_k)\geq x \}$ are asymptotically equivalent (note that one inclusion always holds). As a consequence, we obtain the following corollary.
\begin{corollary}\label{cor:reinsurance}
Assume that the assumptions of Theorem~\ref{theo:RV-D} are satisfied. When $\alpha\geq 1$, assume furthermore that $\mathbb{E}[N^p]<\infty$ for some $p>(k+1)\alpha$. Then, the conclusion of Theorem~\ref{theo:RV-D} holds with $\mathbb{D}\setminus \mathbb{D}_k$ replaced by $\mathbb{D}\setminus\mathbb{J}_k$.
\end{corollary}
 In terms of successive hidden regular variations, it is sensible to work with $\mathbb{J}_k$ instead of $\mathbb{D}_k$ because the support of $\mu^{\#}_{k}$ is exactly $\mathbb{J}_k$. From a technical point of view, the results are stronger and more difficult to establish in $\mathbb{D}\setminus\mathbb{J}_k$ because the distance $d(R,\mathbb{J}_k)$ involves a sum of order statistics so that extra integrability conditions are required to control the sum.
 
 \smallskip
 Going back to our original problem of largest claims in reinsurance, we deduce from Corollary~\ref{cor:reinsurance} the following results.
\begin{proposition}\label{prop:reinsurance2}
Assume that the assumptions of Theorem~\ref{theo:RV-D} are satisfied and note $F(x)=\nu([x,\infty))$, $x\geq 0$. When $\alpha\geq 1$, assume furthermore that $\mathbb{E}[N^p]<\infty$ for some $p>(k+1)\alpha$.  Then:
\begin{enumerate} 
\item[i)] (regular variation)  The residual risk satisfies
\[
\mathbb{P}\left(  R_k^- >x \right) \sim  \frac{\mathbb{E}[N^{[k+1]}]}{(k+1)!}(1-F(x))^{k+1},\quad \mbox{as $x\to\infty$}.
\] 
\item[ii)] (conditional limit theorem) The   typical behavior of the risk process given a large residual risk is given by
\[
\mathbb{P}\left(x^{-1}R\in \cdot \mid R_k^->x\right)\stackrel{d}{\longrightarrow} \mathbb{P}\left( \left(\sum_{i=1}^{k+1}Z_i\mathds{1}_{\{S_i\leq t\}}\right)_{t\in[0,T]}\in \,\cdot\,\right),\quad \mbox{as $x\to\infty$},
\]
where  $\stackrel{d}{\longrightarrow}$ stands for weak convergence in Skorokhod space $\mathbb{D}$, $Z_1,\ldots,Z_{k+1}$ are  i.i.d.  with standard $\alpha$-Pareto distribution  and, independently, $(S_1,\ldots,S_{k+1})$ has distribution $M_{k+1}(\rmd s_1,\ldots,\rmd s_{k+1})/M_{k+1}([0,T]^{k+1})$.
\item[iii)] (residual risk monitoring) Assume that $M_{k+2}$ is finite and non-null and that both $M_{k+1}$ and $M_{k+2}$ are continuous measures. If $\alpha\geq 1$, assume furthermore that $\mathbb{E}[N^p]<\infty$ for some $p>(k+2)\alpha$. Then, for all $0<t_0<t_1<T$ and $u>1$, 
\begin{align}
&\lim_{\varepsilon \to 0}\lim_{x\to\infty}  (1-F(x))^{-1} \mathbb{P}\left(  R_k^-(t_1) > ux \mid x <R_k^-(t_0)<(1+\varepsilon)x \right)\nonumber\\
&=\frac{M_{k+2}\left([0,t_0]^{k+1}\times (t_0,t_1] \right)}{M_{k+1}\left([0,t_0]^{k+1}\right)}\left( (u-1)^{-(k+1)\alpha}+\left((u-1)^{-\alpha}-1\right)_+\right).\label{eq:item3}
\end{align}
\end{enumerate}
\end{proposition}
The third item allows to monitor the residual risk during the contract lifetime and assesses the risk of a larger loss  at time $t_1>t_0$ given that the loss is approximately $x$ at time $t_0$. The regular variation of the  conditional  probability with respect to the null event $R_k^-(t_0)=x$ does not follow from Corollary~\ref{cor:reinsurance} explaining why we introduce the approximate conditioning with $\varepsilon\to 0$.

\begin{example}\label{examplePP} In the simple case of a compound Poisson process when the claims occurs according  an homogeneous Poisson point process  with intensity $\tau^{-1}>0$, the results from Proposition~\ref{prop:reinsurance2} simplify as follows. The residual risk is regularly varying such that 
\[
\mathbb{P}\left(  R_k^- >x \right) \sim  \frac{\tau^{-(k+1)}T^{k+1}}{(k+1)!}(1-F(x))^{k+1}\quad \mbox{as $x\to\infty$}.
\] 
The limiting conditional risk process in item $ii)$ corresponds to $k+1$ independent claim arrivals with  occurence times $S_1,\ldots,S_{k+1}$ uniform on $[0,T]$ and  independent magnitudes $Z_1,\ldots,Z_{k+1}$ with standard $\alpha$-Pareto distribution. Finally, the limit in Equation~\eqref{eq:item3} equals\[
\frac{(t_1-t_0)}\tau\left( (u-1)^{-(k+1)\alpha}+\left((u-1)^{-\alpha}-1\right)_+\right).
\]
More details on the calculations of the factorial moment measures are given in the following example.
\end{example}

\begin{example} Continuation of Example~\ref{ex:renewal}.  In the case of a renewal/reward process on $[0,T]$ when the claims occur according to a stationary renewal point process with mean inter-arrival time $\tau>0$,  Proposition~\ref{prop:reinsurance2} entails the following results. The residual risk is regularly varying such that 
\[
\mathbb{P}\left(  R_k^- >x \right) \sim  m_{k+1}(1-F(x))^{k+1},\quad \mbox{as $x\to\infty$},
\]
with $m_{k+1}$ the factorial moment of $N(T)$ of order $k+1$ given by
\[
 m_{k+1}=\int_{[0,T]^{k+1}} \tau^{-1}u(t_{(2)}-t_{(1)})\ldots u(t_{(k+1)}-t_{(k)})\rmd t_1 \rmd t_2\cdots \rmd t_{k+1}.
\] 
Then the limit distribution of the conditional risk process  arising in point $ii)$ is the distribution of $\sum_{i=1}^{k+1}Z_i\mathds{1}_{\{S_i\leq t\}}$, $t\in [0,T]$, where  $(S_1,\ldots,S_{k+1})$ has density 
\[
m_{k+1}^{-1} \tau^{-1}u(s_{(2)}-s_{(1)})\ldots u(s_{(k+1)}-s_{(k)})\mathds{1}_{[0,T]^{k+1}}(s)
\]
and, independently, $Z_1,\ldots,Z_{k+1}$ are i.i.d. with standard $\alpha$-Pareto distribution.
Finally, the residual risk monitoring Equation~\eqref{eq:item3} holds with 
\begin{align*}
M_{k+1}\left([0,t_0]^{k+1}\right)&=\int_{[0,t_0]^{k+1}} \tau^{-1}u(s_{(2)}-s_{(1)})\cdots u(s_{(k+1)}-s_{(k)})\rmd s_1\ldots\rmd s_{k+1},\\
M_{k+2}\left([0,t_0]^{k+1}\times (t_0,t_1]\right)&=\int_{[0,t_0]^{k+1}\times (t_0,t_1]} \tau^{-1}u(s_{(2)}-s_{(1)})\cdots u(s_{(k+2)}-s_{(k+1)})\rmd s_1\ldots\rmd s_{k+2}.
\end{align*}
In the explicit case $g(t)=te^{-t}\mathds{1}_{[0,\infty)}(t)$  mentioned in Example~\ref{ex:renewal},  $u(t)=(1-e^{-2t})/2\mathds{1}_{[0,\infty)}(t)$  and we compute
\[
m_2=\frac{1}{4}T^2-\frac{1}{8}(2T-1+e^{-2T}).
\]
Furthermore, for $k=1$,  Equation~\eqref{eq:item3} holds with 
\begin{align*}
M_{2}\left([0,t_0]^{2}\right)&=\frac{1}{4}t_0^2-\frac{1}{8}(e^{-2t_0}-1+2t_0),\\
M_{3}\left([0,t_0]^{2}\times (t_0,t_1]\right)&=\frac{1}{8}t_0^2(t_1-t_0)- \frac{1}{16}(t_1-t_0)(e^{-2t_0}-1+2t_0)\\
& +\frac{1}{16}{\left( e^{-2(t_1-t_0)}-1 \right)\left( t_0e^{-2t_0} +e^{-2t_0}+ t_0-1 \right)} .
\end{align*}
Letting $t_0\to 0^+$ in the limit of Equation~\eqref{eq:item3} we obtain the expression 
\[
\frac14\left(e^{-2t_1}+2t_1-1\right)\left( (u-1)^{-(k+1)\alpha}+\left((u-1)^{-\alpha}-1\right)_+\right)\,,\qquad t_1>0\,.
\]
It should be compared with the expression obtained in Example~\ref{examplePP} for   the Poisson process with mean inter-arrival $\tau=2$, letting $t_0\to 0^+$,
\[
\frac{t_1}2\left( (u-1)^{-(k+1)\alpha}+\left((u-1)^{-\alpha}-1\right)_+\right)\,,\qquad t_1>0.
\]
The lower residual risk monitoring for the Gamma distributed inter-arrivals renewal process may be seen as a consequence of the repulsive effect.
\end{example}

\section{Proofs related to Section~\ref{sec:RVPP}}\label{sec:proof2}
\subsection{Proof of Theorem~\ref{thm:RV-PP1}}
For the proof of Theorem \ref{thm:RV-PP1}, we need the following lemma that characterizes the convergence of the tensor product of measures in $\bbM(E\setminus F)$. We recall that $E=\cT \times \cX$ and $F=\cT \times \{ 0\}$.
\begin{lemma}\label{lem:cv_product}
	Assume $\nu_n\to \mu$ in $\bbM(\cX\setminus\{0\})$ and $\lambda\in\bbM_b(\cT)$. Then for $k\geq 1$, we have
	\[
	\left( \lambda\otimes \nu_n\right)^{\otimes k}\to \left( \lambda\otimes \mu \right)^{\otimes k} \quad \mbox{ in $\bbM(E^k\setminus \cup_{i=1}^k(E^{i-1}\times F\times E^{k-i}))$.}
	\]
\end{lemma}
\begin{proof}[Proof of Lemma \ref{lem:cv_product} ]
	Let $f:E\to\mathbb{R}$ be a continuous bounded function vanishing on $F^r=\cT\times B_{0,r}$ for some $r>0$. Possibly replacing $r>0$ by a smaller value, we can assume $\nu_n(B_{0,r}^c)\to \mu(B_{0,r}^c)$ (see Theorem 2.2 (i) in \cite{LR06} that asserts that this convergence holds for all but countably many $r>0$). By Fubini-Tonelli theorem,
	\[
	\int_{E\setminus F} f(t,x)\lambda(\rmd t)\nu_n(\rmd x)=\int_\cT \left(\int_{\cX\setminus\{0\}} f(t,x)\nu_n(\rmd x)\right)\lambda(\rmd t).
	\]
	For all $t\in\cT$, the function $x\mapsto f(t,x)$ is bounded continuous and vanishes on $B_{0,r}$ so that the convergence $\nu_n\to\mu$ in $\bbM(\cX\setminus\{0\})$ implies from Proposition \ref{ConvMEF} that
	\[
	\int_{\cX\setminus\{0\}} f(t,x)\nu_n(\rmd x)\to \int_{\cX\setminus\{0\}} f(t,x)\mu(\rmd x) \quad \mbox{as $n\to \infty$}.
	\]
	Furthermore, since $f$ is bounded, say by $M>0$, and since $\nu_n(B_{0,r}^c)\to \mu(B_{0,r}^c)$ is also bounded, say by $L>0$, we have
	\[
	\left|\int_{\cX\setminus\{0\}} f(t,x)\nu_n(\rmd x)\right|\leq M\nu_n(B_{0,r}^c)\leq ML,\quad t\in\cT.
	\]
	Lebesgue convergence theorem finally entails
	\[
	\int_{E\setminus F} f(t,x)\lambda(\rmd t)\nu_n(\rmd x)\to 
	\int_{E\setminus F} f(t,x)\lambda(\rmd t)\mu(\rmd x),
	\]
	proving the convergence $\lambda\otimes \nu_n\to \lambda\otimes \mu$  in $\bbM(E\setminus F)$. A direct application of  \citet[Lemma 2.2]{RBZ19} ensures that, for $ k\geq 1$,
\begin{align}\label{Tensor}
\left( \lambda \otimes \nu_n \right)^{\otimes k} \longrightarrow \left( \lambda \otimes \mu \right)^{\otimes k} \ \ \ in  \ \mathbb{M}\left( E^k\setminus \cup_{i=1}^k(E^{i-1}\times F\times E^{k-I}\right).
\end{align}	
\end{proof}

We are now ready to prove Theorem \ref{thm:RV-PP1}. For the sake of clarity, we begin with the proof of regular variation in the case $k=0$, before considering hidden regular variation in the case $k\geq 1$.

\begin{proof}[Proof of Theorem \ref{thm:RV-PP1}, case $k=0$] 	The proof uses similar arguments as in Theorem 3.3 in Dombry et al.\ (2018). We have to show that  
	\begin{align*}
	n \bbP (a_n^{-1}\Pi \in \cdot) \longrightarrow   \mu^\ast_1(\cdot) \ \ \ \text{in} \  \mathbb{M}( \mathcal{N}\setminus\{0\}).
	\end{align*}
	According to Theorem~\ref{theo:cv-N-Nk} Equation \eqref{Laplacek0}, the convergence holds if and only if
	\begin{align}\label{eqTh:conv}
	 \lim_{n \to \infty} n \bbE \left[ 1-\rme^{-(a_n^{-1}\Pi)(f)}\right]=\int_{\mathcal{N}}\left(1-\rme^{-\pi(f)}\right)\mu^\ast_1(\rmd \pi)
	\end{align}
for all $f:E\to [0,\infty)$ bounded continuous with support bounded away from $F$. 	By definition of $\mu^\ast_1$,   
		\[
	\int_{\mathcal{N}}\left(1-\rme^{-\pi(f)}\right)\mu^\ast_1(\rmd \pi)=\int_{ \cT\times\cX}\left(1-\rme^{-f(t,x)}\right) \lambda(dt) \mu(\rmd x).
	\]
	Besides, since $\Pi$ is a Poisson point process,
	 	\begin{align*}
	  n \bbE \left[ 1-\rme^{-(a_n^{-1}\Pi)(f)}\right] & =  n  \left( 1-  \bbE\left[ \rme^{-\int_{E} f(a_n^{-1}z)\Pi(\rmd z)}\right] \right)  \\
	   & = n \left(  1- \exp \left[ \int_{\cT\times\cX} \left( e^{-f(t,a_n^{-1}x)}-1 \right)   \lambda(dt) \nu (dx)\right] \right)\\
	  & = n \left(  1- \exp \left[ -\frac{1}{n} \int_{\cT\times\cX} \left(1- e^{-f(t,y)} \right)  \lambda(dt) \nu_n (dy)\right] \right)
	 \end{align*}
	with $\nu_n=n \nu(a_n \cdot )$. As the function $1 -e^{-f}$ is bounded continuous   with support bounded away from $F$ and as $\lambda\otimes \nu_n  \to \lambda\otimes\mu$ in $\bbM(E\setminus F)$ as $n \to \infty$ (see Lemma~\ref{lem:cv_product}), we have 
	\[
	\int_{\cT\times\cX} \left(1- e^{-f(t,y)} \right)  \lambda(dt) \nu_n (dy)\longrightarrow \int_{\cT\times\cX} \left(1- e^{-f(t,y)} \right)  \lambda(dt) \mu (dy),
\quad \mbox{as $n\to\infty$},
	\]
	and hence
	 \begin{align*}
	  n \bbE \left[ 1-\rme^{-\int_{E\setminus F} f(a_n^{-1}z)\Pi(\rmd z)}\right] & =  n \left(  1- \exp \left[ -\frac{1}{n} \int_{\cT\times\cX} \left(1- e^{-f(t,y)} \right)  \lambda(dt) \nu_n (dy)\right] \right)\\
	 &\longrightarrow\int_{ \cT\times\cX}\left(1-\rme^{-f(t,x)}\right) \lambda(dt) \mu(\rmd x).
	 \end{align*}
This proves Equation~\eqref{eqTh:conv} and concludes the proof of Theorem \ref{thm:RV-PP1}, case $k=0$.
\end{proof}

\begin{proof}[Proof of Theorem \ref{thm:RV-PP1}, case $k\geq1$] 	
Let $r>0$ be fixed such that $\mu(\partial B_{0,r}^c)=0$  and let $f:  \cT\times \cX \mapsto [0,\infty)$ be a bounded continuous function vanishing on $F^r=\cT\times B_{0,r}$. For $k\geq 1$, we will prove that
\begin{align}
&n^{k+1}\bbE\Big[\rme^{-\sum_{i=1}^Nf(T_i,a_n^{-1}X_i)} \mathds{1}_{\{\sum_{i=1}^N\varepsilon_{a_n^{-1}X_i}(B_{0,r}^c)\ge k+1\}}\Big]\nonumber\\
&\longrightarrow \frac1{(k+1)!}\int_{E^{k+1}} \rme^{-\sum_{i=1}^{k+1}f(t_i, x_i)} \otimes_{i=1}^{k+1}\mathds{1}_{\{x_i\in  B_{0,r}^c\}} \lambda(\rmd t_i)\mu(\rmd x_i) \nonumber\\
&=\int_{\mathcal{N}} e^{-\pi(f)} \mathds{1}_{\{\pi(B_{0,r}^c)\geq k+1\}}\mu^*_{k+1}(\rmd \pi)   \,.\label{eq:HRV-PPP-proof1}
\end{align}
In view of Theorem \ref{theo:cv-N-Nk} (iii), this implies the convergence~\eqref{thm:PiRV1}.\\
Since $\sum_{i=1}^N\varepsilon_{X_i}$ is a Poisson random measure with intensity $\lambda(\cT)\nu(\rmd x)$, we have
\begin{align}\label{eq1:th31}
&n^{k+1}\bbP\Big(\sum_{i=1}^N\varepsilon_{a_n^{-1}X_i}(B_{0,r}^c)\ge k+1\Big) \nonumber\\
&=n^{k+1}\rme^{-\lambda(\cT)\nu(a_nB_{0,r}^c)}\sum_{j=k+1}^\infty \dfrac{(\lambda(\cT)\nu(a_nB_{0,r}^c))^j}{j!} \nonumber \\
&=\frac{(n\lambda(\cT)\nu(a_nB_{0,r}^c))^{k+1}}{(k+1)!} \rme^{-\lambda(\cT)\nu(a_nB_{0,r}^c)}\sum_{j=0}^\infty \dfrac{(k+1)!}{(k+1+j)!}(\lambda(\cT)\nu(a_nB_{0,r}^c))^j
\end{align}
and 
\begin{align}\label{eq1:th32}
n^{k+1}\bbP\Big(\sum_{i=1}^N\varepsilon_{a_n^{-1}X_i}(B_{0,r}^c)= k+1\Big)=\frac{(n\lambda(\cT)\nu(a_nB_{0,r}^c))^{k+1}}{(k+1)!} \rme^{-\lambda(\cT)\nu(a_nB_{0,r}^c)}\,.
\end{align}
Regular variation $\nu\in\mathrm{RV}(\cX\setminus\{0\},\{a_n\},\mu)$ implies $n\nu(a_nB_{0,r}^c)\to \mu(B_{0,r}^c)>0$  as $n\to\infty$. Then, combining Equations \eqref{eq1:th31} and \eqref{eq1:th32}, we deduce
\begin{align*}
n^{k+1}\bbP\Big(\sum_{i=1}^N\varepsilon_{a_n^{-1}X_i}(B_{0,r}^c)\ge k+1\Big) \sim n^{k+1}\bbP\Big(\sum_{i=1}^N\varepsilon_{a_n^{-1}X_i}(B_{0,r}^c)= k+1\Big)\to\frac{c^{k+1}}{(k+1)!} 
\end{align*}
with $c=\lambda(\cT)\mu(B_{0,r}^c)$. As a consequence, the left-hand side of Equation \eqref{eq:HRV-PPP-proof1} satisfies 
\begin{align}
&n^{k+1}\bbE\Big[ \rme^{-\sum_{i=1}^Nf(T_i,a_n^{-1}X_i)} \mathds{1}_{\{\sum_{i=1}^N\varepsilon_{a_n^{-1}X_i}(B_{0,r}^c)\ge k+1\}}\Big]\nonumber\\
&= n^{k+1}\bbE\Big[ \rme^{-\sum_{i=1}^Nf(T_i,a_n^{-1}X_i)} \mathds{1}_{\{\sum_{i=1}^N\varepsilon_{a_n^{-1}X_i}(B_{0,r}^c)= k+1\}}\Big] \nonumber \\
&\quad \quad +n^{k+1}\bbE\Big[ \rme^{-\sum_{i=1}^Nf(T_i,a_n^{-1}X_i)} \mathds{1}_{\{\sum_{i=1}^N\varepsilon_{a_n^{-1}X_i}(B_{0,r}^c)\ge k+2\}}\Big] \nonumber \\
&= \frac{c^{k+1}}{(k+1)!} \bbE\Big[\rme^{-\sum_{i=1}^Nf(T_i,a_n^{-1}X_i)} \mid \sum_{i=1}^N\varepsilon_{a_n^{-1}X_i}(B_{0,r}^c)=  k+1 \Big](1+o(1))+O\left(\frac{1}{n}\right).\label{eq:HRV-PPP-proof2}
\end{align}
Because $f$ vanishes on $\cT\times B_{0,r}$, the sum $\sum_{i=1}^Nf(t_i,a_n^{-1}X_i)$ depends only on the points of $\Pi$ in $\cT\times a_nB_{0,r}^c$. The independence property of Poisson point processes ensures that, given $\Pi$ has $k+1$ points in  $\cT\times a_n B_{0,r}^c$, those points are independent and uniformly distributed with distribution $c_n^{-1}\mathds{1}_{\{x\in a_n B_{0,r}^c\}}\lambda(\rmd t)\nu(\rmd x)$ for $c_n=\lambda(\cT)\nu(a_n B_{0,r}^c)$ (see \cite{K93}, Chapter 2.4). We deduce 
\begin{align*}
&\bbE\Big[\rme^{-\sum_{i=1}^Nf(T_i,a_n^{-1}X_i)} \mid \sum_{i=1}^N\varepsilon_{a_n^{-1}X_i}(B_{0,r}^c)=  k+1 \Big] \\
&=\int_{E^{k+1}}\rme^{-\sum_{i=1}^{k+1}f(t_i, a_n^{-1}x_i)}\otimes_{i=1}^{k+1} c_n^{-1}\mathds{1}_{\{x_i\in a_n B_{0,r}^c\}}\lambda(\rmd t_i)\nu(\rmd x_i).
\end{align*}
Introduce $\nu_n(\cdot)=n\nu(a_n\cdot)$. 
Since,  $c_n\sim cn^{-1}$, from Lemma \ref{lem:cv_product}, we deduce
\begin{align*}
&\bbE\Big[\rme^{-\sum_{i=1}^Nf(t_i,a_n^{-1}X_i)} \mid \sum_{i=1}^N\varepsilon_{a_n^{-1}X_i}(B_{0,r}^c)=  k+1 \Big] \\
&=\frac{1}{n^{k+1}c_n^{k+1}}\int_{E^{k+1}}\rme^{-\sum_{i=1}^{k+1}f(t_i, x_i)}\otimes_{i=1}^{k+1}\mathds{1}_{\{x_i\in B_{0,r}^c\}} \lambda(\rmd t_i)\nu_n(\rmd x_i)\\
&\to \frac{1}{c^{k+1}}\int_{E^{k+1}}\rme^{-\sum_{i=1}^{k+1}f(t_i, x_i)}\otimes_{i=1}^{k+1} \mathds{1}_{\{x_i\in  B_{0,r}^c\}}\lambda(\rmd t_i)\mu(\rmd x_i).\\ 
\end{align*}
This together with Equation \eqref{eq:HRV-PPP-proof2} implies Equation \eqref{eq:HRV-PPP-proof1}, completing the proof of Theorem~\ref{thm:RV-PP1} in the case $k\geq 1$.
\end{proof}

\subsection{Proof of Theorem~\ref{thm:RV-PP2}}
We begin with three preliminary lemmas that will be useful for the proof of Theorem~\ref{thm:RV-PP2}.

\begin{lemma}\label{lem:binomial}
For $n\geq 1$ and $p\geq 0$, let $S$ be a random variable with binomial distribution with parameter $(n,p)$. Then, for $k=1,\ldots,n$,
\[
\bbP(S\geq k) \leq {n\choose k} p^k.
\] 
\end{lemma}
\begin{proof}
Let $X_1,\ldots,X_n$ be independent Bernoulli random variables with parameter $p$ and for $j \geq 1$, denote by $\mathcal{P}_k(n)$ the set of subsets  $I\subset\{1,\ldots, n\}$ with exactly $k$ elements. The sum $S=\sum_{i=1}^n X_i$ has a binomial distribution with parameter $(n,p)$. Furthermore, $S\geq k$ if and only if there exists $I\in\mathcal{P}_k(n)$ such  that $X_i=1$ whenever $i\in I$. Using sub-additivity and independence, for any $1 \leq k \leq n$ and $n \geq 1$ we obtain 
\[
\bbP(S\geq k)= \bbP\left(\bigcup_{I\in\mathcal{P}_k(n)}  \bigcap_{i \in I}\left\lbrace X_i=1 \right\rbrace \right) \leq \sum_{I\in\mathcal{P}_k(n)} \prod_{i \in I} \bbP(X_i=1) = {n\choose k}  p^k.
\]
\end{proof}

\begin{lemma}\label{lem:LntoL}
Assume $\nu\in\mathrm{RV}(\cX\setminus\{0\},\{a_n\},\mu)$ and let $f:\cT\times\cX\to[0,\infty)$ be a bounded continuous function with support bounded away from the axis $\cT\times\{0\}$.  Then, as $n\to\infty$,
	 \begin{equation}\label{eq:conv_Lnf}
	 (L_nf)(t):=n\bbE[1-e^{-f(t,X_i/a_n)}]\rightarrow \int_{\cX}\left(1-e^{-f(t,x)}\right)\mu(\rmd x)=:(Lf)(t) \quad \mbox{for all $t\in\cT$}
	 \end{equation}
	 and furthermore the functions $L_nf$, $n\geq 1$, are uniformly bounded on $\cT$.
\end{lemma}
\begin{proof}
For all fixed $t\in\cT$, the function $x\in\cX\mapsto 1-e^{-f(t,x)}$ is bounded continuous and vanishes on a neighborhood of $0$ so that the convergence $n\nu(a_n\cdot)\to\mu(\cdot)$ in $\bbM(\cX\setminus\{0\})$ implies Equation \eqref{eq:conv_Lnf}. Furthermore, since $f$ is nonnegative bounded and with support bounded away from $\cT\times\{0\}$, we have 
\[
0\leq f(t,x)\leq M\mathds{1}_{\{d_\cX(0,x)\geq r\}}\quad \mbox{for some $M\geq 0$ and $r>0$},
\]
whence 
\[
(L_n f)(t)\leq (1-e^{-M})n\nu(a_nB_{0,r}^c)\to (1-e^{-M})\mu(B_{0,r}^c)
\]
is uniformly bounded  on $\cT$.
\end{proof}

\begin{lemma}\label{lem:Laplace_derivative}
Assume $\Psi=\sum_{i=1}^N\varepsilon_{T_i}$ has a finite intensity measure $\lambda\in\bbM_b(\cT)$ and define its Laplace functional as
\[
L_\Psi(g)= \bbE[e^{-\Psi(g)}]=\bbE\left[\prod_{i=1}^N e^{-g(T_i)}\right],\quad \mbox{for  $g:\mathcal{T}\to [0,\infty)$   measurable}.
\]
Let $(g_n)_{n\geq 1}$ and $g$ be nonnegative bounded functions on $\cT$ such that 
\[
\lim_{n\to\infty}g_n(t)= g(t),\quad \mbox{for all $t\in\cT$},
\]
and assume $(g_n)_{n\geq 1}$ is uniformly bounded on $\cT$. Then, as $n\to \infty$,
\[
n\left[1-L_\Psi(n^{-1}g_n)\right]\longrightarrow\int_{\cT} g(t)\lambda(\rmd t).
\]
\end{lemma}
\begin{proof} By definition of the Laplace functional $L_\Psi$, we have 
\[
n\left[1-L_\Psi(n^{-1}g_n)\right]=\bbE\left[n\left(1-e^{-n^{-1}\Psi(g_n)}\right)\right]
\]
where
\[
n\left(1-e^{-n^{-1}\Psi(g_n)}\right)\longrightarrow \Psi(g), \quad \mbox{almost surely as $n\to\infty$.}
\]
Furthermore, the inequality $1-e^{-x}\leq x$, $x\geq 0$, together with the uniform bound  $g_n(t)\leq M$, $n\geq 1$, imply the domination condition
\[
n\left(1-e^{-n^{-1}\Psi(g_n)}\right)\leq \Psi(g_n) \leq M \Psi(1),
\]
where $\Psi(1)=N$ satisfies $\mathbb{E}[N]=\lambda(\cT)<\infty$.
We deduce, thanks to Lebesgue convergence Theorem,
\[
n\left[1-L_\Psi(n^{-1}g_n)\right]\longrightarrow \bbE[\Psi(g)].
\]
Campbell's Theorem  (see Chapter 3.2 in \cite{K93}) states the equality $\bbE[\Psi(g)]=\int_\cT g(t)\lambda(\rmd t)$ which concludes the proof.
\end{proof}

\begin{proof}[Proof of Theorem~\ref{thm:RV-PP2}, case $k=0$]
	Similarly as in the proof of Theorem~\ref{thm:RV-PP1}, we need to prove that 
	\begin{equation}\label{eq:conv_laplace}
 n \bbE \left[ 1-\rme^{-\int_{E} f(a_n^{-1}z)\Pi(\rmd z)}\right]  \rightarrow \int_{E}\left(1-\rme^{-f(t,x)}\right) \lambda(\rmd t) \mu(\rmd x),\quad \mbox{as $n\to\infty$},
	 \end{equation}
	 for all bounded continuous functions $f:E\to[0,\infty)$ with support bounded away from $F$. Conditioning with respect to the base point process $\Psi$, the left-hand side of Equation \eqref{eq:conv_laplace} is rewritten as
\begin{align*}
 n \bbE \left[ 1-\exp\left(-\int_{E} f(a_n^{-1}z)\Pi(\rmd z)\right)\right]
&= n \bbE\left[\bbE \left[ 1-\exp\left(-\sum_{i=1}^N f(T_i,X_i)\right)\Big| \Psi\right]\right]\\
&=n\left\{1-\bbE\left[\bbE\left[\prod_{i=1}^N e^{-f(T_i,X_i)}  \Big| \Psi\right]\right]\right\}\\
 &= n\left\{1-\bbE\left[\prod_{i=1}^N \left(1-\frac{1}{n}(L_nf)(T_i) \right) \right]\right\}.
\end{align*}
In the last equality, we use the conditional independence of the $X_i$ given $\Psi$ and the definition of $L_nf$ in Equation~\eqref{eq:conv_Lnf}.
The Laplace functional of the base point process $\Psi$ defined in Lemma~\ref{lem:Laplace_derivative} satisfies, for $h:\cT\to (0,1]$ measurable, 
\[
\bbE\left[\prod_{i=1}^N h(T_i) \right]=\bbE\left[\exp\left(-\sum_{i=1}^N -\ln h(T_i) \right)\right]
=L_\Psi(-\ln h).
\]
This gives
\begin{equation}\label{eq:proof-laplaca-1}
n \bbE \left[ 1-\exp\left(-\int_{E} f(a_n^{-1}z)\Pi(\rmd z)\right)\right]
=n\left[1-L_\Psi\left(-\ln\left(1-\frac{1}{n}L_nf\right) \right) \right].
\end{equation}
We define 
\[
g_n(t)=-n\log \left(1-\frac{1}{n}(L_nf)(t)\right)  \quad \mbox{and}\quad g(t)=(Lf)(t).
\]
Lemma~\ref{lem:LntoL} implies that $g_n(t)\to g(t)$, as $n\to\infty$, and is uniformly bounded on $\cT$. Lemma~\ref{lem:Laplace_derivative} then   
entails
\begin{align*}
&n \bbE \left[ 1-\exp\left(-\int_{E} f(a_n^{-1}z)\Pi(\rmd z)\right)\right]
= n\left[1-L_\Psi(n^{-1} g_n) \right]\\
& \longrightarrow  \int_\cT g(t)\lambda(\rmd t)
=\int_{E}\left(1-e^{-f(t,x)}\right)\lambda(\rmd t)\mu(\rmd x).
\end{align*}
This proves Equation~\eqref{eq:conv_laplace} and Theorem~\ref{thm:RV-PP2} in the case $k=0$.
\end{proof}

\begin{proof}[Proof of Theorem~\ref{thm:RV-PP2}, case $k\geq 1$]
The proof is an adaptation of the proof of Theorem~\ref{thm:RV-PP1}, case $k\geq 1$. We fix $r>0$.  Conditioning upon the events $\{ N=j\}$, $j\geq 1$, and using the independence of the marks, we have
\begin{align}
n^{k+1}\bbP\Big(\sum_{i=1}^N\varepsilon_{a_n^{-1}X_i}(B_{0,r}^c)\ge k+1\Big)&=n^{k+1}\sum_{j=k+1}^\infty \bbP(N=j)\bbP\Big(\sum_{i=1}^j\varepsilon_{a_n^{-1}X_i}(B_{0,r}^c)\ge k+1\Big)\nonumber\\
&=\sum_{j=k+1}^\infty \bbP(N=j)n^{k+1}\bbP(S_{j,n}\geq k+1)\label{eq:HRV-PPP2-proof1}
\end{align}
with $S_{j,n}$ a random variable with binomial distribution with parameter $(j, \nu(a_nB_{0,r}^c))$. Regular variation implies $\nu(a_nB_{0,r}^c)\sim n^{-1}\mu(B_{0,r}^c)$, as $n\to\infty$, so that the binomial probability distribution has asymptotic
\begin{equation}\label{eq:HRV-PPP2-proof2}
n^{k+1}\bbP(S_{j,n}\geq k+1)\sim n^{k+1}\bbP(S_{j,n}= k+1) \sim { j \choose k+1}  \mu(B_{0,r}^c)^{k+1}.
\end{equation}
In order to apply dominated convergence and plug-in the equivalent \eqref{eq:HRV-PPP2-proof2}  in Equation \eqref{eq:HRV-PPP2-proof1}, we use   the upper bound from Lemma~\ref{lem:binomial} yielding
\[
n^{k+1}\bbP(S_{j,n}\geq k+1)\leq { j \choose k+1} n^{k+1} \nu(a_nB_{0,r}^c)^{k+1}.
\]
The convergence $n\nu(a_nB_{0,r}^c)\to \mu(B_{0,r}^c)$ implies the bound $n\nu(a_nB_{0,r}^c)\leq H$ for some $H>0$. We deduce the uniform bound
\[
n^{k+1}\bbP(S_{j,n}\geq k+1)\leq H^{k+1} { j \choose k+1} \leq \frac{H^{k+1}}{(k+1)!} j^{k+1}.
\] 
Since $\Psi$ has a finite factorial moment measure of order $k+1$,  $N$ has a finite moment of order $k+1$ and
\[
\sum_{j= k+1}^\infty \bbP(N=j)\frac{H^{k+1}}{(k+1)!} j^{k+1} \leq \frac{H^{k+1}}{(k+1)!} \bbE[N^{k+1}]<\infty.
\] 
Equations~\eqref{eq:HRV-PPP2-proof1} and~\eqref{eq:HRV-PPP2-proof2} together with dominated convergence imply
\begin{align*}
n^{k+1}\bbP\Big(\sum_{i=1}^N\varepsilon_{a_n^{-1}X_i}(B_{0,r}^c)\ge k+1\Big)
&\sim n^{k+1}\bbP\Big(\sum_{i=1}^N\varepsilon_{a_n^{-1}X_i}(B_{0,r}^c)= k+1\Big) \\
&\to c:=\mu(B_{0,r}^c)^{k+1} \sum_{j\geq k+1}{ j \choose k+1}\bbP(N=j).
\end{align*}
The limit  $c$ is positive as soon as $\bbP(N\geq k+1)>0$ which is ensured by the condition $M_{k+1}$ non-null. We deduce as in the proof of Theorem \ref{thm:RV-PP1}
\begin{align*}
&n^{k+1}\bbE\Big[\rme^{-\sum_{i=1}^Nf(T_i,a_n^{-1}X_i)} \mathds{1}_{\{\sum_{i=1}^N\varepsilon_{a_n^{-1}X_i}(B_{0,r}^c)\ge k+1\}}\Big]\nonumber\\
&\sim n^{k+1}\bbE\Big[\rme^{-\sum_{i=1}^Nf(T_i,a_n^{-1}X_i)} \mathds{1}_{\{\sum_{i=1}^N\varepsilon_{a_n^{-1}X_i}(B_{0,r}^c)= k+1\}}\Big].
\end{align*}
Next,  we condition on $\{N=j\}$ and introduce the symmetric probability $\pi_j$ on $\cT^j$ giving the position of the points of $\Psi$  conditionally on $\{ N=j \}$, see \cite[Section 5.3]{DVJ03}. Denoting $p_j=\bbP(N=j)$, we get
\begin{align*}
  &\bbE\Big[\rme^{-\sum_{i=1}^Nf(T_i,a_n^{-1}X_i)} \mathds{1}_{\{\sum_{i=1}^N\varepsilon_{a_n^{-1}X_i}(B_{0,r}^c)= k+1\}}\Big]\\
  &=\sum_{j\geq k+1}\int_{E^{j}}\rme^{-\sum_{i=1}^j f(t_i,a_n^{-1}x_i)}\mathds{1}_{\{\sum_{i=1}^j\varepsilon_{x_i}(a_nB_{0,r}^c)= k+1\}}p_j \pi_j(\rmd t_1,\ldots,\rmd t_j)\nu(\rmd x_1)\ldots \nu(\rmd x_j).
\end{align*}
Consider the $j$-th term for fixed $j\geq k+1$. The event $\sum_{i=1}^j\varepsilon_{x_i}(a_nB_{0,r}^c)= k+1$ can be decomposed into a disjoint union of ${j\choose k+1}$ disjoint events indexed by a subset  $I\subset\{1,\ldots,j\}$ of size $k+1$ such that $x_i\in a_nB_{0,r}^c$ if $i\in I$ and $x_i\in a_nB_{0,r}$ if $i\notin I$. By symmetry, each event yields an equal contribution and we consider $I=\{1,\ldots,k+1\}$. Then  $f(t_i,a_n^{-1}x_i)=0$ for $i>k+1$ and we can integrate out the corresponding $x_i$'s, yielding the contribution
\begin{align*}
\nu(a_nB_{0,r})^{j-k-1}\int_{\cT^{j}\times(a_nB_{0,r}^c)^{k+1}}\rme^{-\sum_{i=1}^{k+1} f(t_i,a_n^{-1}x_i)} p_j \pi_j(\rmd t_1,\ldots,\rmd t_j) \nu(\rmd x_1)\ldots\nu(\rmd x_{k+1}).
\end{align*}
We then multiply by ${j\choose k+1}$, introduce the $j$th Janossy measure $J_j=j! p_j\pi_j$ of $\Psi$ (see \cite{DVJ03} Chapter 5.3) and sum over $j\geq k+1$ to get
\begin{align*}
  &\bbE\Big[\rme^{-\sum_{i=1}^Nf(T_i,a_n^{-1}X_i)} \mathds{1}_{\{\sum_{i=1}^N\varepsilon_{a_n^{-1}X_i}(B_{0,r}^c)= k+1\}}\Big]\\
 &=\sum_{j\geq k+1} \frac{\nu(a_nB_{0,r})^{j-k-1}}{(k+1)!(j-k-1)!}\cdots\\
 &\quad\cdots \int_{\cT^{j}\times(a_nB_{0,r}^c)^{k+1}}\rme^{-\sum_{i=1}^{k+1} f(t_i,a_n^{-1}x_i)}  J_j(\rmd t_1,\ldots,\rmd t_j) \nu(\rmd x_1)\ldots\nu(\rmd x_{k+1}).
\end{align*}
Note that  $\nu(a_n B_{0,r})\to 1$ and $n\nu(a_n\cdot)\to\mu(\cdot)$ as $n\to\infty$. Then, multiplying by $n^{k+1}$, from  Lemma \ref{lem:cv_product}, we have as $n \to \infty$
\begin{align*}
  &n^{k+1}\bbE\Big[\rme^{-\sum_{i=1}^Nf(T_i,a_n^{-1}X_i)} \mathds{1}_{\{\sum_{i=1}^N\varepsilon_{a_n^{-1}X_i}(B_{0,r}^c)= k+1\}}\Big]\\
 & \longrightarrow \frac{1}{(k+1)!}\sum_{j\geq k+1} \int_{\cT^{j}\times(B_{0,r}^c)^{k+1}}\rme^{-\sum_{i=1}^{k+1} f(t_i,x_i)}  \frac{J_j(\rmd t_1,\ldots,\rmd t_j)}{(j-k-1)!} \mu(\rmd x_1)\ldots\mu(\rmd x_{k+1}) \\
 &= \frac{1}{(k+1)!}\int_{\cT^{k+1}\times(B_{0,r}^c)^{k+1}}\rme^{-\sum_{i=1}^{k+1} f(t_i,x_i)}  M_{k+1}(\rmd t_1,\ldots,\rmd t_{k+1})\mu(\rmd x_1)\ldots\mu(\rmd x_{k+1})\\
&= \int_{\mathcal{N}}\rme^{-\pi(f)}\mathds{1}_{\{\pi(F_r^c )\geq k+1\}}\mu_{k+1}^*(\rmd\pi).
\end{align*}
In view of Theorem~\ref{theo:cv-N-Nk}, the convergence~\eqref{thm:PiRV2} in $\bbM(\cN\setminus\cN_k)$ follows. Note that in the third line above, we use Fubini-Tonelli Theorem to exchange summation and integration as well as  the following identity linking Janossy measures and factorial moment measures (\cite{DVJ03} Theorem 5.4.II)
\[
M_{k+1}(B)=\sum_{m\geq 0} \frac{J_{k+1+m}(B\times \cT^m)}{m!},\quad B\in \mathcal{B}(\cT^{k+1}) .
\]
This concludes the proof.
\end{proof}

\subsection{Proof of Theorem~\ref{thm:RV-PP3}}
The next two lemmas are variants of Lemmas~\ref{lem:LntoL} and \ref{lem:Laplace_derivative} that will be useful for the proof of Theorem~\ref{thm:RV-PP3}.
\begin{lemma}\label{lemunifL}
Under the same notation and assumptions as in Lemma~\ref{lem:LntoL}, assume furthermore that $f:\cT\times\cX\to [0,\infty)$ is Lipschitz continuous and $\cT$ is locally compact. Then $Lf$ is continuous and the convergence $L_nf\to Lf$ in Equation~\eqref{eq:conv_Lnf} is uniform on $\cT$.
\end{lemma}
\begin{proof}
Lemma~\ref{lem:LntoL} states the pointwise convergence $(L_nf)(t)\to (Lf)(t)$ for all $t\in\cT$. Under the assumption that $f$ is $K$-Lipschitz, the family of functions $\{L_nf,n\geq 1\}$ is equicontinuous, since for $t,t' \in \cT$, we have
\begin{align*}
|(L_nf)(t)-(L_nf)(t')|&\leq n\int_{B_{0,r}^c}|e^{-f(t,x/a_n)}-e^{-f(t',x/a_n)}|\nu(\rmd x)\\
&\leq n\nu(a_nB_{0,r}^c)\max(2,Kd_\cT(t,t')),
\end{align*}
where    $n\nu(a_nB_{0,r}^c)\to\mu(B_{0,r}^c)$ is bounded. Pointwise convergence together with equicontinuity implies uniform convergence to a continuous limit thanks to Arzela-Ascoli theorem applied on $\cT$ locally compact.
\end{proof}

\begin{lemma}\label{lem:Laplace_derivative2}
Assume $\Psi_n=m_n^{-1}\sum_{i=1}^{N_n}\varepsilon_{T_i^n}$ satisfies Equation~\eqref{Psi_triangle} and $N_n/m_n$ is uniformly integrable. Define the Laplace functional 
\[
L_{\Psi_n}(g)= \bbE[e^{-\Psi_n(g)}]=\bbE\left[\prod_{i=1}^{N_n} e^{-m_n^{-1}g(T_i)}\right],\quad \mbox{for  $g:\mathcal{T}\to [0,\infty)$   measurable}.
\]
If $g$ is a nonnegative bounded continuous function and $(g_n)_{n\geq 1}$ are nonnegative  functions  such that $g_n\to g$ uniformly on $\cT$, then
\[
n\left[1-L_{\Psi_n}(n^{-1}g_n)\right]\longrightarrow\int_{\cT} g(t)\lambda(\rmd t),\quad \mbox{as $n\to\infty$}.
\]
\end{lemma}
\begin{proof} By definition of the Laplace functional $L_{\Psi_n}$, 
\[
n\left[1-L_{\Psi_n}(n^{-1}g_n)\right]-\mathbb{E}[\Psi_n(g_n)] = \bbE\left[n\left(1-e^{-n^{-1}\Psi_n(g_n)} - n^{-1}\Psi_n(g_n)\right)\right].
\]
Since the functions $g_n$ are uniformly bounded by some constant $M$, $\Psi_n(g_n)$ is  bounded by $MN_n/m_n$ and we deduce
\begin{equation}\label{eq:asymptotic}
n\left[1-L_{\Psi_n}(n^{-1}g_n)\right]=\mathbb{E}[\Psi_n(g_n)]+o(1).
\end{equation}
Let $\varepsilon>0$. The uniform convergence $g_n\to g$ implies that for sufficiently large $n$, $|g_n(x)-g(x)|\leq \varepsilon$ for all $x\in\mathcal{X}$ so that
\begin{equation}\label{eq:sandwich}
\mathbb{E}[\Psi_n(g)]-\varepsilon\leq \mathbb{E}[\Psi_n(g_n)]\leq \mathbb{E}[\Psi_n(g)]+\varepsilon.
\end{equation}
We have used here the property that $\mathbb{E}[\Psi_n(1)]=\mathbb{E}[N_n/m_n]=1$. The function $g$ being continuous and bounded, Equation~\eqref{Psi_triangle} implies the convergence $\Psi_n(g)\to \int g\rmd \lambda$ in probability. Using the upper bound $\Psi_n(g)\leq M N_n/m_n$ and the uniform integrability of $N_n/m_n$, we deduce  $\mathbb{E}[\Psi_n(g)]\to \int g\rmd \lambda$. Thanks to the inequalities \eqref{eq:sandwich}, we deduce $\mathbb{E}[\Psi_n(g_n)]\to \int g\rmd \lambda$. Together with Equation~\eqref{eq:asymptotic}, this concludes the proof of the Lemma.
\end{proof}

\begin{proof}[Proof of Theorem~\ref{thm:RV-PP3}, case $k=0$]
	The proof uses arguments and notation from the proof of Theorem \ref{thm:RV-PP2}. We need to prove that for all bounded Lipschitz continuous function $f:\cT\times\cX\to[0,\infty)$ with support bounded away from the axis $\cT\times\{0\}$,
\begin{equation}\label{eq:tri}
 n \bbE \left[ 1-\rme^{-\int_{E } f(a_{nm_n}^{-1}z)\Pi_n(\rmd z)}\right]  \rightarrow \int_{ \cT\times\cX}\left(1-\rme^{-f(t,x)}\right) \lambda(\rmd t) \mu(\rmd x),\quad \mbox{as $n\to\infty$}.
	 \end{equation}	
On the one hand, similarly as in Equation~\eqref{eq:proof-laplaca-1}, we have
\begin{align*}
	n \bbE \left[ 1-\rme^{-\int_{E} f(a_{nm_n}^{-1}z)\pi(\rmd z)}\right] & = n \left[  1- \mathbb{E}\left[  \prod_{i=1}^{N_n} \left( 1- \frac{1}{nm_n}(L_{nm_n}f)(T_i^n) \right) \right]  \right]  \\ 
	&= n\left[1-L_{\Psi_n}\left(-n^{-1}g_n\right)\right]
\end{align*}
	with 
\[
g_n(t)=-nm_n \log \left(1- \frac{1}{nm_n}(L_{nm_n }f )(t) \right),\quad t\in\cT.
\]
By Lemma~\ref{lemunifL}, $g_n$ converges to $g=Lf$ uniformly on $\cT$ and $g$ is continuous.  Lemma~\ref{lem:Laplace_derivative2} then implies
\[
 n\left[1-L_{\Psi_n}\left(-n^{-1}g_n\right)\right]\to \int_{\cT} (Lf)(t)\lambda(\rmd t)=\int_{ \cT\times\cX}\left(1-\rme^{-f(t,x)}\right) \lambda(\rmd t) \mu(\rmd x).
\]
This shows Equation \eqref{eq:tri} and concludes the proof.
\end{proof}

The next two lemmas are used in the proof of Theorem~\ref{thm:RV-PP3} for the  case $k\geq 1$. We recall that $\nu_n(\cdot)=n\nu(a_n \cdot)$.
\begin{lemma}\label{lemtodo1}
Assume $\cT$ locally compact and let $k\geq 0$. For $f:\cT\times\cX\to [0,\infty)$  Lipschitz continuous vanishing on $\cT\times B_{0,r}$, the sequence of functions $L_{k+1,n} f$  defined on $\cT^{k+1}$ by
$$
(L_{k+1,n} f)(t_1,\ldots,t_{k+1})= \int_{( B_{0,r}^c)^{k+1}}\rme^{-\sum_{1\le i\le k+1}f(t_i,x_i)} \nu_n(dx_{1})\cdots \nu_n(dx_{k+1}),\quad n\geq 1\,,
$$
 converges uniformly  to $L_{k+1}f$  defined by
$$
(L_{k+1}f)(t_1,\ldots,t_{k+1})= \int_{( B_{0,r}^c)^{k+1}}\rme^{-\sum_{1\le i\le k+1}f(t_i,x_i)} \mu(dx_1)\cdots \mu(dx_{k+1})\,.
$$
Furthermore the sequence $L_{k+1,n}f$ is uniformly bounded.\end{lemma}
\begin{proof}
When $k=0$, Lemma~\ref{lemtodo1} reduces to Lemma~\ref{lemunifL}. The proof of the case $k\geq 1$ follows the same lines and is omitted for the sake of brevity.
\end{proof}

\begin{lemma}\label{lemtodo2}
Assume that condition~\eqref{Psi_triangle} holds. Then, for $k\geq 1$, the $k$-th factorial power of $\Psi_n$ defined by
\[
\Psi_n^{(k)}=\sum_{1\leq i_1\neq \ldots \neq i_k\leq N_n} \varepsilon_{(T_{i_1}^n,\ldots,T_{i_k}^n)} 
\]
satisfies
\begin{equation}\label{eq:lem-todo0}
m_n^{-k}\Psi_n^{(k)}\stackrel{\mathbb{M}_b(\cT^k)}\longrightarrow \lambda^{\otimes k} \quad \mbox{in probability as $n\to\infty$.}
\end{equation}
Furthermore, if  $(N_n/m_n)^{k}$ is uniformly integrable, then the normalized factorial moment measure  of order $k$ of $\Psi_n$ denoted by $M_{k}^{\Psi_n}$ satisfies 
\begin{equation}\label{eq:lem-todo}
m_n^{-k}M_{k}^{\Psi_n}
\longrightarrow \lambda^{\otimes k}\quad 
\mbox{in $\mathbb{M}_b(\cT^{k})$ as $n\to\infty$.}
\end{equation}
\end{lemma}
\begin{proof}
Note that
\[
\Psi_n^{(k)}=\Psi_n^{\otimes k}-\sum \varepsilon_{(T_{i_1}^n,\ldots,T_{i_k}^n)} 
\]
where the sum runs over $k$-tuples $1\leq i_1, \ldots ,i_k\leq N_n$ that are not pairwise distinct. There are at most ${ k \choose 2}N_n^{k-1}$ such indices and this is thus a bound for the total  mass of the sum. Normalizing by $m_n^{-k}$ we deduce
\begin{equation}\label{eq:sum}
m_n^{-k}\Psi_n^{(k)}=(m_n^{-1}\Psi_n)^{\otimes k} - \mathrm{remainder}.
\end{equation}
The remainder is a measure with total variation bounded by  ${ k \choose 2}N_n^{k-1}/m_n^k$ which tends to zero in probability because $m_n\to\infty$ and Equation~\eqref{Psi_triangle} implies $N_n/m_n\to 1$ in probability. On the other hand, by continuity of the tensor $k$-th power of measures, Equation~\eqref{Psi_triangle}  implies $(m_n^{-1}\Psi_n)^{\otimes k}\stackrel{\mathbb{M}_b(\cT^k)}\longrightarrow \lambda^{\otimes k}$ in probability. The convergence~\eqref{eq:lem-todo0} then follows from Equation~\eqref{eq:sum}. 

Next we prove the convergence \eqref{eq:lem-todo}. Let $f:\cT^k\to\mathbb{R}$ be bounded continuous. By definition of the  $k$-th factorial moment measure, $M_{k}^{\Psi_n}[f]=\mathbb{E}[\Psi_n^{(k)}(f)]$ 
and Equation~\eqref{eq:lem-todo0} implies that $m_n^{-k}\Psi_n^{k}(f)\to \lambda^{\otimes k}(f)$ in probability. Furthermore, if $f$ is bounded by $M$ in absolute value, then $m_n^{-k}\Psi_n^{(k)}(f)$ is bounded by $M(N_n/m_n)^k$ and is hence uniformly integrable. The convergence of expectations 
\[
m_n^{-k}M_{k}^{\Psi_n}[f]=\mathbb{E}[m_n^{-k}\Psi_n^{(k)}(f)]\to \lambda^{\otimes k}(f)
\]
follows, proving the weak convergence \eqref{eq:lem-todo}.
\end{proof}

\begin{proof}[Proof of Theorem~\ref{thm:RV-PP3}, case $k\geq 1$] 
We fix $r>0$ such that $\mu(\partial B_{0,r}^c)=0$  and a bounded continuous function $f:  \cT\times \cX \mapsto [0,\infty)$ vanishing on $F^r=\cT\times B_{0,r}$. For $k\geq 1$, we need to prove that
\begin{align}
&n^{k+1}\bbE\Big[\rme^{-\sum_{i=1}^{N_n}f(t_i,a_{nm_n}^{-1}X_i)} \mathds{1}_{\{\sum_{i=1}^{N_n}\varepsilon_{a_{nm_n}^{-1}X_i}(B_{0,r}^c)\ge k+1\}}\Big]\nonumber\\
\longrightarrow&\ \frac1{(k+1)!}\int_{E^{k+1}}\rme^{-\sum_{i=1}^{k+1}f(t_i, x_i)}\otimes_{i=1}^{k+1}\mathds{1}_{\{x_i\in  B_{0,r}^c\}} \lambda(\rmd t_i)\mu(\rmd x_i) \nonumber\\
=&\int_{\mathcal{N}}e^{-\pi(f)} \mathds{1}_{\{\pi(B_{0,r}^c)\geq k+1\}}\mu^*_{k+1}(\rmd \pi)   \,.\label{eq:toprove0}
\end{align}
The proof follows the lines of the proof of Theorem~\ref{thm:RV-PP2}. We also use the fact that Equation~\eqref{Psi_triangle} implies the convergence $N_n/m_n\to 1$ in probability and hence the convergence  $N_n/m_n\to 1$ in $\mathbb{L}^{k+1}$ because of  uniform integrability.

Due to independent marking, we have 
\begin{align*}
n^{k+1} \mathbb{P} \left( \sum_{i=1}^{N_n} \varepsilon_{a_{nm_n}^{-1}X_i }( B_{0,r}^c) \geq k+1\right) = n^{k+1} \mathbb{P} \left( S_{N_n} \geq k+1\right)
\end{align*}
where, conditionally on $N_n$, $S_{N_n}$ has a binomial distribution with parameter $(N_n,\nu(a_{nm_n}B_{0,r}^c))$. We have the basic decomposition
\[
n^{k+1}\bbP(S_{N_n}\geq k+1)= n^{k+1}\bbP(S_{N_n}= k+1)+n^{k+1}\bbP(S_{N_n}\geq k+2).
\]
The second term is controlled by Lemma~\ref{lem:binomial} and the uniform integrability of $(N_n/m_n)^{k+1}$ as follows: 
\begin{align*}
n^{k+1}\bbP(S_{N_n}\geq k+2)&\leq n^{k+1}\mathbb{E}\left[\min\left(\nu(a_{nm_n}B_{0,r}^c)^{k+1}{ N_n \choose k+1},\nu(a_{nm_n}B_{0,r}^c)^{k+2}{ N_n \choose k+2} \right)\right]\\
&\leq \frac{1}{(k+1)!} \mathbb{E}\left[\min\left(n^{k+1}\nu(a_{nm_n}B_{0,r}^c)^{k+1}N_n^{k+1},n^{k+1}\nu(a_{nm_n}B_{0,r}^c)^{k+2}N_n^{k+2}\right)\right]\\
&=o(1).
\end{align*}
The first term has a non vanishing limit: using $\nu(a_{nm_n}B_{0,r}^c) \sim (nm_n)^{-1}\mu(B_{0,r}^c)$, the convergence $N_n/m_n\to 1$  in  $\mathbb{L}^{k+1}$, the convergence  $\nu(a_{nm_n}B_{0,r})^{N_n-k-1}\to 1$ in probability with uniform bound $1$, we deduce 
\begin{align*}
n^{k+1}\bbP(S_{N_n}= k+1)&=n^{k+1} \nu(a_{nm_n}B_{0,r}^c)^{k+1} \mathbb{E}\left[\nu(a_{nm_n}B_{0,r})^{N_n-k-1}{ N_n \choose k+1}\right]\\
&\longrightarrow  \frac{ \mu(B_{0,r}^c)^{k+1}}{(k+1)!}\quad \mbox{as $n\to\infty$}.
\end{align*}
We deduce that the left hand side of Equation~\eqref{eq:toprove0} satisfies
\begin{align*}
&n^{k+1}\bbE\Big[\rme^{-\sum_{i=1}^{N_n}f(T_i^n,a_n^{-1}X_i)} \mathds{1}_{\{\sum_{i=1}^{N_n}\varepsilon_{a_{nm_n}^{-1}X_i}(B_{0,r}^c)\ge k+1\}}\Big]\nonumber\\
&= n^{k+1}\bbE\Big[\rme^{-\sum_{i=1}^{N_n}f(T_i^n,a_{nm_n}^{-1}X_i)} \mathds{1}_{\{\sum_{i=1}^{N_n}\varepsilon_{a_{nm_n}^{-1}X_i}(B_{0,r}^c)= k+1\}}\Big]+o(1).
\end{align*}
The event $\left\lbrace \sum_{i=1}^{N_n}\varepsilon_{a_{nm_n}^{-1}X_i}(B_{0,r}^c)= k+1 \right\rbrace$ can be decomposed into an union of ${ N_n \choose k+1}$ disjoints events indexed by a subset $I\subset \{ 1,\ldots,N_n\}$ of size $k+1$ such that $x_i\in a_{nm_n}B_{0,r}^c$ if $i\in I$ and $x_i\in a_{nm_n}B_{0,r}$ if $i\notin I$. We obtain
\begin{align*}
& n^{k+1}\bbE\Big[\rme^{-\sum_{i=1}^{N_n}f(T_i^n,a_{nm_n}^{-1}X_i)} \mathds{1}_{\{\sum_{i=1}^{N_n}\varepsilon_{a_{nm_n}^{-1}X_i}(B_{0,r}^c)= k+1\}}\Big]\\
& = n^{k+1}\bbE\Big[\nu(a_{nm_n}B_{0,r})^{N_n-k-1}\sum_{I\subset \{ 1,\ldots,N_n\}}\rme^{-\sum_{i\in I}f(T_i^n,a_{nm_n}^{-1}X_i)} \mathds{1}_{\{X_i\in a_{nm_n}B_{0,r}^c,\,i\in I\}}\Big]\,.
\end{align*}
Conditionally on $N_n$ and $(T_i^n)_{1\le i\le N_n}$, we have
\begin{align*}
n^{k+1}\bbE&\Big[\rme^{-\sum_{i\in I}f(T_i^n,a_{nm_n}^{-1}X_i)} \mathds{1}_{\{X_i\in a_{nm_n}B_{0,r}^c,\,i\in I\}}\mid N_n,(T_i^n)_{1\le i\le N_n},I\subset\{1,\ldots,N_n\}\Big]\\
&=n^{k+1}\int_{( a_{nm_n}B_{0,r}^c)^{k+1}}\rme^{-\sum_{i\in I}f(T_i^n,a_{nm_n}^{-1}x_i)} \nu(dx_{i_1})\cdots \nu(dx_{i_{k+1}})\\
&= m_n^{-(k+1)}\int_{(  B_{0,r}^c)^{k+1}}\rme^{-\sum_{i\in I}f(T_i^n, x_i)} \nu_{nm_n}(dx_{i_1})\cdots \nu_{nm_n}(dx_{i_{k+1}})\\
&= m_n^{-(k+1)}(L_{k+1,nm_n}f)(T_{i_1}^n,\ldots,T_{i_k}^n)\;,
\end{align*}
with $L_{k+1,nm_n}f$ defined in Lemma~\ref{lemtodo1} and the notation $I=\{i_1,\ldots,i_k\}$.  We obtain 
\begin{align}
n^{k+1}\bbE&\Big[\rme^{-\sum_{i=1}^{N_n}f(T_i^n,a_{nm_n}^{-1}X_i)} \mathds{1}_{\{\sum_{i=1}^{N_n}\varepsilon_{a_{nm_n}^{-1}X_i}(B_{0,r}^c)= k+1\}}\mid N_n,(T_i^n)_{1\le i\le N_n}\Big]\nonumber\\
&= \frac{\nu(a_{nm_n}B_{0,r})^{N_n-k-1}}{(k+1)!m_n^{k+1}}\sum_{1\leq i_1\neq\ldots\neq i_{k+1}\leq N_{n}}(L_{k+1,nm_n}f)(T_{i_1}^n,\ldots,T_{i_k}^n)\nonumber\\
&=\frac{\nu(a_{nm_n}B_{0,r})^{N_n-k-1}}{(k+1)!}\int (L_{k+1,nm_n}f)(t)m_n^{-(k+1)}{\Psi_n^{(k+1)}}(dt)\,.\label{eq:proof00}
\end{align}
We have already seen that $\nu(a_{nm_n}B_{0,r})^{N_n-k-1}\to 1$ in probability and . On the other hand, Lemma \ref{lemtodo2} states that  $m_n^{-(k+1)}\Psi_n^{(k+1)}\stackrel{\mathbb{M}_b(\cT^{k+1})}\to \lambda^{\otimes (k+1)}$ in probability and  Lemma \ref{lemtodo1} states that  $L_{k+1,nm_n}f\to L_{k+1}f$ uniformly on $\cT^{k+1}$ with $L_{k+1}f$ continuous.
We deduce that the right hand side of Equation~\eqref{eq:proof00} converges in probability to $\frac{1}{(k+1)!}\int (L_{k+1}f)(t)\lambda^{\otimes(k+1)}(dt)$. The  uniformly integrable upper bound of order $(N_n/m_n)^{k+1}$ implies the convergence of the expectations, that is
\begin{align*}
n^{k+1}\bbE&\Big[\rme^{-\sum_{i=1}^{N_n}f(T_i^n,a_{nm_n}^{-1}X_i)} \mathds{1}_{\{\sum_{i=1}^{N_n}\varepsilon_{a_{nm_n}^{-1}X_i}(B_{0,r}^c)= k+1\}}\Big]\\
&\to \frac{1}{(k+1)!}\int (L_{k+1}f)(t)\lambda^{\otimes(k+1)}(dt)\\
&= \frac1{(k+1)!}\int_{E^{k+1}}\rme^{-\sum_{i=1}^{k+1}f(t_i, x_i)}\otimes_{i=1}^{k+1}\mathds{1}_{\{x_i\in  B_{0,r}^c\}} \lambda(\rmd t_i)\mu(\rmd x_i).
\end{align*}
This proves Equation~\eqref{eq:toprove0} and concludes the proof.
\end{proof}

\section{Proofs related to Section~\ref{sec:app}}\label{sec:proof3}
\subsection{Proofs related to Section~\ref{sec:app1}}
The following lemma gives an explicit expression for the distance in the Skorokhod space $\bbD=\bbD([0,T],\mathbb{R})$ to the cone $\bbD_k$ of functions with at most $k$ discontinuity points. We first introduce some notation. A c\`ad-l\`ag function $x\in\bbD$ has at most countably many discontinuity points $(t_i)_{i\in I}$ with size $|x(t_i)-x(t_i^-)|$ and for every $\varepsilon>0$, the number of jumps with size larger than $\varepsilon$ is finite - for a discontinuity point $t_i$, $i\in I$, the notation $x(t_i^-)$ stands for the left limit of $x$ at $t_i$. Reordering the sequence of jump sizes $|x(t_i)-x(t_i^-)|$, $i\in I$, we define the non-negative sequence 
\[
\Delta_1(x)\geq \Delta_2 (x)\geq \Delta_3 (x) \geq \cdots
\]
That is $\Delta_1 (x)$ is the largest jump (in absolute value), $\Delta_2 (x)$ the second largest jump, etc. If $x$ has finitely many jumps, say $k\geq 0$, we set $\Delta_m (x)=0$ for $m\geq k+1$. It follows that $x\in \bbD_k$ if and only if $\Delta_{k+1} (x)=0$. 
\begin{lemma}\label{lem:dist-Dk}
For all $k\geq 0$ and $x\in\bbD$, $d(x,\bbD_k)=\frac{1}{2}\Delta_{k+1} (x)$.
\end{lemma}
\begin{proof}
We first consider the case $k=0$ when $\bbD_0$ is the space of continuous functions. For $x\in\bbD$ and $y\in\mathbb{D}_0$, we prove  that $d(x,y)\geq \frac{1}{2}\Delta_1(x)$. We  observe that
\[
d(x,\bbD_0)=\inf_{y\in\bbD_0}d(x,y)=\inf_{y\in\bbD_0}\|x-y\|_\infty
\]
with $\|\cdot\|_\infty$ the uniform norm for bounded functions on $[0,T]$. To see this, we recall that the Skorokhod distance is defined by
\[
d(x,y)=\inf_{\lambda \in\Lambda} \max(\|\lambda - \mathrm{Id}\|_\infty,\|x-y\circ \lambda\|_\infty)
\]
where $\Lambda$ denotes the set of increasing bi-continuous bijections $\lambda:[0,T]\to [0,T]$ and $\mathrm{Id}$ the identity function on $[0,T]$. Since $y\mapsto y\circ \lambda$ is a bijection on $\bbD_0$, we have
\begin{align*}
d(x,\bbD_0)&=\inf_{y\in\bbD_0}\inf_{\lambda \in\Lambda}\max(\|\lambda - \mathrm{Id}\|_\infty,\|x-y\circ \lambda\|_\infty)\\
&=\inf_{y\in\bbD_0}\inf_{\lambda \in\Lambda}\max(\|\lambda - \mathrm{Id}\|_\infty,\|x-y\|_\infty)=\inf_{y\in\bbD_0}\|x-y\|_\infty.
\end{align*}
For  $t_1\in[0,T]$ such that $\Delta_1(x)=|x(t_1)-x(t_1^-)|$ and $y\in\bbD_0$, we have
\[
\|x-y\|_\infty\geq \max(|x(t_1)-y(t_1)|,|x(t_1^-)-y(t_1^-)|)\geq \frac{1}{2}\Delta_1(x).
\] 
Taking the infimum over $y\in\bbD_0$, we deduce $d(x,\bbD_0)\geq \frac{1}{2}\Delta_1(x)$. 

For the reverse inequality, it is enough to exhibit a sequence of continuous functions $y_n$ such that 
\begin{equation}\label{eq:dist-DK_conv}
d(x,y_n)=\|x-y_n\|_\infty \to \frac{1}{2}\Delta_1(x).
\end{equation}
A simple construction is \textit{via}   convolution: define $y_n= x\ast f_n$ where $\ast$ is the convolution operator and $f_n$ is the density of the uniform distribution on $[-1/n,1/n]$, first extending the definition of $x$  by letting $x(t)=x(0)$ for $t\leq 0$ and $x(t)=x(T)$ for $t\geq T$. Then, Equation~\eqref{eq:dist-DK_conv} is satisfied, proving $d(x,\bbD_0)\leq \frac{1}{2}\Delta_1(x)$.

We next consider the case $k\geq 1$. Since the mapping $y\in\bbD_k\mapsto y\circ\lambda\in\bbD_k$ is bijective, we have
\[
d(x,\bbD_k)=\inf_{y\in\bbD_k}d(x,y)=\inf_{y\in\bbD_k}\|x-y\|_\infty
\]
Any function $y\in\bbD_k$ can be decomposed as $y=j+c$ where $j$ is a pure jump function with at most $k$ jumps and $c$ is a continuous function. It follows
\[
\|x-y\|_\infty=\|x-j-c\|_\infty= \frac{1}{2}\Delta_{1}(x-j),
\]
where the last equality relies on the case $k=0$. Since $j$ has at most $k$ jumps, the functions $x$ and $x-j$ share the same discontinuity points except at most $k$ of them. This implies $\Delta_{1}(x-j)\geq \Delta_{k+1}(x)$ with equality if $j$ kills the $k$ largest jumps of $x$. Hence 
\[
\|x-y\|_\infty= \frac{1}{2}\Delta_{1}(x-j)\geq \frac{1}{2}\Delta_{k+1}(x).
\]
Taking the infimum for $y\in \bbD_k$, we get $d(x,\bbD_k)\geq \frac{1}{2}\Delta_{k+1}(x)$. The reverse inequality is proven taking $y=c+j$ with $j$  killing exactly the $k$ largest jumps of $x$ so that
\[
d(x,\bbD_k)\leq \|x-y\|_\infty= \frac{1}{2}\Delta_{1}(x-j)=\frac{1}{2}\Delta_{k+1}(x).
\]
\end{proof}

\begin{proof}[Proof of Theorem~\ref{theo:RV-D}]  
Consider the rescaled risk process
\[
R_n^0(t)= \sum_{i=1}^N a_n^{-1}X_i \mathds{1}_{\{T_i\leq t\}}, \ \ \ t\in [0,T],
\]
and, for $\delta>0$, the  truncated rescaled risk process 
\begin{align}\label{truncatedprocess}
R_n^\delta(t)= \sum_{i=1}^N a_n^{-1}X_i \mathds{1}_{\{a_n^{-1}X_i >  \delta\}}\mathds{1}_{\{T_i\leq t\}}, \ \ \ t\in [0,T].
\end{align}
The proof involves the following three steps, corresponding to conditions $i)-iii)$ of Proposition~\ref{prop:HRV-criterion}: 
\begin{enumerate}
\item[1)] Using the  continuous mapping theorem \citep[Theorem 2.3]{LRR14}, we show that $n^{k+1}\bbP(R_n^\delta \in \cdot)\longrightarrow  \hat \mu_{k+1}^\delta(\cdot)$ in $\bbM(\mathbb{D}\setminus \mathbb{D}_k)$, with  limit measure  
\[
\mu^{\# \delta}_{k+1}(B)=\int_{E^{k+1}} \mathds{1}_{\left\{\sum_{i=1}^{k+1}z_{(t_i,x_i)}\in B\right\}} M_{k+1}(\rmd t_1,\ldots,\rmd t_{k+1})\otimes_{i=1}^{k+1}\mathds{1}_{\{x_i>\delta\}}\mu(\rmd x_i),
\]
where $z_{(t,x)}=(x\mathds{1}_{\{t \leq u\}})_{0\leq u\leq T}\in\mathbb{D}$.
\item[2)]  We prove that   $ \mu_{k+1}^{\#\delta}  \longrightarrow  \mu_{k+1}^{\#}$ in $\mathbb{M}(\mathbb{D}\setminus \mathbb{D}_k)$ as $\delta \to 0$.
\item[3)] We prove that $R_n^\delta$ and $R_n^0$ satisfy, for any $\varepsilon,r>0$, 
\begin{equation}\label{tail_cond}
\lim_{\delta\to 0}\limsup_{n\to \infty}n^{k+1}\bbP(d(R_n^0,R_n^\delta)>\varepsilon,d(R_n^0,\mathbb{D}_k)>r) = 0
\end{equation}
where $d$ denotes the Skorokhod metric on $\mathbb{D}$.
\end{enumerate} 
\noindent
Both conditions iii) of Proposition \ref{prop:HRV-criterion}  hold under the only condition\eqref{tail_cond} because $d(R_n^\delta,\mathbb{D}_k)\leq d(R_n^0,\mathbb{D}_k)$. Then the result $n^{k+1}\bbP(a_n^{-1}R \in \cdot)\longrightarrow   \mu_{k+1}^{\#}(\cdot)$ in $\bbM(\mathbb{D}\setminus \mathbb{D}_k)$ follows from Proposition \ref{prop:HRV-criterion} with $E=\mathbb{D}$, $F=\mathbb{D}_k$, $X=R=a_nR_n^0$, $X_{n,m}=a_nR_n^\delta$ and $m=[1/\delta]$.\newline

\textbf{Step 1}. For $\delta>0$, we have $R_n^\delta = T_\delta(a_n^{-1}\Pi)$ with  $T_\delta : \cN \longrightarrow \mathbb{D}$ the measurable mapping defined by 
$$T_\delta: \pi=\sum_{i\in I} \varepsilon_{(t_i,x_i)} \longmapsto \sum_{i\in I} \mathds{1}_{\{x_i >\delta\}} z_{(t_i,x_i)}.$$
Note that $I$ is  countable and that there are only finitely many points $x_i> \delta$.
Theorem \ref{thm:RV-PP2} together with the continuous mapping theorem \citep[Theorem 2.3]{LRR14} imply
$$n^{k+1}\mathbb{P}(R_n^\delta  \in \cdot)=n^{k+1} \mathbb{P}(T_\delta(a_n^{-1} \Pi) \in \cdot) \overset{\bbM(\mathbb{D}\setminus\mathbb{D}_k)}\longrightarrow \mu_{k+1}^* \circ T_\delta^{-1}(\cdot)=\mu_{k+1}^{\# \delta}(\cdot). $$
It remains to check that the conditions for the continuous mapping theorem are satisfied. First, note that $T_\delta^{-1}(\mathbb{D}_k) \subset \cN_k$ and that $T_\delta^{-1}(B)$ is bounded away from $\cN_k$ for all $B \in \mathcal{B}(\mathbb{D})$ bounded away from $\mathbb{D}_k$. Besides,  $T_\delta$ is continuous at every point $\pi$ such that $\pi([0,T] \times \{ \delta \})=0$, which can be proved with similar arguments as in the proof of Lemma 3.2 in \cite{EMD16}. It is easily seen that $\mu_{k+1}^*$ has no mass on $\{\pi([0,T] \times \{ \delta \})\neq 0\}$ so that the discontinuity set of $T_\delta$ has vanishing $\mu_{k+1}^*$-measure.

\textbf{Step 2}. It is a straightforward application of the monotone convergence Theorem since the indicator function $\mathds{1}_{\{x_i>\delta\}}$ converges monotonically to $\mathds{1}_{\{x_i>0\}}$ as $\delta\downarrow 0$.

\textbf{Step 3}. The Skorokhod distance between the risk process $R_n^0$ and its truncated version $R_n^\delta$ is upper bounded by
$$d(R_n^\delta,R_n^0) \leq \Vert R_n^0- R_n^\delta \Vert_{\infty}=\sum_{i=1}^N a_n^{-1}X_i \mathds{1}_{\{a_n^{-1}X_i \leq \delta \} }\le \delta N.$$ 
On the other hand, Lemma~\ref{lem:dist-Dk} implies 
$$
\{d(R_n^0,\mathbb{D}_k)>r\} = \{\Delta_{k+1}(R_n^0)>2r\} = \{(a_n^{-1}\Pi)([0,T]\times (2r,\infty))\ge k+1\}\,.
$$
We deduce 
$$
 \bbP\Big(d(R_n^\delta,R_n^0)>\varepsilon,d(R_n^0,\mathbb{D}_k)>r\Big)
 \le  \bbP\Big( \delta N>\varepsilon, (a_n^{-1}\Pi)([0,T]\times (2r,\infty))\ge k+1\Big)\,.
$$
Denote $S_{j,n}$ a random variable with binomial distribution with parameter $(j, \nu(2 a_n r,\infty))$. Conditioning on $N=j$ and applying Lemma \ref{lem:binomial},  we get the upper bound, for $\delta>0$ small enough so that $\varepsilon/\delta>k+1$,
$$
\sum_{j=\varepsilon/\delta}^\infty\bbP(S_{j,n}\ge k+1)\bbP(N=j)\le\sum_{j=\varepsilon/\delta}^\infty{j \choose k+1}\nu(2a_n r,\infty)^{k+1}\bbP(N=j).
$$
The convergence $\nu_n(2r,\infty)\to \mu(2r,\infty)$ implies the bound $\nu_n(2r,\infty)\leq H$ for some $H>0$, whence the uniform bound
\begin{align*}
n^{k+1}\bbP\Big(d(R_n^\delta,R_n^0)>\varepsilon,d(R_n^0,\mathbb{D}_k)>r\Big)
&\le \sum_{j=\varepsilon/\delta}^\infty{j \choose k+1} j^{k+1}\nu_n(2r,\infty)^{k+1}\bbP(N=j) \\
&\le \sum_{j=\varepsilon/\delta}^\infty \bbP(N=j)\frac{H^{k+1}}{(k+1)!} j^{k+1}\\
& \leq \frac{H^{k+1}}{(k+1)!} \bbE[N^{k+1}\mathds{1}_{\{N\ge \varepsilon/\delta\}}].
\end{align*}
The assumption that $\Psi$ has a finite factorial measure $M_{k+1}$ implies $\bbE[N^{k+1}]<\infty$ so that the right-hand side converges to $0$ as $\delta\to 0$, proving Equation \eqref{tail_cond}.
\end{proof}

\begin{proof}[Proof of Theorem~\ref{theo:RV-D2}]  
The proof is very similar to the proof of Theorem~\ref{theo:RV-D} with $R_n^0$ and $R_n^\delta$ replaced respectively by 
\[
\tilde R_n^0(t)= \sum_{i=1}^{m_n} a_{nm_n}^{-1}(X_i-c) \mathds{1}_{\{T_i^n\leq t\}}, \ \ \ t\in [0,T],
\]
and
\[
R_n^\delta(t)= \sum_{i=1}^{m_n} a_{nm_n}^{-1}X_i \mathds{1}_{\{a_{nm_n}^{-1}X_i >  \delta\}}\mathds{1}_{\{T_i^n\leq t\}}, \ \ \ t\in [0,T].
\]
Steps 1) and 2) are proved exactly in the same way but the proof of Step 3)  is more involved in the case of a triangular array. We now prove 
\begin{equation}\label{tail_cond2}
\lim_{\delta\to 0}\limsup_{n\to \infty}n^{k+1}\bbP(d(\tilde R_n^0,R_n^\delta)>\varepsilon,d(\tilde R_n^0,\mathbb{D}_k)>r) = 0.
\end{equation}
Recall that $d(\tilde R_n^0,\mathbb{D}_k)=\frac{1}{2}\Delta_{k+1}(\tilde R_n^0) $ and note that $ \Delta_{k+1}(\tilde R_n^0)\le  \Delta_{k+1}( R_n^0)\vee (a_{nm_n}^{-1}c)$.  Moreover, the Skorokhod distance between the risk process $\tilde R_n^0$ and $R_n^\delta$ is upper bounded by
\[
d( R_n^\delta,\tilde R_n^0)  \leq \Vert  R_n^\delta -\tilde R_n^0\Vert_{\infty} = \max_{1\le t\le m_n}\Big|\sum_{i=1}^{t} a_{nm_n}^{-1}(X_i\mathds{1}_{\{a_{nm_n}^{-1}X_i \leq \delta \}} -c) \Big|\,.
\]
We deduce 
\begin{align*}
&n^{k+1} \bbP\Big(d(R_n^\delta,\tilde{R}_n^0)>\varepsilon,d(\tilde{R}_n^0,\mathbb{D}_k)>r\Big)\\
 & \le  n^{k+1}  \bbP\Big(\max_{1\le t\le m_n}\Big|\sum_{i=1}^{t} a_{nm_n}^{-1}(X_i\mathds{1}_{\{a_{nm_n}^{-1}X_i \leq \delta \}} -c)\Big|>\varepsilon, \Delta_{k+1}( R_n^0)\vee (a_{nm_n}^{-1}c)>2r  \Big)\\
  &=  n^{k+1}\bbP\Big(\max_{1\le t\le m_n}\Big| \sum_{i=1}^{t}a_{nm_n}^{-1}(X_i\mathds{1}_{\{a_{nm_n}^{-1}X_i \leq \delta \}}-c) \Big|>\varepsilon  ,\Delta_{k+1}(R_n^0)>2r  \Big)\,.
\end{align*}
where the last equality holds for $n$ large enough so that $a_{nm_n}^{-1}c\leq 2r$.
Hence, Equation~\eqref{tail_cond2} is equivalent to
\begin{equation}\label{tail_cond3}
\lim_{\delta\to 0}\limsup_{n\to \infty}n^{k+1}p_n^\delta = 0
\end{equation}
with 
$$p_n^\delta=\bbP\Big(\max_{1\le t\le m_n}\Big| \sum_{i=1}^{t}a_{nm_n}^{-1}(X_i\mathds{1}_{\{a_{nm_n}^{-1}X_i \leq \delta \}}-c) \Big|>\varepsilon   ,\Delta_{k+1}(R_n^0)>2r  \Big).
$$
Note that  $\Delta_{k+1}(R_n^0)=a_{nm_n}^{-1}X_{m_n-k:m_n}$.
Introduce independent random variables $U_1,\ldots,U_{m_n}$ with uniform distribution on $[0,1]$ and their order statistics $U_{1:m_n} \leq\ldots\leq U_{m_n:m_n}$ and denote $F^{\leftarrow}$ the quantile function of $X_1$. By the inversion method, $(X_1,\ldots,X_{m_n})$ has the same distribution as $(F^\leftarrow(U_1),\ldots,F^\leftarrow(U_{m_n}))$ . Possibly changing the underlying probability space, we assume without loss of generality that $X_i=F^\leftarrow(U_i)$ and similarly $X_{i:m_n}=F^\leftarrow(U_{i:m_n})$, $1\leq i\leq m_n$. The conditioning event $\Delta_{k+1}(R_n^0)>2r$ is then equal to $U_{m_n-k,m_n}>F(2ra_{nm_n})$ and we have 
\[
p_n^\delta=\bbP\Big(\max_{1\le t\le m_n}\Big| \sum_{i=1}^{t}a_{nm_n}^{-1}(F^\leftarrow(U_{i}) \mathds{1}_{\{F^\leftarrow(U_{i}) \leq a_{nm_n}\delta \}}-c)\Big|>\varepsilon, U_{m_n-k:m_n}> F(2 ra_{nm_n})  \Big)
\]
For $\delta\leq 2r$, $F^\leftarrow(U_{m_n-k:m_n})> 2ra_{nm_n}$ implies $F^\leftarrow(U_{m_n+1-i:m_n})> a_{nm_n}\delta $ for $i=1,\ldots,k+1$ and the terms corresponding to the $k+1$ largest order statistics are equal to $ca_{nm_n}^{-1}$. Since there are at most $k+1$ of them, their contribution is at most $ (k+1)ca_{nm_n}^{-1}\to 0$ and is smaller than $\varepsilon/2$ for large $n$ so that the contributions of the terms corresponding to the $m_n-k-1$ smallest order statistics must be larger than $\varepsilon/2$. Denote by $\sigma(i)$ the  rank of observation $i$, i.e.  $U_{\sigma(i):m_n}=U_i$, $1\leq i\leq m_n$. There is a unique permutation $\sigma'$ of $\{1,\ldots,m_n-k-1\}$ such that the $m_n-k-1$ smallest order statistics appear in the same order in the sequences $(U_i)_{1\leq i\leq m_n}=(U_{\sigma(i):m_n})_{1\leq i\leq m_n}$ and $(U_{\sigma'(i):m_{n}})_{1\leq i\leq m_n-k-1}$. We obtain that $p_n^\delta$ is bounded from above by
\[
 \bbP\Big(\max_{1\le t\le m_n-k-1}\Big| \sum_{i=1}^{t}a_{nm_n}^{-1}(F^\leftarrow(U_{\sigma'(i):m_n}) \mathds{1}_{\{F^\leftarrow(U_{\sigma'(i):m_n}) \leq a_{nm_n}\delta \}}-c)\Big| >\varepsilon/2 ,U_{m_n-k:m_n}>u_n  \Big)
\]
with $u_n=F(2r a_{nm_n})$.  Conditionally on $U_{m_n-k:m_n}=u$, the vector $(U_{i:m_n})_{1\leq i\leq m_n-k-1}$ has the same distribution as the order statistic on an independent uniform sample on $[0,u]$ with size $m_n-k-1$. By exchangeability of $(U_1,\ldots,U_{m_n})$, the distribution of $ \sigma$ is independent of the order statistics $(U_{i:m_n})_{1\leq i\leq m_n}$ and uniform on the set of permutations of $\{1,\ldots,m_n\}$. It follows that the  permutation $\sigma'$ over $\{1,\ldots,m_n-k-1\}$ is uniform and independent of $(U_{i:m_n})_{1\leq i\leq m_n-k-1}$ and hence that $(V_i)_{1\leq i\leq m_n-k-1}=(U_{\sigma'(i):m_n-k-1}/u)_{1\leq i\leq m_n-k-1}$ has independent components  uniform on $[0,1]$. We deduce
\begin{align*}
&\bbP\Big(\max_{1\le t\le m_n-k-1}\Big|\sum_{i=1}^{t} a_{nm_n}^{-1} (F^\leftarrow(U_{\sigma'(i):m_n}) \mathds{1}_{\{F^\leftarrow(U_{\sigma'(i):m_n}) \leq a_{nm_n}\delta \}}-c)\Big|>\varepsilon/2 \Big| U_{m_n-k:m_n}=u  \Big)\\
&=\bbP\Big(\max_{1\le t\le m_n-k-1}\Big|\sum_{i=1}^{t} a_{nm_n}^{-1}(F^\leftarrow(uV_i)\mathds{1}_{\{F^\leftarrow(uV_i) \leq a_{nm_n}\delta \}}-c) \Big|>\varepsilon/2\Big)\\
&\leq 3\max_{1\le t\le m_n-k}\bbP\Big(\Big|\sum_{i=1}^{t} a_{nm_n}^{-1}(F^\leftarrow(uV_i) \mathds{1}_{\{F^\leftarrow(uV_i) \leq a_{nm_n}\delta \}}-c)\Big|>\varepsilon/6\Big)
\end{align*}  
where the last line follows from  Etemadi's inequality \citep{E85}. Integrating with respect to $U_{m_n-k:m_n}>F(2r a_{nm_n})$, we obtain the upper bound
\[
p_n^\delta \leq 3\mathbb{E}\left[ \pi_n^\delta\big(U_{m_n-k:m_n}\big)\mathds{1}_{\{U_{m_n-k:m_n}>F(2r a_{nm_n}) \}} \right]
\]
with
\[
\pi_n^\delta(u)=\max_{1\le t\le m_n-k}\bbP\Big(\Big|\sum_{i=1}^{t} a_{nm_n}^{-1}(F^\leftarrow(uV_i) \mathds{1}_{\{F^\leftarrow(uV_i) \leq a_{nm_n}\delta \}}-c)\Big|>\varepsilon/6\Big).
\]
In the following, we provide upper bounds for $\pi_n^\delta(u)$ and prove that
\begin{equation}\label{tail_cond4}
\lim_{\delta\to 0}\limsup_{n\to \infty}n^{k+1}\mathbb{E}\left[ \pi_n^\delta\big(U_{m_n-k:m_n}\big)\mathds{1}_{\{U_{m_n-k:m_n}>F(2r a_{nm_n}) \}} \right] = 0
\end{equation}
which implies Equation~\eqref{tail_cond3}. We classically have to distinguish four different cases. In each case, we will use the following Lemma.
\begin{lemma}\label{lem:last}
Let $X$ be a non-negative regularly varying random variable with index $\alpha>0$ and survival function $\bar{F}=1-F$. Then, for any $p>\alpha$, we have
\begin{align*}
\bbE\Big[\Big(\dfrac{X}{x}\Big)^p\mathds{1}_{\{X\le x \}}\Big]\sim \dfrac{\alpha}{p-\alpha} \bar F(x)\,,\qquad x\to \infty\,.
\end{align*}

\begin{proof}
The proof follows from Karamata's theorem; see for instance \citep[Theorem 1.6.4]{BGT89}, Equation (1.6.3). 
\end{proof}
\end{lemma}

\medskip\noindent
\textbf{Proof of Equation~\eqref{tail_cond4} in the case $\alpha<1$.}\\
Using Markov inequality, we provide an upper bound for $\pi_n^\delta(u)$ as follows: for any $1\le t\le m_n-k$ and $\varepsilon >0$, we have
\begin{align*}
&\bbP\Big( \left| \sum_{i=1}^{t} a_{nm_n}^{-1}(F^\leftarrow(uV_i)\mathds{1}_{\{F^\leftarrow(uV_i) \leq a_{nm_n}\delta \}}-c)\right| >\varepsilon/6\Big)\\
&\leq 6{\varepsilon}^{-1} \sum_{i=1}^{t}\mathbb{E}\Big[  a_{nm_n}^{-1}F^\leftarrow(uV_i) \mathds{1}_{\{F^\leftarrow(uV_i) \leq a_{nm_n}\delta\}}\Big]\\
&\leq  6{\varepsilon}^{-1} m_n\mathbb{E}\Big[  a_{nm_n}^{-1}F^\leftarrow(uV_1) \mathds{1}_{\{F^\leftarrow(uV_1) \leq a_{nm_n}\delta\}}\Big]\,. \end{align*}  
We use the fact that there are at most $m_n$ summands as well as the fact that the $V_i$'s are iid. Using the change of variable $x=F^\leftarrow(uv)$, we have
$$
\mathbb{E}\Big[ F^\leftarrow(uV_1) \mathds{1}_{\{F^\leftarrow(uV_1) \leq a_{nm_n}\delta\}}\Big]=\int_0^{F(a_{nm_n}\delta)/u} F^\leftarrow(uv)dv =u^{-1}\int_0^{ a_{nm_n}\delta }  xF(dx)\,.
$$
We recognize the expression $u^{-1}\bbE[X_1\mathds{1}_{X_1\le a_{nm_n}\delta}]$ on which we apply Lemma \ref{lem:last} with $p=1>\alpha$. We deduce that
\begin{align*}
\pi_n^\delta(u)&\leq   \dfrac6{\varepsilon u} m_n \delta\,\mathbb{E}\Big[  \dfrac{X_1}{\delta a_{nm_n}} \mathds{1}_{\{X_1 \leq a_{nm_n}\delta\}}\Big]\sim \dfrac6{\varepsilon u} m_n \delta\dfrac{\alpha}{1-\alpha}\bar F(a_{nm_n} \delta)=o(1)\,,
\end{align*}  
as $m_n\bar F(a_{nm_n \delta}) \sim \delta^{-\alpha}/n\to 0$ as $n\to \infty$.\\

\noindent
\textbf{Proof of Equation~\eqref{tail_cond4} in the case   $\mathbb{E}[X_1]<\infty$.}\\
When the variable $X_1$ is integrable, there is a common previous step consisting in centering the partial sums in Equation~\eqref{tail_cond4}. Let us denote $c_{n,\delta}(u)=\mathbb{E}[F^\leftarrow(uV_1) \mathds{1}_{\{F^\leftarrow(uV_1) \leq a_{nm_n}\delta\}} ] $ the centering term. Using  $c=\mathbb{E}[X_1]$ and $u>F(2ra_{nm_n})$, we obtain
\begin{align*}
c_{n,\delta}(u)-c& = u^{-1}\int_0^{F(a_{nm_n}\delta)} F^\leftarrow(v)dv -\int_0^{1} F^\leftarrow(v)dv \\
&=u^{-1}\int_{F(a_{nm_n}\delta)}^1 F^\leftarrow(v)dv +(u^{-1}-1)\int_0^{1} F^\leftarrow(v)dv\,.
\end{align*}
Using that for $n$ large enough we have $u^{-1}< F(2ra_{nm_n})^{-1} \le 2$ so that
\[
|c_{n,\delta}(u)-c|\le 2( \bbE[X_1\mathds{1}_{X_1>a_{nm_n}\delta}] +\bar F(a_{nm_n}\delta) \bbE[X_0])=O(a_{nm_n}\bar F(a_{nm_n}))\,,\qquad n\to\infty\,,
\]
by an application of Karamata's theorem. Thus $m_na_{nm_n}^{-1}|c_{n,\delta}-c|=O(n^{-1})$ uniformly over $u>F(2ra_{nm_n})$ as $n\to \infty$ and  Equation~\eqref{tail_cond3} is implied by
\begin{equation}\label{tail_cond4t}
\lim_{\delta\to 0}\limsup_{n\to \infty}n^{k+1}\mathbb{E}\left[ \tilde \pi_n^\delta\big(U_{m_n-k:m_n}\big)\mathds{1}_{\{U_{m_n-k:m_n}>F(2r a_{nm_n}) \}} \right] = 0
\end{equation}
with
\[
\tilde \pi_n^\delta(u)=\max_{1\le t\le m_n-k}\bbP\Big(\Big|\sum_{i=1}^{t} a_{nm_n}^{-1}(F^\leftarrow(uV_i) \mathds{1}_{\{F^\leftarrow(uV_i) \leq a_{nm_n}\delta \}}-c_{n,\delta}(u))\Big|>\varepsilon/7\Big).
\]

\noindent
\textbf{Proof of Equation~\eqref{tail_cond4t} in the case $1\leq \alpha<2$ with $\mathbb{E}[X_1]<\infty$.}\\
Applying Markov inequality of order $p=2$, we obtain for any $1 \leq t\le m_n-k$
\begin{align*}
&\bbP\left(\left| \sum_{i=1}^{t} a_{nm_n}^{-1}(F^\leftarrow(uV_i)\mathds{1}_{\{F^\leftarrow(uV_i) \leq a_{nm_n}\delta\}}-c_{n,\delta}(u)) \right|  >\varepsilon/7\right)\\
&\leq 49 \varepsilon^{-2} a_{nm_n}^{-2}\mathbb{E}\left[\left(\sum_{i=1}^{t}  (F^\leftarrow(uV_i) \mathds{1}_{\{F^\leftarrow(uV_i) \leq a_{nm_n}\delta\}}-c_{n,\delta}(u))\right)^2\right]\\
&\leq 49 \varepsilon^{-2}a_{nm_n}^{-2}  m_n \mathrm{Var}\Big(F^\leftarrow(uV_1) \mathds{1}_{\{F^\leftarrow(uV_1) \leq a_{nm_n}\delta\}} \Big).
\end{align*}  
Using the change of variable $x=F^\leftarrow(uv)$ and Lemma~\ref{lem:last} with $p=2>\alpha$, we have
\begin{align*}
\mathrm{Var}\Big(F^\leftarrow(uV_1) \mathds{1}_{\{F^\leftarrow(uV_1) \leq a_{nm_n}\delta\}} \Big)&\leq \int_0^{F(a_{nm_n}\delta)/u} F^\leftarrow(uv)^2dv =u^{-1}\bbE[X_1^2\mathds{1}_{\{X_1\le a_{nm_n}\delta\}}]\\
&=O\big(a_{nm_n}^2\bar F(a_{nm_n} )\big).
\end{align*}
We finally get
\begin{align*}
\tilde\pi_n^\delta(u)=O\big(m_n \bar F(a_{nm_n} )\big)=O(n^{-1})=o(1),\quad \mbox{as $n\to\infty$}.
\end{align*}  

\noindent
\textbf{Proof of Equation~\eqref{tail_cond4t} in the case $ \alpha\geq 2$ with $\mathrm{Var}(X_1)<\infty$.}\\
From the Fuk-Nagaev inequality, see  \citep{petrov1995} Equation (2.79) page 78, for any $p>\alpha$, we have
\begin{align*}
&\bbP\Big(\sum_{i=1}^{t} a_{nm_n}^{-1}(F^\leftarrow(uV_i)\mathds{1}_{\{F^\leftarrow(uV_i) \leq a_{nm_n}\delta \}}) -c>\varepsilon/7 \Big)\\
&\leq \Big(\dfrac {7(p+2)}p\Big)^p\frac{m_n \bbE\Big[\Big|a_{nm_n}^{-1}(F^\leftarrow(uV_1)\mathds{1}_{\{F^\leftarrow(uV_1) \leq a_{nm_n}\delta \}}) \Big|^p\Big]}{\varepsilon^p}\\
& \quad+ \exp\Big(-\dfrac{2}{e^p\,(p+2)^2}\,
\dfrac {\varepsilon^2 }{49 m_n\text{Var}(a_{nm_n}^{-1}F^\leftarrow(uV_1) \mathds{1}_{\{F^\leftarrow(uV_1) \leq a_{nm_n}\delta\}} )}\Big) \\
&\leq c_p^1\,\varepsilon^{-p}m_na_{nm_n}^{-p} \bbE\Big[F^\leftarrow(uV_1)^p \mathds{1}_{\{F^\leftarrow(uV_1) \leq a_{nm_n}\delta\}} \Big]+ \exp\Big(-c_p^2
\dfrac{\varepsilon^2 a_{nm_n}^2}{m_n\text{Var}(F^\leftarrow(uV_1)  )}\Big)\\
&=  A_1 + A_2
\end{align*}  
for some constants $c_p^1,c_p^2>0$ depending only of the order $p$. Lemma \ref{lem:last} again implies for some $p>\alpha \geq 2$
\begin{align*}
A_1 \sim c_p^1\,\varepsilon^{-p} \delta^p m_n \frac{\alpha}{p- \alpha} \bar F(a_{nm_n \delta})=o(1)
\end{align*}
since $\delta^p m_n \bar F(a_{nm_n \delta}) \sim \delta^{p-\alpha}/n \to 0$ as $n\to \infty$ for $p>\alpha$. It remains to deal with the second term $A_2$. Since $\alpha \geq 2$, under the condition $a^2_{nm_n}/m_n \to \infty$ then $A_2 \to 0$, which ends the proof of the case $ \alpha\geq 2$ with $\mathrm{Var}(X_1)<\infty$. This proves Equations \eqref{tail_cond4t} and \eqref{tail_cond3} and concludes the proof of Theorem \ref{theo:RV-D2}.
\end{proof}

\subsection{Proofs related to Section~\ref{sec:app2}}

\begin{proof}[Proof of Proposition~\ref{prop:reinsurance}]
The proof involves three steps. We start by proving regular variation of $X_{N-k:N}$ and then we show that the tail distribution of $\sum_{i=1}^{N-k-1} X_{i:N}$ is negligeable regarding the tail of $X_{N-k:N}$. In Step 3, we finally show 
\begin{align}\label{EquivRV}
\mathbb{P}\left(  \sum_{i=1}^{N-k} X_{i:N} >x\right)   \underset{x \to \infty}{\sim} \mathbb{P}\left(  X_{N-k:N} >x\right)  .
\end{align}
Recall the notation: $F^{\leftarrow} $ denotes the quantile function and $ U_{1,n}\leq \cdots \leq U_{n,n}$ the order statistics of a iid sample $U_1,\ldots,U_n$  uniformly distributed on $[0,1]$.\newline

\textbf{Step 1}. For $x>0$, we have 
\begin{align*}
\mathbb{P}\left(  X_{N-k:N} >x\right)= \sum_{n=0}^{\infty} \mathbb{P}\left(  X_{n-k:n} >x\right) \mathbb{P}(N=n)=\sum_{n=0}^{\infty}\mathbb{P}\left( S_{n,x} >k\right) \mathbb{P}(N=n)
\end{align*}
where $S_{n,x} \sim \mathcal{B}(n, 1-F(x))$ is a Binomial random variable.
Since $\mathbb{P}(S_{n,x} \geq k+1)\underset{x \to \infty}{\sim} \mathbb{P}(S_{n,x} = k+1)$   and 
$$\frac{\mathbb{P}(S_{n,x} = k+1)}{(1-F(x))^{k+1}} \leq \dbinom{n}{k+1}= \frac{n^{[k+1]}}{(k+1)!}, \  \ \ x>0 $$
Lebesgue's theorem yields 
\begin{align*}
\sum_{n=0}^{\infty} \frac{\mathbb{P}\left(  S_{n,x} >x\right) }{(1-F(x))^{k+1}}\mathbb{P}(N=n) \underset{x \to \infty}{\longrightarrow}  \sum_{n=0}^{\infty} \frac{ n^{[k+1]} }{(k+1)!}\mathbb{P}(N=n)= \frac{1}{(k+1)!}\mathbb{E}[N^{[k+1]}].
\end{align*}
Then, 
\begin{align}\label{RVorder}
\lim_{x \to \infty} \frac{\mathbb{P}\left(  X_{N-k:N} >x\right)}{(1-F(x))^{k+1}}= \frac{1}{(k+1)!}\mathbb{E}[N^{[k+1]}]< \infty
\end{align}
under the assumption $\mathbb{E}[N^{k+1}]<\infty$. Since $X$ is regularly varying with index $\alpha$, Equation \eqref{RVorder} proves that $ X_{N-k:N}$ is regularly varying with index $(k+1) \alpha$.\newline

\textbf{Step 2}. For $x>0$, we have
\begin{align*}
\mathbb{P}\left( \sum_{i=1}^{N-{k-1}}  X_{i:N} >x \right) = \sum_{n=k}^{\infty} \mathbb{E} \left[  \mathbb{P}\left( \sum_{i=1}^{n-{k-1}}  F^{\leftarrow} (U_{i:n}) \geq  x \ | U_{n-k:n} \right) \right]\mathbb{P}(N=n) .
\end{align*}
By independence and Markov inequality to the order $p$, we obtain 

\begin{align*}
\mathbb{P}\left( \sum_{i=1}^{n-{k-1}}  F^{\leftarrow} (U_{i:n}) \geq  x \ | U_{n-k:n}=u \right) & \leq x^{-p} \mathbb{E} \left[\left( \sum_{i=1}^{n-{k-1}} F^{\leftarrow} (U_{i:n})\right)^p \ | \ U_{n-k:n}=u \right]\\
& = x^{-p} \mathbb{E} \left[ \left( \sum_{i=1}^{n-k-1} F^{\leftarrow} (uU_i) \right)^p  \right]\\
& \leq x^{-p} n^p \int_{0}^1 F^{\leftarrow} (uv)^p dv.
\end{align*}

We have for $u \geq 0$

$$\int_{0}^1 F^{\leftarrow} (uv)^p dv \leq M  F^{\leftarrow} (u)^p (1-u).$$

Integrating with respect to $u$, it follows that 
\begin{align*}
\mathbb{P}\left( \sum_{i=1}^{n-k-1} X_{i:n} >x\right) & \leq \mathbb{E} \left[  \mathbb{P}\left( \sum_{i=1}^{n-k-1}  F^{\leftarrow} (U_{i:n}) \geq  x \ | U_{n-k:n} \right) \right]\\
& \leq x^{-p} n^p M \mathbb{E}\left[ F^{\leftarrow} (U_{n-k,n})^p (1-U_{n-k:n}) \right]\\
& \leq x^{-p} n^p M  \mathbb{E} \left[ X_{n-k:n}^p (1-F(X_{n-k:n})) \right]
\end{align*}
as $F(F^{\leftarrow}(u))\leq u$ for all $u>0$. From Step 1, $X_{n-k:n}$ is regularly varying with index $(k+1) \alpha$. Since $X$ is regularly varying with index $ \alpha$, then $ \mathbb{E} \left[ X_{n-k:n}^p (1-F(X_{n-k:n})) \right]=c < \infty$ for $p <k\alpha$. Integrating now with respect to $n$, for $p>(k+1)\alpha$ we have

\begin{align*}
\mathbb{P}\left( \sum_{i=1}^{N-k-1}  X_{i:N} >x \right)&  \leq \sum_{n=k}^{\infty} x^{-p} n^p M  \mathbb{E} \left[ X_{n-k:n}^p (1-F(X_{n-k:n})) \right]\mathbb{P}(N=n) \\
& \leq x^{-p} M c\sum_{n=k+2}^{\infty}  n^p \mathbb{P}(N=n) \\
& =o((1-F(x))^{k+1})
\end{align*}
under the assumption $\mathbb{E}[N^p]<\infty$ for $p>(k+1) \alpha$. This proves $\mathbb{P}\left( \sum_{i=1}^{N-k-1} X_{i:n} >x\right)=o\left(\mathbb{P}\left(  X_{N-k:N} >x\right)\right)$ and concludes Step 2.\newline

\textbf{Step 3}. Obviously we have $\mathbb{P}\left(  \sum_{i=1}^{N-k} X_{i:N} >x\right)    \geq \mathbb{P}\left(  X_{N-k:N} >x\right) $
for $x>0$, it remains to prove that 

$$\limsup_{x \to \infty} \frac{\mathbb{P}\left(  \sum_{i=1}^{N-k} X_{i:N} >x\right) }{\mathbb{P}\left(  X_{N-k:N} >x\right)  } \leq 1.$$

For any $\varepsilon >0$ and $x\geq 0$ it holds 
\begin{align*}
& \mathbb{P}\left(  \sum_{i=1}^{N-k} X_{i:N} >x\right)\\
& = \mathbb{P}\left(  \sum_{i=1}^{N-k} X_{i:N} >x, \sum_{i=1}^{N-k-1} X_{i:N} > \varepsilon x \right) +  \mathbb{P}\left(  \sum_{i=1}^{N-k} X_{i:N} >x, \sum_{i=1}^{N-k-1} X_{i:N} \leq \varepsilon x \right)\\
& \leq \mathbb{P}\left(  \sum_{i=1}^{N-k-1} X_{i:N} > \varepsilon x\right) + \mathbb{P}\left(  X_{N-k:N} >(1 - \varepsilon)x\right).
\end{align*}

On the one hand, we have 
\begin{align}\label{RVeq1}
\frac{\mathbb{P}\left(  \sum_{i=1}^{N-k-1} X_{i:N} > \varepsilon x\right)}{\mathbb{P}\left(  X_{N-k:N} > x\right)} & =
\frac{\mathbb{P}\left(  \sum_{i=1}^{N-k-1} X_{i:N} >\varepsilon x\right)}{\mathbb{P}\left(  X_{N-k:N} >\varepsilon x\right)} \cdot \frac{\mathbb{P}\left(  X_{N-k:N} >\varepsilon x\right)}{\mathbb{P}\left(  X_{N-k:N} > x\right)} \underset{x \to \infty}{\longrightarrow} 0
\end{align}
since $\mathbb{P}\left( \sum_{i=1}^{N-k-1} X_{i:n} >x\right)=o\left(\mathbb{P}\left(  X_{N-k:N} >x\right)\right)$ from Step 2 and since $ X_{N-k:N}$ is regularly varying with index $(k+1) \alpha$ from Step 1. 
On the other hand , regular variation of $ X_{N-k:N}$  also yields
\begin{align}\label{RVeq2}
 \frac{\mathbb{P}\left(  X_{N-k:N} > (1+ \varepsilon) x \right)}{\mathbb{P}\left(  X_{N-k:N} > \varepsilon x\right)} \longrightarrow (1 +\varepsilon)^{-(k+1) \alpha}.
\end{align}
Combining Equations \eqref{RVeq1} and \eqref{RVeq2}, we obtain $$\limsup_{x \to \infty} \frac{\mathbb{P}\left(  \sum_{i=1}^{N-k} X_{i:N} >x\right) }{\mathbb{P}\left(  X_{N-k:N} >x\right)  } \leq (1 + \varepsilon)^{-(k+1) \alpha}   \underset{\varepsilon \to 0}{\longrightarrow} 1.$$
This proves Equation \eqref{EquivRV} and concludes the proof of Proposition \ref{prop:reinsurance}.
\end{proof}

\begin{proof}[Proof of Corollary~\ref{cor:reinsurance}]
Proposition~\ref{prop:reinsurance} together with Equation~\ref{eq:dist_expl} imply that, for all $\varepsilon>0$,
\begin{align*}
n^{k+1}\mathbb{P}(d(a_n^{-1}R,\mathbb{J}_k)>\varepsilon)&= n^{k+1}\mathbb{P}(a_n^{-1}\sum_{i=1}^{N-k+1}X_{i:N}>\varepsilon )\\
&\sim  n^{k+1}\mathbb{P}(a_n^{-1}X_{N+1-k:N}>\varepsilon )=n^{k+1}\mathbb{P}(d(a_n^{-1}R,\mathbb{D}_k)>\varepsilon/2)\\
&\to \frac{1}{k!}\mathbb{E}[N^{[k]}]\varepsilon^{-\alpha k}.
\end{align*}
Since $\{d(a_n^{-1}R,\mathbb{D}_k)>\varepsilon/2 \}\subset \{d(a_n^{-1}R,\mathbb{J}_k)>\varepsilon\}$, we deduce
\begin{align}
&n^{k+1}\mathbb{P}\left(d(a_n^{-1}R,\mathbb{J}_k)>\varepsilon,d(a_n^{-1}R,\mathbb{D}_k)\leq \varepsilon/2 \right)\nonumber\\
&= n^{k+1}\Big(\mathbb{P}\left(d(a_n^{-1}R,\mathbb{J}_k)>\varepsilon\right)-\mathbb{P}\left(d(a_n^{-1}R,\mathbb{D}_k)> \varepsilon/2 \right)\Big)\to 0,\quad \mbox{as $n\to\infty$}.\label{eq:negligeable}
\end{align}
We need to prove that, as $n\to\infty$, 
\begin{equation}\label{eq:toprove}
n^{k+1}\mathbb{E}\left[ f(a_n^{-1}R)\right]\longrightarrow \int_{\mathbb{D}}f(r)\mu^{\#}_{k+1}(\rmd r)
\end{equation}
for all continuous bounded functions $f:\mathbb{D}\to\mathbb{R}$ with support bounded away from $\mathbb{J}_k$.
Let $\varepsilon>0$ such that $f$ vanishes on an $\varepsilon$-neighborhood of $\mathbb{J}_k$. We have 
\begin{align*}
&n^{k+1}\mathbb{E}\left[ f(a_n^{-1}R)\right]\\
&=n^{k+1}\mathbb{E}\left[ f(a_n^{-1}R)\mathds{1}_{\{ d(a_n^{-1}R,\mathbb{D}_k)>\varepsilon/2\}}\right]+n^{k+1}\mathbb{E}\left[ f(a_n^{-1}R)\mathds{1}_{\{d(a_n^{-1}R,\mathbb{J}_k)>\varepsilon, d(a_n^{-1}R,\mathbb{D}_k)\leq \varepsilon/2\}}\right]\\
&= n^{k+1}\mathbb{E}\left[ f(a_n^{-1}R)\mathds{1}_{\{ d(a_n^{-1}R,\mathbb{D}_k)>\varepsilon/2\}}\right]+o(1).
\end{align*}
The $o(1)$ term is justified by Equation~\eqref{eq:negligeable} and the boundedness of $f$. Theorem~\ref{theo:RV-D} Equation~\eqref{eq:thm:RV-D} provides the asymptotic
\[
n^{k+1}\mathbb{E}\left[ f(a_n^{-1}R)\mathds{1}_{\{ d(a_n^{-1}R,\mathbb{D}_k)>\varepsilon/2\}}\right]\longrightarrow \int_{\mathbb{D}}f(r)\mathds{1}_{\{ d(r,\mathbb{D}_k)>\varepsilon/2\}}\mu^{\#}_{k+1}(\rmd r)
\]
because the function $r\in\mathbb{D}\mapsto f(r)\mathds{1}_{\{ d(r,\mathbb{D}_k)>\varepsilon/2\}}$ is bounded and continuous $\mu^{\#}_{k+1}(\rmd r)$-a.e. with support bounded away from $\mathbb{D}_k$ . Note that, by Lemma~\ref{lem:dist-Dk}, the discontinuity set is included in $\{r\in\mathbb{D}:\Delta_{k+1}(r)=\varepsilon/2\}$ and that the Lebesgue measure involved in the definition of $\mu^{\#}_{k+1}$ has no atom. Finally, we have
\[
\int_{\mathbb{D}}f(r)\mathds{1}_{\{ d(r,\mathbb{D}_k)>\varepsilon/2\}}\mu^{\#}_{k+1}(\rmd r)=\int_{\mathbb{D}}f(r)\mu^{\#}_{k+1}(\rmd r)
\]
because $f(r)\mathds{1}_{\{ d(r,\mathbb{D}_k)\leq \varepsilon/2\}}=0$ $\mu_{k+1}(\rmd r)$-a.e. This concludes the proof of Equation~\eqref{eq:toprove} and Corollary~\ref{cor:reinsurance}.
\end{proof}

\begin{proof}[Proof of Proposition~\ref{prop:reinsurance2}]
The regular variation result stated in point $i)$ for $R_k^-=\sum_{i=1}^{N-k}X_{i:N}$ follows from  Proposition~\ref{prop:reinsurance}. 

The conditional limit theorem stated in point $ii)$ is a consequence of Corollary~\ref{cor:reinsurance} together with Proposition~\ref{prop:cond_limit_thm}. Indeed, since $R_k^-=d(R,\mathbb{J}_k)$, we have
\[
\mathbb{P}(x^{-1}R\in\cdot\mid R_k^{-}>x)=\mathbb{P}(x^{-1}R\in\cdot\mid x^{-1}R\in A)
\]
with $A=\{r\in\mathbb{D}:d(r,\mathbb{J}_k)>1\}$. Then, Corollary~\ref{cor:reinsurance} together with Proposition~\ref{prop:cond_limit_thm} imply
\[
\mathbb{P}(x^{-1}R\in\cdot\mid x^{-1}R\in A)\stackrel{d}\longrightarrow \frac{\mu^{\#}_{k+1}(A\cap \cdot)}{\mu^{\#}_{k+1}(A)}
\]
because $A$ is bounded away from $\mathbb{J}_k$ and such that $\mu^{\#}_{k+1}(A)=M_{k+1}([0,\infty)^{k+1})>0$ and $ \mu^{\#}_{k+1}( \partial A)=\mu^{\#}_{k+1}(\{r\in\mathbb{D}:d(r,\mathbb{J}_k)=1\})=0$. Using the definition~\eqref{eq:thm:RV-D-mu} of $\mu^{\#}_{k+1}$, the expression of the limiting distribution $\mu^{\#}_{k+1}(A\cap \cdot)/\mu^{\#}_{k+1}(A)$ is easily deduced as in statement $ii)$.

Proof of statement $iii)$. The heuristic is that if the residual risk is large at $t_0$, then there are at least $k+1$ large claims between $0$ and $t_0$. If the residual risk is even larger at $t_1$, then some large claim must have occurred between $t_0$ and $t_1$, leading to a total of at least $k+2$ large claims between $0$ and $t_1$. This is formalized as follows.
Noting that $R_k^-(t_0)=d(R_k(\cdot\wedge t_0),\mathbb{J}_k)$ and similarly  $R_k^-(t_1)=d(R_k(\cdot\wedge t_1),\mathbb{J}_k)$, we consider $A_\varepsilon=\{r\in\mathbb{D}:d(r(\cdot\wedge t_0),\mathbb{J}_k)\in (1,1+\varepsilon)\}$ and $B_u=\{r\in\mathbb{D}:d(r(\cdot\wedge t_1),\mathbb{J}_k)>u\}$. With these notations, for $\varepsilon>0$ and $u>0$, we have
\begin{align*}
\mathbb{P}(R_k^-(t_1)>ux\mid x<R_k^-(t_0)<(1+\varepsilon)x)
&=\frac{\mathbb{P}(x^{-1}R\in A_\varepsilon\cap B_u)}{\mathbb{P}(x^{-1}R\in A_\varepsilon)}.
\end{align*}
Because $A_\varepsilon$ and $A_\varepsilon\cap B_u$ are bounded away from $\mathbb{J}_k$ and  $\mathbb{J}_{k+1}$ respectively and because the continuity of $M_{k+1}$ and $M_{k+2}$ ensures that $A$ and $B$ are continuity sets for $\mu^{\#}_{k+1}$ and $\mu^{\#}_{k+2}$ respectively, Corollary~\ref{cor:reinsurance} implies
\[
\mathbb{P}(x^{-1}R\in A_\varepsilon)\sim \mu^{\#}_{k+1}(A_\varepsilon)(1- F(x))^{k+1}
\]
and
\[
\mathbb{P}(x^{-1}R\in A_\varepsilon\cap B_u)\sim \mu^{\#}_{k+2}(A_\varepsilon\cap B_u)(1- F(x))^{k+2}.
\]
As a consequence, the quotient satisfies
\begin{equation}\label{eq:proof1}
\lim_{x\to\infty}\frac{ \mathbb{P}(R_k^-(t_1)>ux\mid x<R_k^-(t_0)<(1+\varepsilon)x)}{1-F(x)}= \frac{\mu^{\#}_{k+2}(A_\varepsilon\cap B_u)}{\mu^{\#}_{k+1}(A_\varepsilon)}.
\end{equation}
The quantity $\mu^{\#}_{k+1}(A_\varepsilon)$ is computed considering $k+1$ claims $((s_i,z_i))_{1\leq i\leq k+1}$ with occurrence times $s_i\leq t_0$ and magnitudes $z_i$ satisfying  $1<\min_{1\leq i\leq k+1} z_i<1+\varepsilon$, yielding the integral form
\begin{align*}
\mu^{\#}_{k+1}(A_\varepsilon)
&=\frac{1}{(k+1)!}\int_{E^{k+1}} \mathds{1}_{\{1<\min_{1\leq i\leq k+1} z_i<1+\varepsilon, s_{i}\leq t_0, 1\leq i\leq k+1\}}M_{k+1}(\rmd s) \mu^{\otimes (k+1)}(\rmd z)\\
&=\frac{1}{(k+1)!}M_{k+1}\left([0,t_0]^{k+1}\right)\int_{(1,\infty)^{k+1}} \mathds{1}_{\{1<\min_{1\leq i\leq k+1} z_i<1+\varepsilon\}} \mu^{\otimes (k+1)}(\rmd z).
\end{align*}
In the integral, the domain can be restricted to $(1,\infty)^{k+1}$ because the minimum is larger than $1$ in the indicator function. Since $\mathds{1}_{z>1}\mu(\rmd z)$ is the $\alpha$-Pareto distribution, the integral is equal to 
\[
\bbP\left(\min_{1\leq i\leq k+1}Z_i<1+\varepsilon\right)=1- (1+\varepsilon)^{-(k+1)\alpha}\sim (k+1)\alpha\varepsilon\quad \mbox{as $\varepsilon\to 0$},
\]
where $(Z_i)_{1\leq i\leq k+1}$ are independent random variables with standard $\alpha$-Pareto distribution. We deduce 
\begin{equation}\label{eq:proof2}
\mu^{\#}_{k+1}(A_\varepsilon)\sim \frac{1}{k!}M_{k+1}\left([0,t_0]^{k+1}\right)\alpha\varepsilon.
\end{equation}
The quantity $\mu^{\#}_{k+2}(A_\varepsilon\cap B_u)$ is computed considering $k+2$ claims $((s_i,z_i))_{1\leq i\leq k+2}$ with $k+1$ of those  occurring before $t_0$ and one between $t_0$ and $t_1$. The $s_i$ are not ordered and there are $k+2$  cases according to which event occurs last. By symmetry, we consider only the case when the last event corresponds to the index $i=k+2$. The magnitudes must then satisfy $1<\min_{i\leq i\leq k+1} z_i<1+\varepsilon$ so that $A_\varepsilon$ is satisfied and $ z_{1:k+2}+z_{2:k+2}>u$ so that $B_u$ is satisfied. This yields, 
\begin{equation}\label{eq:proof3}
\mu^{\#}_{k+2}(A_\varepsilon\cap B_u)=\frac{(k+2)}{(k+2)!}M_{k+2}\left([0,t_0]^{k+1}\times (t_0,t_1] \right)I(\varepsilon)
\end{equation}
with
\[
I(\varepsilon)= \int_{(0,\infty)^{k+2}}\mathds{1}_{\{ 1<\min_{i\leq i\leq k+1} z_i<1+\varepsilon, z_{1:k+2}+z_{2:k+2}>u \}}\mu^{\otimes (k+2)}(\rmd z).
\]
We claim that, as $\varepsilon\to 0$,
\begin{equation}\label{eq:proof4}
I(\varepsilon)\sim \left( (u-1)^{-(k+1)\alpha}+\left((u-1)^{-\alpha}-1\right)_+\right)(k+1)\alpha\varepsilon.
\end{equation}
Then, Equation~\eqref{eq:item3} follows from Equations (\ref{eq:proof1}-\ref{eq:proof4}). 

It only remains to prove  Equation~\eqref{eq:proof4}. 
The  integral $I(\varepsilon)$ can be decomposed into two integrals $I_1(\varepsilon)+I_2(\varepsilon)$, corresponding to $z_{k+2}>1$ and  $z_{k+2}<1$ respectively. In the  case $z_{k+2}>1$,  the $z_i$'s are  larger than $1$ and 
\[
I_1(\varepsilon)=\mathbb{P}\left(1<\min_{i\leq i\leq k+1} Z_i<1+\varepsilon, Z_{1:k+2} +Z_{2:k+2}>u\right).
\]
We decompose $I_1(\varepsilon)$ as a sum of two terms according to $Z_{k+2}< \min_{i\leq i\leq k+1} Z_i$ or $Z_{k+2}> \min_{i\leq i\leq k+1} Z_i$. The first contribution is upper bounded by
\[
\mathbb{P}\left(1<Z_{k+2}< \min_{i\leq i\leq k+1} Z_i<1+\varepsilon\right)=o(\varepsilon^2)
\]
and is negligible. The second contribution corresponds to 
\begin{align*}
&\mathbb{P}\left(1<\min_{i\leq i\leq k+1} Z_i<1+\varepsilon, Z_{1:k+2} +Z_{2:k+2}>u,Z_{k+2}> \min_{i\leq i\leq k+1} Z_i\right)\\
&=\mathbb{P}\left(1<Z_{1:k+2}<1+\varepsilon, Z_{1:k+2} +Z_{2:k+2}>u,R_{k+2}\neq 1\right)
\end{align*}
where $R_{k+2}$ denotes the rank of $Z_{k+2}$ among $Z_1,\ldots, Z_{k+2}$ and $(Z_{i:k+2})_{1\leq i\leq k+2}$ the order statistics. Since $R_{k+2}$ is uniform on $\{1,\ldots,k+2\}$ and independent of the order statistics, we obtain
\begin{align*}
I_1(\varepsilon)&=\bbP(1<Z_{1:k+2}<1+\varepsilon,Z_{1:k+2}+Z_{2:k+2}>u,R_{n+2}\neq 1)+o(\varepsilon^2)\\
&=\frac{k+1}{k+2}\bbP(1<Z_{1:k+2}<1+\varepsilon,Z_{1:k+2}+Z_{2:k+2}>u)+o(\varepsilon^2).
\end{align*}
The joint density of $(Z_{1:k+2},Z_{2:k+2})$ is given by (see for instance \cite{ahsanullah2013introduction}, Chapter 2)  
\[
(k+1)(k+2)\alpha^2z_1^{-\alpha-1}z_2^{-(k+1)\alpha-1}\mathds{1}_{\{1<z_1<z_2\}},
\]
so that, for $u>1$,
\begin{align*}
&\bbP(1<Z_{1:k+2}<1+\varepsilon,Z_{1:k+2}+Z_{2:k+2}>u)\\
&= \int_{z_1=1}^{1+\varepsilon} \int_{z_2=(u-z_1)\vee z_1}^\infty (k+1)(k+2)\alpha^2z_1^{-\alpha-1}z_2^{-(k+1)\alpha-1}\rmd z_1\rmd z_2\\
&=\int_{z_1=1}^{1+\varepsilon}(k+2)\alpha z_1^{-\alpha-1}((u-z_1)\vee z_1)^{-(k+1)\alpha}\rmd z_1\\
&=(k+2)\alpha(u-1)^{-(k+1)\alpha}\varepsilon+o(\varepsilon).
\end{align*}
Gathering the previous estimates, we obtain
\[
I_1(\varepsilon)\sim (u-1)^{-(k+1)\alpha}(k+1)\alpha\varepsilon.
\]
We now consider the contribution $I_2(\varepsilon)$ that corresponds to the case $z_{k+2}<1$. Then, $z_{1:k+2}=z_{k+2}$ and $z_{2:k+2}=\min_{1\leq i\leq k+1}z_i$ so that
\begin{align*}
I_2(\varepsilon)
&= \int_{(1,\infty)^{k+1}}\int_{z_{k+2}=0}^1\mathds{1}_{\{ 1<\min_{i\leq i\leq k+1} z_i<1+\varepsilon, z_{k+2}>u-\min_{i\leq i\leq k+1}z_i \}}\mu^{\otimes (k+2)}(\rmd z)\\
&=\mathbb{E}\left[\left((u-\min_{1\leq i\leq k+1} Z_i)^{-\alpha}-1\right)_+\mathds{1}_{\{1<\min_{i\leq i\leq k+1} Z_i<1+\varepsilon\}} \right]\\
&= \left((u-1)^{-\alpha}-1\right)_+(k+1)\alpha\varepsilon+o(\varepsilon).
\end{align*}
Note that the main term vanishes when $u\geq 2$ which reflect the fact the residual loss cannot double if $z_{k+2}< \min_{1\leq i\leq k+1} z_i$. The asymptotic results for  $I_1(\varepsilon)$ and $I_2(\varepsilon)$ imply  Equation~\eqref{eq:proof4} and the proof is complete.\end{proof}

\section{Proofs related to Section~\ref{sec:background}}\label{sec:proof1}
\subsection{Proof of Propositon~\ref{prop:HRV-criterion}}

\begin{proof}[Proof of Propositon~\ref{prop:HRV-criterion}] We proceed as in the proof of Theorem 4.2 in~\cite{B68}. Consider $r>0$ and a $\mu$-continuity set $B\in \cB(E\setminus F^r)$. Possibly replacing $r$ by a smaller value, we can assume that $E\setminus F^r$ is a $\mu$-continuity set (Theorem 2.2 (i) in \cite{LR06} ensures that all but countably many $r>0$ have this property). We have
$$
n^{k}\bbP(a_n^{-1}X \in B)\le n^k\bbP(a_n^{-1}X_{n,m} \in B^\varepsilon) + n^k\bbP(d(a_n^{-1}X_{n,m},a_n^{-1}X)>\varepsilon,a_n^{-1}X \in B)\,.
$$
For $\varepsilon<r/2$,  $B^\varepsilon \in \cB(E\setminus F^{r/2})$ and Assumption $i)$ implies
\[
\limsup_{n\to \infty}n^k\bbP(a_n^{-1}X_{n,m} \in B^\varepsilon)\leq  \mu_m( \mathrm{cl}B^\varepsilon),
\]
so that
$$
\limsup_{n\to \infty}\bbP(a_n^{-1}X \in B)\le \mu_m( \mathrm{cl} B^\varepsilon) +\limsup_{n\to \infty} n^k\bbP(d(a_n^{-1}X_{n,m},a_n^{-1}X)>\varepsilon,a_n^{-1}X \in B)\,.
$$
Letting $m\to\infty$ in the right-hand side and using assumptions ii) and iii), we deduce, 
$$
\limsup_{n\to \infty}n^k\bbP(a_n^{-1}X \in B)\le \mu(\mathrm{cl} B^\varepsilon)\,.
$$
Since $B$ is a $\mu$-continuity set, letting $\varepsilon\downarrow 0$, we obtain by  monotone convergence
$$
\limsup_{n\to \infty}n^k\bbP(a_n^{-1}X \in B)\le \mu(\mathrm{cl} B)=\mu( B )\,.
$$
In order to obtain a lower bound, we use completion and write
$$
n^{k}\bbP(a_n^{-1}X \in B)= n^{k}\bbP(a_n^{-1}X \in E\setminus F^r) - n^{k}\bbP(a_n^{-1}X \in E\setminus F^r\cap B^c)\,.
$$
Thanks to the previous bound on the limsup, 
\begin{align*}
\liminf_{n\to \infty}n^{k}\bbP(a_n^{-1}X \in B)&\ge  \liminf_{n\to \infty} n^{k}\bbP(a_n^{-1}X \in E\setminus F^r) - \lim \sup_{n\to \infty}n^{k}\bbP(a_n^{-1}X \in E\setminus F^r\cap B^c)\\
&\ge \liminf_{n\to \infty} n^{k}\bbP(a_n^{-1}X \in E\setminus F^r) - \mu(E\setminus F^r\cap B^c).
\end{align*}
 Moreover, 
\begin{align*}
&\bbP(a_n^{-1}X_{n,m}\in E\setminus F^{r+\varepsilon})\\
&= \bbP(a_n^{-1}X \in  F^r, a_n^{-1}X_{n,m} \in E\setminus F^{r+\varepsilon})
+\bbP(a_n^{-1}X \in  E\setminus F^r, a_n^{-1}X_{n,m} \in E\setminus F^{r+\varepsilon})\\
  &\le \bbP(d(a_n^{-1}X_{n,m},a_n^{-1}X)>\varepsilon, a_n^{-1}X_{n,m} \in E\setminus F^{r+\varepsilon})+\bbP(a_n^{-1}X \in  E\setminus F^r),
\end{align*}
whence
$$
\bbP(a_n^{-1}X \in  E\setminus F^r)\ge \bbP(a_n^{-1}X_{n,m} \in E\setminus F^{r+\varepsilon}) - \bbP(d(a_n^{-1}X_{n,m},a_n^{-1}X)>\varepsilon, a_n^{-1}X_{n,m} \in E\setminus F^{r+\varepsilon})\,.
$$
Using assumptions i), ii) and iii) with the same reasoning as above, we get
$$
 \liminf_{n\to \infty} n^{k}P(a_n^{-1}X \in E\setminus F^r) \ge  \mu (\mathrm{int} E\setminus F^{r+\varepsilon})\,.
 $$
Letting $\varepsilon\downarrow 0$,  monotone convergence entails 
$$
 \liminf_{n\to \infty} n^{k}P(a_n^{-1}X \in E\setminus F^r) \ge  \mu (\mathrm{int} E\setminus F^r )=\mu ( E\setminus F^r )
 $$
 as $E\setminus F^r$ is a $\mu$-continuity set.  
 Finally, we have obtained
 $$
\mu (B )\le \liminf_{n\to \infty} n^{k}P(a_n^{-1}X \in B) \le   \lim \sup_{n\to \infty} n^{k}P(a_n^{-1}X \in B)= \mu (B )\,,
 $$
proving the desired convergence $ n^{k}P(a_n^{-1}X \in \cdot)\longrightarrow \mu(\cdot)$ in $\bbM(E\setminus F)$.
\end{proof}

\subsection{Proof of Theorem~\ref{theo:cv-N-Nk}}
The proof of Theorem~\ref{theo:cv-N-Nk} requires a good understanding of the distance to the cone $\cN_k$ in $\cN$. The  following lemma generalizes  inequalities (3.3) and (3.4) in \cite{DHS18}. It characterizes the distance of a point measure $\pi \in\mathcal{N}$ to the cone $\cN_k$. For  $\pi=\sum_{i\geq 1}\varepsilon_{(t_i,x_i)}$, we define  $\|\pi\|_{k+1}$ the $(k+1)$-th largest distance within $\{d_\cX(0,x_i),\ i\geq 1\}$, with the convention $\|\pi\|_{k+1}=0$ if $\pi$ has less than $k$ points.

\begin{lemma}\label{lem:dist_Nk} Let $\rho$ be the distance defined in \eqref{distance_rho} and let $\pi\in\mathcal{N}$. Then, for all $k\geq 0$, 
	\[
	\frac{1}{2}\left(\|\pi\|_{k+1}\wedge 1 \right)\leq \rho(\pi,\cN_k)\leq \|\pi\|_{k+1}.
	\]
\end{lemma}

\begin{proof}[Proof of Lemma~\ref{lem:dist_Nk}] For $\pi\in\mathcal{N}$, let $r_0=\|\pi\|_{k+1}$ and denote $\pi_0$ the restriction of $\pi$ to $\{(t,x)\in E: d_\cX(0,x)>r_0\}$. By definition of $r_0$, $\pi_0$ has at most $k$ points, that is $\pi\in\cN_k$. For all $r>r_0$, the restrictions $\pi^{r}$ and $\pi_0^{r}$ coincide so that $\rho_r(\pi^{r},\pi_0^{r})=0$ and 	
\begin{align*}
\rho(\pi,\pi_0)=\int_0^\infty \left\{\rho_r(\pi^{r},\pi_0^{r})\wedge 1\right\} e^{-r}\rmd r  \leq \int_0^{r_0} e^{-r}\rmd r\leq r_0=\|\pi\|_{k+1}.
\end{align*}
Since $\pi_0\in\cN_k$, $\rho(\pi,\cN_k)\leq \|\pi\|_{k+1}$, proving the right-hand side of the inequality. 

On the other hand, for $r<r_0$, the restriction $\pi^{r}$ has at least $k+1$ points. If $\psi\in\cN_k$, that is $\psi$ is a point measure on $E\setminus F$ with  at most $k$ points, then the restrictions $\pi^{r}$ and $\psi^{(r)}$ do not have the same number of points. Then, a straightforward application of  Lemma B.1 in \citealp{DHS18} leads to 
\[
\rho_r(\pi^{r},\psi^{r})\geq |\pi(E\setminus F^r)-\psi(E\setminus F^r)| \geq 1.
\]
We deduce
\[
\rho(\pi,\psi)\geq \int_0^{r_0} e^{-r}\rmd r=1-e^{-\|\pi\|_{k+1}}\geq \frac{1}{2}\left( \|\pi\|_{k+1}\wedge 1\right)
\]
and, $\psi\in\cN_k$ being arbitrary,
\[
\rho(\pi,\cN_k)\geq \frac{1}{2}\left( \|\pi\|_{k+1}\wedge 1\right).
\]
This proves the left-hand side of the inequality and concludes the proof.
\end{proof}

\begin{proof}[Proof of Theorem~\ref{theo:cv-N-Nk}]
The proof of Theorem~\ref{theo:cv-N-Nk} is very similar to the proof of Theorem A.1 in \cite{DHS18}, except that Equations (3.3) and (3.4) are replaced  here by Lemma~\ref{lem:dist_Nk}. This implies, in the same way as Equation (A1) in  \citet[Theorem A.1]{DHS18}, 
\begin{equation}\label{eq:inclusion}
(\cN \setminus\cN_k^r) \subset  \{\pi\in\cN: \|\pi\|_{k+1}> r\} \subset (\cN\setminus\cN_k^{r/2}),\quad r\leq 1.
\end{equation} 
We note also that
\[
\{\pi\in\cN: \|\pi\|_{k+1}> r\}=\{\pi \in\cN:\pi(E\setminus F^r)\geq k+1\}.
\]
According to the Portmanteau Theorem for $\bbM$-convergence \citep[Theorem 2.1]{LRR14}, the convergence $\mu^*_n\to\mu^*$  in $\bbM(\cN\setminus\cN_k)$ is equivalent to the convergence of the restrictions 
$\mu_n^{* r_i}\to\mu^{* r_i}$  in $\bbM_b(\cN\setminus\cN_k^{r_i})$ for each $i$,  for some sequence $r_i\downarrow 0$. Using Equation~\eqref{eq:inclusion}, it is more convenient to consider restrictions $\tilde\mu_n^{* r_i}, \tilde \mu^{* r_i}$ to the subsets $\{\pi(E\setminus F^{r_i})\geq k+1\}$ rather than restrictions $\mu_n^{* r_i}, \mu^{* r_i}$ to the subsets $\cN\setminus\cN_k^{r_i}$. The convergence  $\mu^*_n\to\mu^*$  in $\bbM(\cN\setminus\cN_k)$ is hence equivalent to the weak convergence of   $\tilde\mu_n^{* r_i}\to\mu^{* r_i}$ for each $i$, for some sequence $r_i\downarrow 0$. We then appeal to the characterization of weak convergence in terms of finite-dimensional distributions or Laplace functional \citep[Theorem 3.10 and Corollary 3.11]{Zhao2016}.
 
In the remainder of the proof, we simply point out some differences with respect to the proof of Theorem A.1 in \cite{DHS18}. \\
\textit{Proof of $ \ref{item:convmzero} \Rightarrow \ref{item:convfidi}$:} For $p \geq 1$, let $A_1,\ldots,A_p \in \mathcal{B}_\mu^*$ and $m_1, \ldots, m_p \in \mathbb{N}$ such that $\sum_{i=1}^p m_i \geq k+1$. Define the event $\mathcal{A}:=\{ \pi(A_i)=m_i, 1 \leq i\leq p\}$.  For $r$ small enough, $A_1,\ldots,A_p \in E\setminus F^r$ and we have  
\begin{align*}
\mathcal{A} \subset  \{ \pi : \pi( E\setminus F^r ) \geq k+1 \}   \subset \cN\setminus\cN_k^r,
\end{align*}
whence $\mu_n^{*r}(\mathcal{A})=\mu^*_n(\mathcal{A})$ and $\mu^{*r}(\mathcal{A})=\mu^*(\mathcal{A})$. The convergence $\mu^*_n\to\mu^*$ in $\bbM(\cN\setminus\cN_k)$ implies the weak convergence $\mu_n^{*r}\to\mu^{*r}$ for small $r$, which implies according to \cite{Zhao2016} the finite-dimensional convergence  
\begin{align*}
\mu^*_n(\mathcal{A})=\mu_n^{*r}(\mathcal{A}) \longrightarrow \mu^{*r}(\mathcal{A})=\mu^*(\mathcal{A}).
\end{align*}
This proves $\ref{item:convfidi}$.\\
\textit{Proof of $ \ref{item:convfidi} \Rightarrow \ref{item:convlaplace}$:}  By Lemma \ref{lem:dist_Nk}, \ref{item:convfidi}  implies the finite-dimensional convergence  $\tilde{\mu}_n^{*r}\overset{fidi}{\longrightarrow} \tilde \mu^{*r}$ for all $r >0$ such that $ \{\pi\in\cN :\pi (E\setminus F^r)\geq  k+1\}\in\mathcal{B}_{\mu{^*}}$. Let $r_i \to 0$ be a sequence of such continuity points. From \cite{Zhao2016}, convergence of the finite-dimensional distribution is equivalent to convergence of the Laplace functional 
$$\int_{\mathcal{N}}\rme^{-\pi(f)} \tilde{\mu}_n^{*r_i}(\rmd \pi) 
		\longrightarrow    \int_{\mathcal{N}}\rme^{-\pi(f)} \tilde{\mu}^{* r_i}(\rmd \pi).$$ 
This proves $\ref{item:convlaplace} $.\\
\textit{Proof of $\ref{item:convlaplace} \Rightarrow \ref{item:convmzero}$:} Similar to \cite{DHS18} where the approximating functions can be taken bounded Lipschitz.
\end{proof}

\section*{Acknowledgements}
The authors gratefully acknowledge the two referees and the associate editor for their constructive remarks and suggestions that lead to significant improvement of the paper.
The research of Cl\'ement Dombry is partially supported by the Bourgogne Franche-Comt\'e region (grant OPE-2017-0068).

\bibliographystyle{apalike}
\bibliography{biblio}

\end{document}